\documentclass[11pt,reqno]{amsart}
\usepackage[pdftex]{graphicx} \usepackage[dvipsnames]{xcolor}

\usepackage{a4wide}

\usepackage{amsfonts}
\usepackage{amsmath}
\usepackage{amssymb}
\usepackage{amsthm}
\usepackage{algorithm2e}
\usepackage{mathtools}
\usepackage{tikz}
\usepackage{tikz-3dplot}
\usetikzlibrary{patterns}
\usepackage{appendix}

\usepackage{paralist}
\usepackage[colorlinks,cite color=blue,pagebackref=true,pdftex]{hyperref}

\usepackage[T1]{fontenc}

\usepackage{caption}
\usepackage{subcaption}

\newcommand\Defn[1]{\textbf{\color{black}#1}}
\newcommand\Def[1]{\Defn{#1}}

\newcommand\eps{\varepsilon}
\renewcommand\emptyset{\varnothing}
\newcommand\Z{\mathbb{Z}}

\newcommand\N{\mathbb{N}}
\newcommand\Q{\mathbb{Q}}
\newcommand\R{\mathbb{R}}

\newcommand\x{\mathbf{x}}
\newcommand\y{\mathbf{y}}

\newcommand\inner[1]{\langle {#1} \rangle}
\newcommand\defeq{\coloneqq}

\newcommand\0{\mathbf{0}}

\newcommand\Fan{\mathcal{N}}%

\newcommand\id{\mathrm{id}}%
\newcommand\TypeSpc{\mathcal{T}}%
\newcommand\TypeCone{\TypeSpc_+}%
\newcommand\cTypeCone{\overline{\TypeSpc}_+}%

\newcommand\InSpc{\mathcal{I}}%
\newcommand\InCone{\InSpc_+}%
\newcommand\cInCone{\overline{\InSpc}_+}%
\newcommand\Arr{\mathcal{A}}%
\newcommand\braid{\mathcal{A}}%
\newcommand\Sym{\mathfrak{S}}%

\newcommand\Hyp{\mathbb{H}}%
\newcommand\Del{\mathcal{D}}%
\newcommand\PP{\mathbb{P}}%
\newcommand\Walk{\mathcal{W}}%
\newcommand\Prof{\mathcal{B}}%
\newcommand\VInProf{\Prof^\mathrm{vin}}%
\newcommand\InProf{\Prof^\mathrm{in}}%
\newcommand\Centers{M}%

\newcommand\BB{\mathcal{B}}%
\newcommand{\BBass}{\BB_{\text{\it ass}}}%
\newcommand{\BBcyc}{\BB_{\text{\it cyc}}}%
\newcommand{\BBflag}{\BB_{\text{\it flag}}}%

\newcommand\PL{\mathrm{PL}}%
\newcommand\Contr{\mathcal{E}}%

\newcommand\Traj{\mathcal{S}}%

\newcommand\Grp{\mathfrak{G}}%
\newcommand\Ob{\mathrm{Ob}}%

\DeclareMathOperator{\relint}{relint}
\DeclareMathOperator{\interior}{int}
\DeclareMathOperator{\aff}{aff}
\DeclareMathOperator{\lin}{lin}
\DeclareMathOperator{\lineal}{lineal}
\DeclareMathOperator{\conv}{conv}

\DeclareMathOperator{\im}{im}

\newtheorem{thm}{Theorem}[section]
\newtheorem{cor}[thm]{Corollary}
\newtheorem{lem}[thm]{Lemma}
\newtheorem{prop}[thm]{Proposition}

\newtheorem{quest}{Question}

\theoremstyle{definition}
\newtheorem{dfn}[thm]{Definition}
\newtheorem{example}[thm]{Example}
\newtheorem{rem}[thm]{Remark}

\title[Inscribable fans I]{Inscribable fans I:\\ Inscribed cones and virtual
polytopes}

\author{Sebastian Manecke} 
\author{Raman Sanyal}

\address{Institut f\"ur Mathematik, Goethe-Universit\"at Frankfurt, Germany} 
\email{manecke@math.uni-frankfurt.de}
\email{sanyal@math.uni-frankfurt.de}



\date{\today}

\binoppenalty=\maxdimen
\relpenalty=\maxdimen

\begin{document}

\begin{abstract}
    We investigate polytopes inscribed into a sphere that are normally
    equivalent (or strongly isomorphic) to a given polytope $P$. We show that
    the associated space of polytopes, called the \emph{inscribed cone} of
    $P$, is closed under Minkowski addition. Inscribed cones are interpreted
    as type cones of ideal hyperbolic polytopes and as deformation spaces of
    Delaunay subdivisions.  In particular, testing if there is an inscribed
    polytope normally equivalent to $P$ is polynomial time solvable.

    Normal equivalence is decided on the level of normal fans and we study the
    structure of inscribed cones for various classes of polytopes and fans,
    including simple, simplicial, and even. We classify (virtually)
    inscribable fans in dimension $2$ as well as inscribable permutahedra and
    nestohedra.

    A second goal of the paper is to introduce inscribed \emph{virtual}
    polytopes. Polytopes with a fixed normal fan $\Fan$ form a monoid with
    respect to Minkowski addition and the associated Grothendieck group is
    called the \emph{type space} of $\Fan$.  Elements of the type space
    correspond to formal Minkowski differences and are naturally equipped with
    vertices and hence with a notion of inscribability. We show that inscribed
    virtual polytopes form a subgroup, which can be non-trivial even if $\Fan$
    does not have actual inscribed polytopes.

    We relate inscribed virtual polytopes to routed particle trajectories,
    that is, piecewise-linear trajectories of particles in a ball with
    restricted directions. The state spaces gives rise to connected groupoids
    generated by reflections, called \emph{reflection groupoids}. The
    endomorphism groups of reflection groupoids can be thought of as discrete
    holonomy groups of the trajectories and we determine when they are
    reflection groups.
\end{abstract}

\maketitle

\tableofcontents

\section{Introduction}\label{sec:intro}

Let $P \subset \R^d$ be a convex polytope. We call $P$ \Def{inscribed} if its
vertices $V(P)$ lie on a common sphere. Prominent examples include the regular
polytopes and, more generally, vertex-transitive polytopes; cf.~\cite{Coxeter}.
In 1832 Steiner (cf.~\cite{Steiner}) asked if every $3$-dimensional polytope is
\Def{inscribable}, that is, if for every $3$-polytope $P$ there is an inscribed
polytope $P'$ that is combinatorially equivalent to $P$. Almost a century later,
Steinitz~\cite{Steinitz} gave a simple combinatorial criterion for
non-inscribability and with it the first counterexample to Steiner's question.
Using Steinitz seminal result that combinatorial types of $3$-polytopes are in
one-to-one correspondence with $3$-connected planar graphs (on $\ge 4$ nodes),
Rivin~\cite{Rivin} showed that testing inscribability can be done in polynomial
time: to a given $3$-connected planar graph, a system of linear equations and
strict inequalities is associated and feasibility of said system is equivalent
to an inscribed realization; see \eqref{eqn:Rivin}. This settled the
inscribability problem in dimensions $\le 3$; see~\cite[Ch.13.5]{grunbaum}
or~\cite{PadrolZiegler} for more on the historic background.

By Corollary~4.16 of~\cite{AdiprasitoPadrolTheran}, the problem of finding an
inscribed polytope combinatorially equivalent to a given (simplicial) polytope
$P$ is polynomial-time equivalent to the \emph{existential theory of the reals}
(ETR) and hence NP-hard. In this paper, we focus on a more restrictive form of
equivalence: two polytopes $P$ and $P'$ are \Def{normally equivalent} (written
$P \simeq P'$) if $P$ is combinatorially isomorphic to  $P'$ and corresponding
faces are parallel.  Figure~\ref{fig:normally_equiv} shows three
normally-equivalent polytopes. 
\begin{figure}[h]
    \centering
     \includegraphics[width=0.2\textwidth]{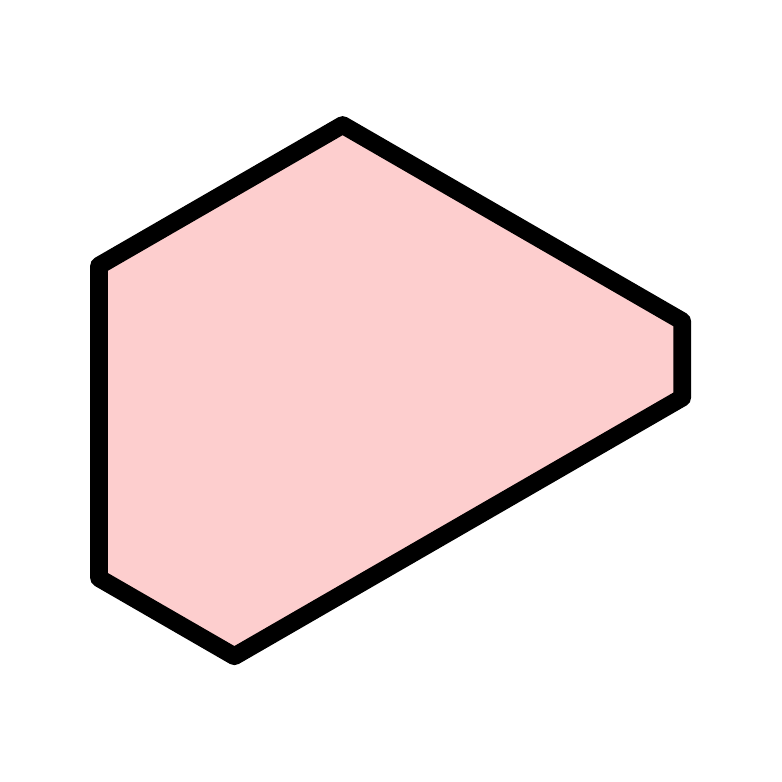}
    \qquad
    \qquad
     \includegraphics[width=0.2\textwidth]{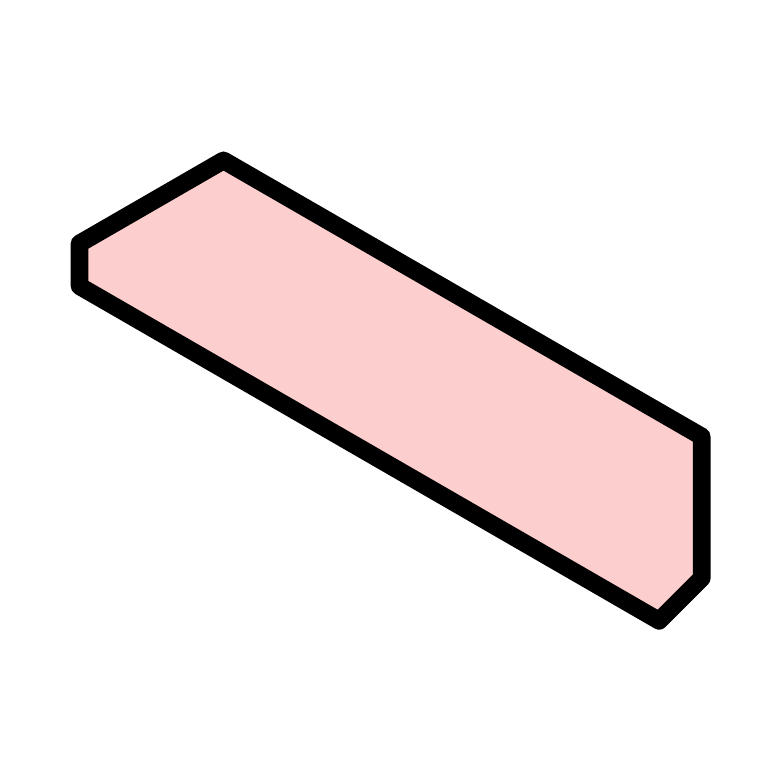}
    \qquad
    \qquad
     \includegraphics[width=0.2\textwidth]{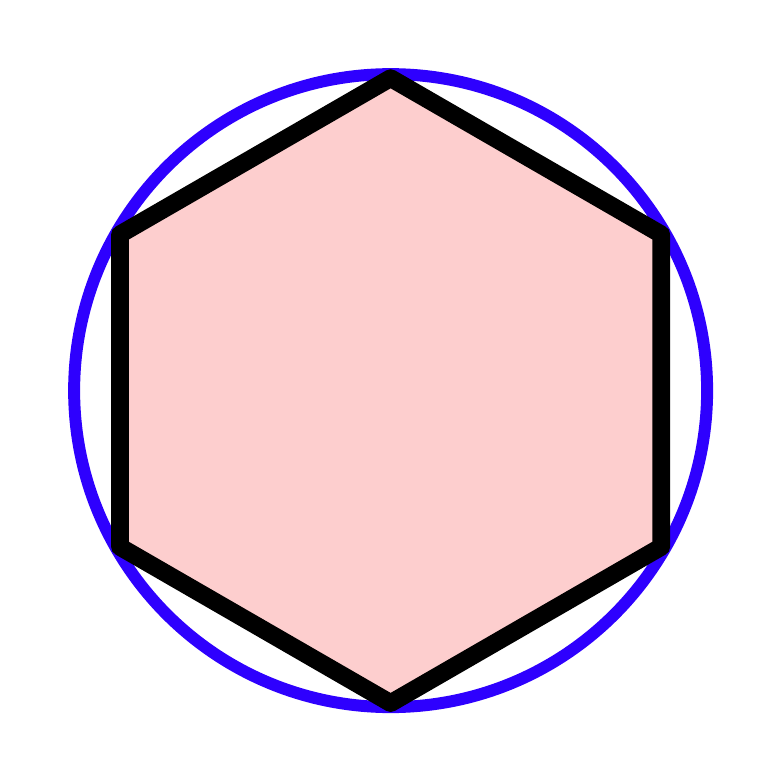}
    \caption{Three normally equivalent hexagons. The rightmost one is
    inscribed.}
    \label{fig:normally_equiv}
\end{figure}
Normal equivalence is a natural and important notion. It occurs, for example,
in McMullen's proof of the $g$-Theorem~\cite{McMullen-simple}, in the study of
torus invariant divisors on toric varieties~\cite[Ch.~6.3]{CLS}, and in
parametric linear programming~\cite{Adler}. In particular, the deformation
space of polytopes normally equivalent to $P$ is (simply) connected, which
prompts the following question:

\begin{center}
    \it
    Is there an inscribed polytope $P'$ normally equivalent to $P$?
\end{center}

We define the \Def{inscribed cone} 
\[
    \InCone(P) \ \defeq \ \{ P' \simeq P : P' \text{ inscribed} \} 
    \ / \ \text{translations}
\]
as the space of inscribed polytopes normally equivalent to $P$ up to
translation.  A polytope $P$ is said to be \Def{normally inscribable} if
$\InCone(P) \neq \emptyset$. Recall that the Minkowski sum of two polytopes
$Q,Q' \subset \R^d$ is the polytope $Q + Q' \defeq \{q + q' : q \in Q, q' \in
Q'\}$. It is easy to see that if $Q \simeq Q' \simeq P$, then $Q + Q' \simeq
P$. Our first main result shows that $\InCone(P)$ has an intriguing structure,
which justifies the name.

\begin{thm} \label{thm:InSpc}
    Let $P$ be a convex $d$-dimensional polytope and $Q,Q' \in \InCone(P)$.
    Then
    \[
        Q + Q' \in \InCone(P) \, .
    \]
    In particular, $\InCone(P)$ has the structure of an open polyhedral cone of
    dimension $\le d$.
\end{thm}

A direct application of Theorem~\ref{thm:InSpc} shows how symmetry is affected
by inscribability.

\begin{cor}\label{cor:symm_intro}
    Let $P \subset \R^d$ be a convex polytope invariant with respect to the action of a
    finite group of orthogonal transformations $G$. If $P$ is normally inscribable, then
    there is an inscribed polytope $P' \in \InCone(P)$ that is invariant under $G$. 
\end{cor}

We give a first proof of Theorem~\ref{thm:InSpc} in
Section~\ref{sec:ref_game}. In Section~\ref{sec:type} we give an \emph{intrinsic}
representation of $\InCone(P)$ that makes no reference to the embedding of $P$
(Theorem~\ref{thm:VInSpc_rep}). We obtain an explicit description in terms of
linear inequalities and equations, which then allows us to give a satisfactory
algorithmic answer to the above question:

\begin{thm}\label{thm:algo_rational}
    Testing if a (rational) polytope $P$ has a normally equivalent inscribed
    polytope can be done in polynomial time.

    Moreover, if $P$ is rational and $\InCone(P) \neq \emptyset$, then there is
    a rational inscribed polytope $P' \simeq P$.
\end{thm}

Whereas testing if a given polytope $P$ is inscribable is tantamount to
solving a system of polynomial equations and strict inequalities
(see~\cite[Section 5]{PadrolZiegler}), the proof of
Theorem~\ref{thm:algo_rational} shows that testing if $P$ is normally
inscribable can be phrased in terms of linear equations and strict
inequalities. Moreover, the feasiblity problem only makes use of the geometric
graph of $P$ and is similiar to that of Rivin~\cite{Rivin}.

Let us write $\Fan(P)$ for the normal fan of $P$. This is a complete and
strongly connected polyhedral fan whose maximal cones are precisely the
domains of linearity of the support function of $P$; see
Section~\ref{sec:ref_game}. It is straightforward to check that $P \simeq P'$
if and only if $\Fan(P) = \Fan(P')$. Hence, the class of polytopes normally
equivalent to $P$ is represented by $\Fan(P)$. We call a fan $\Fan$
\Def{inscribable} if there is an inscribed polytope $P$ with $\Fan = \Fan(P)$
and we denote by $\InCone(\Fan)$ the inscribed cone of $\Fan$. Thus,
$\InCone(P) = \InCone(\Fan(P))$.  We study inscribed cones for certain classes
of polytopes and fans in Section~\ref{sec:gen_insc_fans}. In particular, we
give a characterization of polytopes $P$ for which $\dim \InCone(P)$ has
maximal dimension (Theorem~\ref{thm:inspc_even}). At the other extreme, we
show that the $\dim \InCone(P) \le 1$ whenever $P$ is a simplicial polytope
(Corollary~\ref{cor:ins_simplicial}) and we show that this can even happen for
simple polytopes (Corollary~\ref{cor:3polytope_odd}).
Using a local characterization (Theorem~\ref{thm:inscribability_j_k})
of inscribability for general polytopes, we show that a
simple polytope $P$ is inscribed if and only if all $k$-faces are inscribed
for some fixed $k > 1$ (Corollary~\ref{cor:simple_inscribed}). In the
follow-up paper~\cite{InFan2}, we use that to show that inscribed zonotopes
are determined by their $2$-faces, extending the local characterization of
zonotopes due to Bolker~\cite{Bolker}.

Many of these results rely on Section~\ref{sec:dim2}, in which we completely
characterize (virtually) inscribable fans in the plane. We show that the set
of (virtually) inscribable $2$-dimensional fans with a fixed number of regions
has the structure of a polytope (Proposition~\ref{prop:2d_virtual_insc},
Theorem~\ref{thm:2d_insc}).

As an example, we investigate the fan determined by the braid arrangement
$\braid_{d-1}$ in Section~\ref{sec:permutahedra}. We show that the polytopes
in $\InCone(\braid_{d-1})$ are vertex-transitive with respect to the symmetric
group and hence permutahedra. The example is continued in
Section~\ref{sec:nestohedra}, where we determine the inscribed nestohedra, a
class of simple, generalized permutahedra~\cite{post}. This characterizes all
matroid polytopes (in the sense of~\cite{Coxetermatroids}) among nestohedra.

In the upcoming paper~\cite{InFan2}, we study the inscribed cones for fans
coming from general hyperplane arrangements and show strong ties to simplicial
arrangements and reflection groups.

The following two subsections highlight two significant implications of
Theorem~\ref{thm:InSpc}.

\subsection{Ideal hyperbolic type cones}
The \emph{projective disk} (or Beltrami--Klein) model identifies
hyperbolic space $\Hyp^d$ with the points in the open unit ball
$D^d = \{ x \in \R^d : \|x\| < 1 \}$. Hyperplanes in $\Hyp^d$
correspond to sets of the form $D^d \cap H$, where $H \subset \R^d$ is
an ordinary hyperplane meeting $D^d$. The two halfspaces induced by
$H$ are convex and finite intersections of halfspaces give rise to
hyperbolic polyhedra. In particular, \Def{hyperbolic polytopes}
correspond to polytopes contained in the unit ball
$B^d = \overline{D^d}$. A vertex $v$ of $P$ is called \Def{ideal} if
$v \in \partial B^d = S^{d-1}$ and $P \subseteq B^d$ is an \Def{ideal
  hyperbolic polyhedron} if all vertices are ideal. For more
information on hyperbolic geometry and hyperbolic polytopes we refer
to~\cite[Ch.~6]{Ratcliffe} and~\cite{Thurston}

In this non-conformal model for hyperbolic space, hyperbolic polytopes are
simply Euclidean polytopes and thus are equipped with the notion of normal
equivalence and normal fans. Steiner's question and Rivin's result pertain to
the combinatorics of \emph{ideal} hyperbolic polyhedra.
Theorem~\ref{thm:InSpc} states that Minkowski sums extend to ideal hyperbolic
polytopes with fixed normal fan. Let $Q,Q' \subset B^d$ be ideal
hyperbolic polytopes. For a generic linear function $l(x)$, let $q$ and $q'$
be the unique vertices maximizing $l(x)$ over $Q$ and $Q'$, respectively. We
define the \Def{angle} between $Q$ and $Q'$ by $\theta(Q,Q') \in [0,\pi)$ as
the angle between $q$ and $q'$. 

\begin{cor}\label{cor:ideal_sum}
    Let $Q,Q' \subset \R^d$ be normally equivalent ideal hyperbolic polytopes.
    The angle $\theta(Q,Q')$ is independent of the choice of linear function
    $l(x)$. Moreover, 
    \[
        \frac{1}{\sqrt{2 + 2 \cos \theta(Q,Q')}} (Q + Q')
    \]
    is again an ideal hyperbolic polytope.
\end{cor}

If $Q$ and $Q'$ are ideal hyperbolic polytopes which are not normally
equivalent, then their Minkowski sum might still be ideal, provided
$Q$ and $Q'$ are \emph{inscribed relative to $\Fan$} with respect to
the fan $\Fan(Q + Q')$; see Section~\ref{sec:relative} for details.

McMullen~\cite{McMullen-rep} introduced the \Def{type cone} of a polytope $P$
as the space of polytopes normally equivalent to $P$ up to translation
\[
    \TypeCone(P) \ \defeq \ \{ P' : P \simeq P' \} 
    \ / \ \text{translations} \, .
\]
It follows from work of Shephard (cf.~\cite[Sect.~15.1]{grunbaum}) that
$\TypeCone(P)$ is an open polyhedral cone.  In that sense, we may view
$\InCone(P)$ as the \Def{ideal hyperbolic type cone} of an ideal hyperbolic
polytope $P$. In Section~\ref{sec:relative}, we investigate the structure of
the closure $\cInCone(P)$, which turns out to be more subtle than that of
$\cTypeCone(P)$.

We believe that further study of the relationship between $\InCone(P)$ and
$\TypeCone(P)$ and their implications for ideal hyperbolic polytopes (such as
ideal hyperbolic flips, for example) will be very exciting.

\subsection{Deformations of Delaunay subdivisions}
A \Def{Delaunay subdivision} of a full-dimen\-sional polytope $P \subset \R^{d-1}$ is a
subdivision $\Del = \{ P_1,\dots,P_m \}$ into full-dimensional and inscribed polytopes
that satisfy the following condition for all $i=1,\dots,m$: If $B$ is the unique ball to which $P_i$ is
inscribed, then $B \cap V(P_j) \subseteq V(P_i)$ for all $j \neq i$.  Delaunay
subdivisions and in particular Delaunay triangulations play an important role in
numerical computations; see~\cite{Aurenhammer} for more. The basic computational task is
for a given $U \subset \R^{d-1}$ to construct a Delaunay subdivision $\Del(U)$ of $P =
\conv(U)$ such that $V(P_i) \subseteq U$ for all $i=1,\dots,m$ and $U = \bigcup_i
V(P_i)$.

Brown~\cite{BROWN1979223} observed a simple correspondence between inscribed
polytopes and Delaunay subdivisions. Let $\eta : S^{d-1} \to \R^{d-1} \times
\{0\}$ be the stereographic projection from the north-pole $e_{d}$ to the
equatorial plane $\R^{d-1} \times \{0\}$. The polytope $\hat{P} \defeq \conv(
\eta^{-1}(U))$ is inscribed into the unit sphere.  The \Def{visibility
complex} of $\hat{P}$ with respect to some fixed $\xi \in S^{d-1}$ is the
collection of faces $F \subset \hat{P}$ such that $\conv(\xi \cup F)$ does not
meet the interior of $\hat{P}$. Then the collection of facets of $\hat{P}$ not
visible from $e_d$ stereographically projects to a Delaunay subdivision of $P
= \conv(U)$. 

Thus every configuration $U$ has a Delaunay subdivision but slight
perturbations of the points can result in drastic changes in the combinatorics
of $\Del(U)$. In Section~\ref{sec:delaunay}, we distill a notion of
\emph{normal equivalence} for (labelled) Delaunay subdivisions, which allows
us to interpret $\InCone(\hat{P})$ as a \emph{deformation space} of the
Delaunay subdivisions.  Figure~\ref{fig:del_norm_equiv} shows two normally
equivalent Delaunay subdivisions.
\begin{figure}[h]
    \centering
    \includegraphics[width=0.35\textwidth]{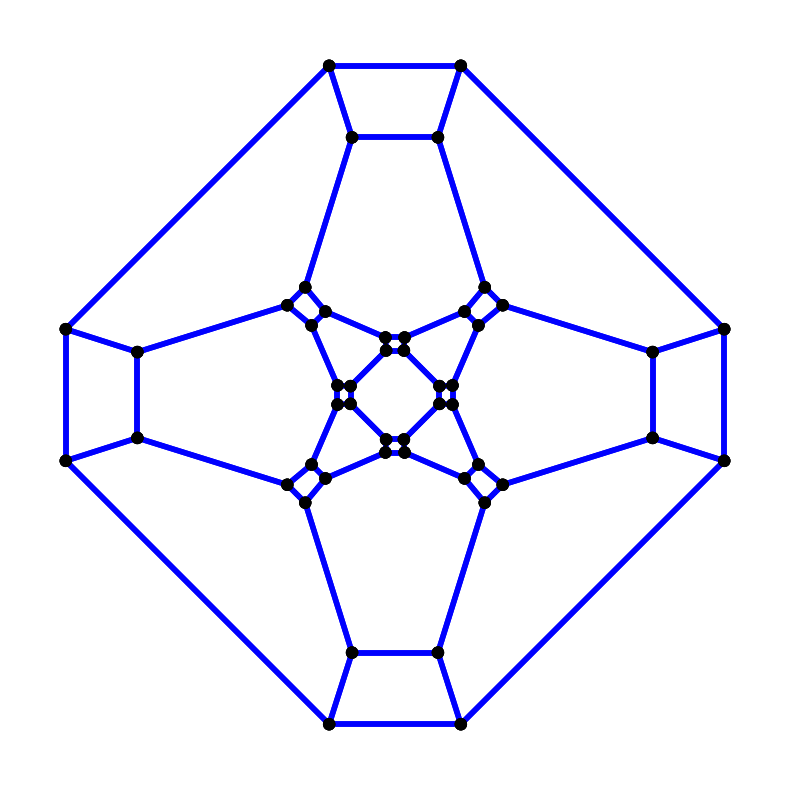}
    \qquad\qquad
    \includegraphics[width=0.35\textwidth]{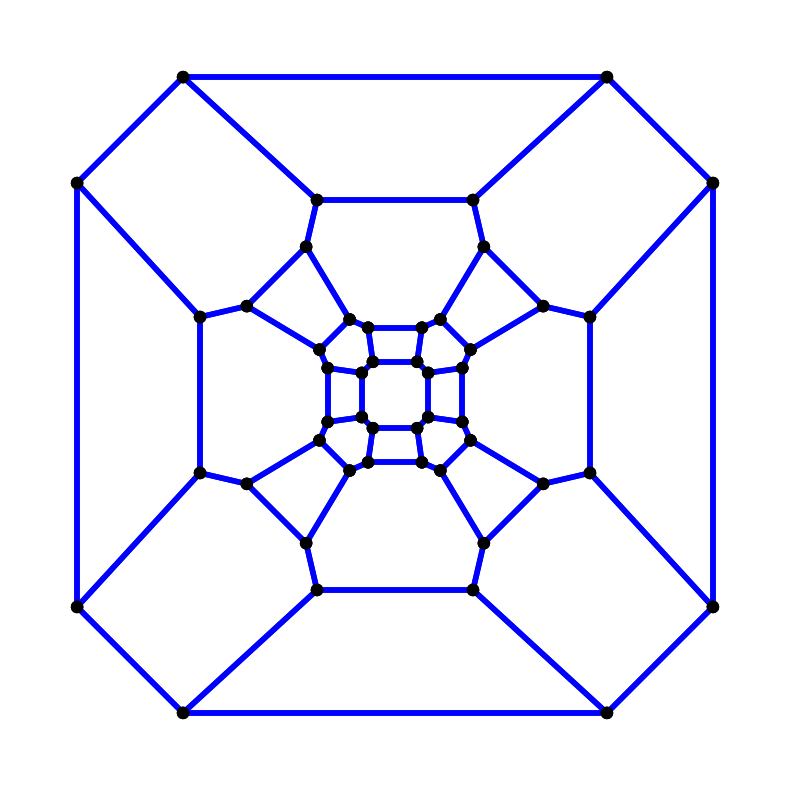}
    \caption{Two positively co-circular Delaunay subdivisions.}
    \label{fig:del_norm_equiv}
\end{figure}

\begin{cor}\label{cor:delaunay_deform}
    Let $U \subset \R^{d-1}$ be an affinely-spanning point configuration.  The
    space of Delaunay subdivisions $\Del(U')$ normally equivalent to $\Del(U)$
    has the structure of a spherical polytope of dimension $\le d-1$.
\end{cor}

If $\Del(U)$ and $\Del(U')$ are normally equivalent, then it is in generally
not true that $\conv(U)$ and $\conv(U')$ are normally equivalent or even
combinatorially equivalent. This stems from the fact that two normally
equivalent polytopes $\hat{P},\hat{P}'$ inscribed to the unit sphere can have
quite different visibility complexes with respect to a fixed point $\xi \in
S^{d-1}$. In Theorem~\ref{thm:del_equiv} we determine equivalence relation on
normally equivalent polytopes inscribed to the unit sphere with fixed
visibility complex. The equivalence classes are nonconvex in general and can
be disconnected. We give a simple necessary condition when a cell in this
subdivision is convex (Corollary~\ref{cor:del_equiv_convex}).

\subsection{Inscribed virtual polytopes}\label{sec:intro_virtual}
A second goal of this paper is to introduce and study inscribed \emph{virtual}
polytopes. A fan $\Fan$ is \Def{polytopal} if there is a polytope $P$ with
$\Fan(P) = \Fan$ and we define $\TypeCone(\Fan) \defeq \TypeCone(P)$, the open
polyhedral cone of polytopes with normal fan $\Fan$ modulo translations.  If
$P,Q,R \in \TypeCone(\Fan)$ satisfy $P = Q+R$, then $R$ is called the
\Def{Minkowski difference} of $P$ and $Q$ and is denoted by $P-Q$. Minkowski
differences exist for all pairs $P$ and $Q$ in the Grothendieck group
$\TypeSpc(\Fan) \defeq (\TypeCone(\Fan) \times \TypeCone(\Fan))/\!\!\sim$ with
$(Q+R,Q'+R) \sim (Q,Q')$ for $Q,Q',R \in \TypeCone(\Fan)$ and $P-Q$ is called
is called a \emph{virtual polytope} if $P-Q \in \TypeSpc(\Fan) \setminus
\TypeCone(\Fan)$. Since $\TypeCone(\Fan)$ is a convex cone, $\TypeSpc(\Fan)$
has the structure of an $\R$-vector space and will be called the \Def{type
space} of $\Fan$. We recall in Section~\ref{sec:PL} that $\TypeSpc(\Fan)$ can
be defined in terms of piecewise-linear functions supported on $\Fan$ and is
thus also defined for non-polytopal fans.  Virtual polytopes are related to
non-nef divisors in toric geometry~\cite[Ch.~6.1]{CLS} and they embody
reciprocity results for translation-invariant valuations by means of
McMullen's polytope algebra~\cite{McMullen-algebra}. For more on virtual
polytopes, see~\cite{PS}. Important for us is that virtual polytopes are
naturally equipped with \emph{vertices} and hence a notion of
\emph{inscribability} (Section~\ref{sec:virtual}).
Figure~\ref{fig:non_insc_hexagon} depicts some examples of inscribed virtual
polytopes with respect to a fixed fan.

\begin{figure}[h]
    \centering
    \includegraphics[width=0.20\textwidth]{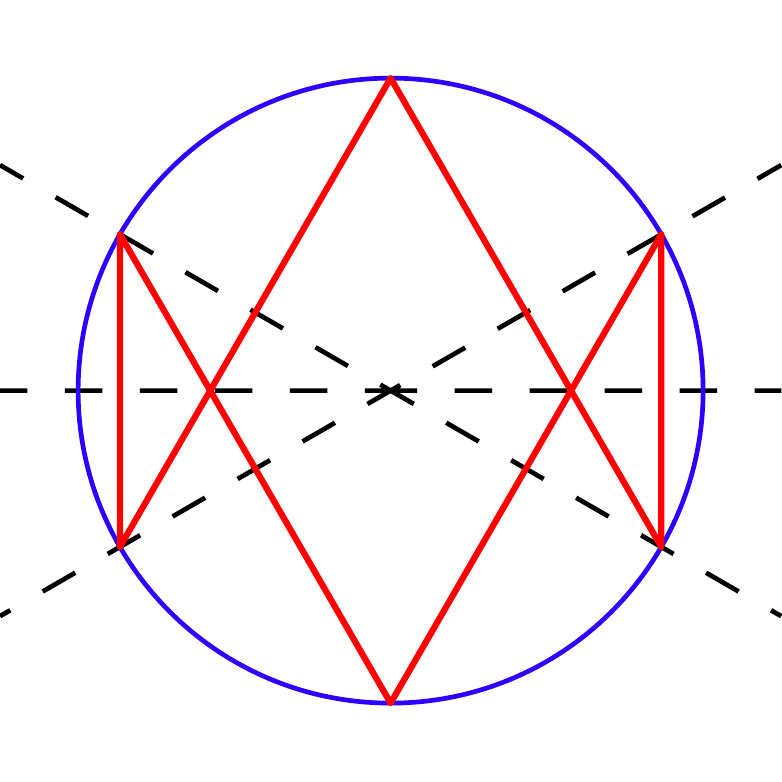}\qquad
    \includegraphics[width=0.20\textwidth]{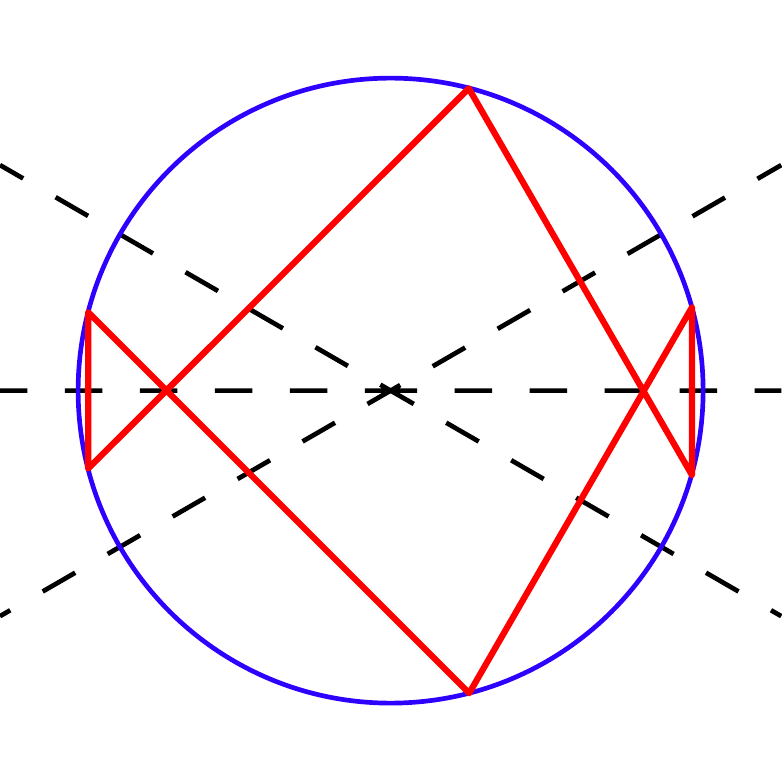}\qquad
    \includegraphics[width=0.20\textwidth]{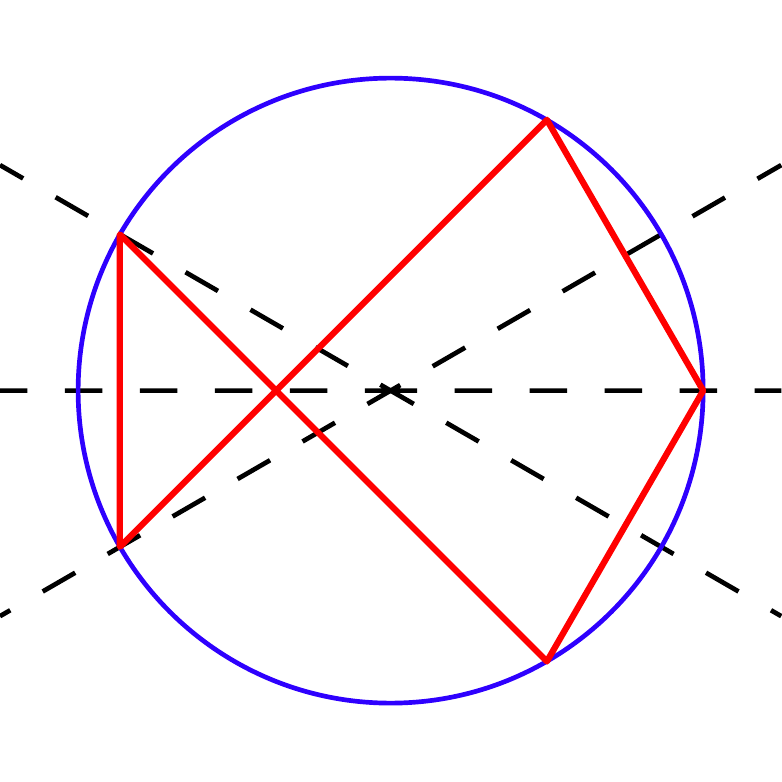}\qquad
    \includegraphics[width=0.20\textwidth]{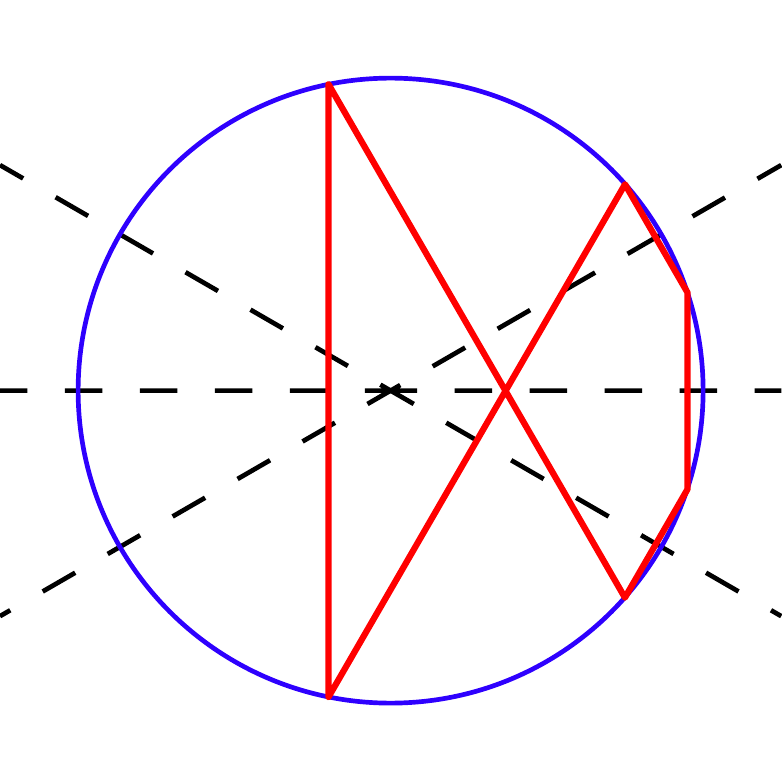}
    \qquad
    \caption{A $2$-dimensional fan all whose inscribed polytopes are virtual.}
    \label{fig:non_insc_hexagon}
\end{figure}

We study inscribed virtual polytopes in Section~\ref{sec:virtual} and show
that virtual polytopes inscribed relative to $\Fan$ form a vector
subspace $\InSpc(\Fan) \subseteq \TypeSpc(\Fan)$.  Naturally, $\InCone(\Fan) +
(-\InCone(\Fan)) \subseteq \InSpc(\Fan)$ with equality if $\Fan$ is
inscribable. However, there are fans that only possess virtual inscribed
polytopes. For example, we show that $\InSpc(\Fan)$ is always $1$-dimensional
for every $2$-dimensional fan with an odd number of rays
(Proposition~\ref{prop:2d_virtual_insc}).

\subsection{Routed billiard trajectories and reflection groupoids}
The inscribed virtual polygons in Figure~\ref{fig:non_insc_hexagon} are
reminiscent of closed, piecewise-linear trajectories of particles inside the
unit disc that bounce off the boundary in random directions. We make this
analogy precise in Section~\ref{sec:erratic}, where we introduce routed
particle trajectories. We model the state-space of a particle as a graph $G =
(V,E)$ together with a map $\alpha : E \to \PP^{d-1}$, which encodes the
admissible directions. The pair $(G,\alpha)$ is called a \Def{routing scheme}.
A trajectory is then a map $T : V \to S^{d-1}$ that
records positions of a trajectory routed by $(G,\alpha)$. We show that the
space of trajectories is isomorphic to a spherical subspace
(Theorem~\ref{thm:subsphere}) and
we show that for routing schemes $(G,\alpha)$ derived from a fan, routed
trajectories correspond to inscribed virtual polytopes
(Theorem~\ref{thm:traj_inspc}). 

From the perspective of state spaces, routing schemes give rise to groupoids,
whose morphisms are generated by reflections and are therefore called
\Def{reflection groupoids}. We discuss reflection groupoids in
Section~\ref{sec:groupoids} and, in particular, study their associated
endomorphism groups. These groups can be thought of as discrete holonomy
groups and generalize Joswig's groups of projectivities~\cite{joswig}.

\begin{rem}
    All results regarding inscribed cones and inscribed spaces  remain valid
    if we replace the unit sphere $S^{d-1} = \{ x \in \R^d : \inner{x,x} = 1
    \}$ with a general, non-degenerate quadric $\mathcal{Q} = \{ x :
    \inner{Ax,x} = 1 \}$. See also~\cite{DMS} for work related to
    $3$-polytopes inscribed in a general quadric.
\end{rem}

\begin{rem}
    A polyhedron $Q \subseteq \R^d$ is \emph{inscribed} if all vertices lie on
    a sphere $S$ and all unbounded edges meet $S$ only in a vertex. All our
    results can be adapted to inscribed polyhedra.
\end{rem}

\subsection*{Acknowledgements} 
We thank Michael Cuntz, Thilo R\"orig, Christian Stump, and Martin Winter for
insightful discussions. Research that led to this paper was supported by the
DFG-Collaborative Research Center, TRR 109 ``Discretization in Geometry and
Dynamics'' and we also thank our colleagues of project A3 for their support.
Many of our findings were inspired by experiments and computations conducted
with SAGE~\cite{SAGE} and Geogebra~\cite{geogebra}.

\section{Local reflections and inscribable fans}\label{sec:ref_game}

A non-empty collection $\Fan$ of polyhedral cones in some $\R^d$ is called a
\Def{fan}~\cite[Sect.~7]{ziegler} if 
\begin{enumerate}[\rm (F1)]
    \item if $C \in \Fan$ and $F \subseteq C$ is face, then $F \in \Fan$;
    \item if $C, C' \in \Fan$, then $C \cap C' \in \Fan$.
\end{enumerate}
The dimension of $\Fan$ is
$\dim \Fan \defeq \max\{\dim C : C \in \Fan\}$. The inclusion-maximal
cones of $\Fan$ are called \Def{regions} and $\Fan$ is \Def{pure} if
all regions have the same dimension.  The \Def{support} of $\Fan$ is
$|\Fan| = \bigcup_{C \in \Fan} C$ and $\Fan$ is \Def{complete} if
$|\Fan| = \R^d$. A convex cone $C$ is \Defn{pointed}, if its \Defn{lineality
 space} $\lineal(C) \defeq \{x \in C : -x \in C\}$ contains only the
origin. All cones in a fan $\Fan$ share the same lineality space
$\lineal(\Fan)$ and we will therefore call $\Fan$ pointed if
$\lineal(\Fan) = \{\0\}$.

Let $P \subset \R^d$ be a non-empty convex polytope. For $c \in \R^d$, we
write
\[
    P^c \ \defeq \ \{ x \in P : \inner{c,x} \ge \inner{c,y} \text{ for all } y
    \in P \} 
\]
for the (non-empty) face that maximizes the linear function $x \mapsto
\inner{c,x}$. 

The \Def{normal cone} of $P$ at a face $F \subseteq P$ is the polyhedral cone
\[
    N_F P \ \defeq \ \{ c \in \R^d : F \subseteq P^c  \} \, .
\]
It is easy to verify that $\Fan(P) = \{ N_F P : F \subseteq P \text{ face} \}$
is a complete fan, called the \Def{normal fan} of $P$. The normal fan is
pointed precisely when $P$ is full-dimensional. We call a fan $\Fan$
\Def{polytopal} if it is the normal fan of a polytope.

Let $\aff(P)$ be the affine hull of $P$.  Recall that two polytopes $P_0,P_1
\subset \R^d$ are normally equivalent ($P_0 \simeq P_1$) if for every $c \in
\R^d$ the affine spaces $\aff(P_0^c)$ and $\aff(P^c_1)$ differ by a
translation. The upcoming characterization of normally equivalent polytopes follows
from the definition of normal fans; see, for example, \cite[Section
7.2]{ziegler}.
\begin{prop}\label{prop:normally_equiv}
    Let $P_0, P_1 \subset \R^d$ polytopes. Then
    \[
        P_0 \ \simeq \ P_1 \quad \text{ if and only if } \quad \Fan(P_0) \ = \
        \Fan(P_1) \, .
    \]
    In particular, $(1 - \mu) P_0 + \mu P_1$ is normally equivalent to $P_0$
    for all $0 \le \mu \le 1$.
\end{prop}
The proposition shows that $\TypeCone(\Fan)$ is a convex cone that depends
only on $\Fan(P)$; see Section~\ref{sec:PL} for details.

\subsection{Local reflections}
The following is the key observation in the proof of Theorem~\ref{thm:InSpc}.

\begin{lem}\label{lem:key}
    Let $P \subset \R^d$ be inscribed to a sphere centered at the
    origin with normal fan $\Fan(P) = \Fan$. Then $P$ is completely
    determined by $\Fan$ and a single vertex.
\end{lem}
\begin{proof}
    We may assume that $P$ is full-dimensional and hence $\Fan$ is pointed.
    Let $v \in V(P)$ be a vertex with normal cone $N_v P$. The polyhedral cone
    has an irredundant representation of the form
    \[
        N_v P \ = \ \{ c \in \R^d : \inner{\alpha_i,c} \le 0 \text{ for } i =
        1,\dots,m \}  
    \]
    for some $\alpha_1,\dots,\alpha_m \in \R^d \setminus \{0\}$.
    If $u \in V(P)$ is a neighbor of $v$, then $u = v + \lambda_i \alpha_i$ for
    some $i \in [m] \defeq \{1,\dots,m\}$ and $\lambda_i > 0$. Since $P$ is
    inscribed into a sphere centered at the origin, we have
    \begin{equation}\label{eqn:key}
        \|v\|^2 \ = \ \|u\|^2 \ = \ \| v + \lambda_i \alpha_i \|^2 \, .
    \end{equation}
    This is a quadratic equation in $\lambda_i$ with a unique solution 
    $\lambda_i > 0$. Thus, knowing $v$ and $N_v P$, we can uniquely recover
    the neighbors of $v$. As the graph of $P$ is connected, we can recover all
    vertices of $P$.
\end{proof}

Let $\Fan$ be a pure $d$-dimensional fan in $\R^d$. The cones of dimension
$d-1$ in $\Fan$ are called \Def{walls}. Every wall $W \in \Fan$
induces a hyperplane $\lin(W) = \{ x : \inner{\alpha,x} =
0 \}$ and we let $s_W : \R^d \to \R^d$ 
\[
    s_W(x) \ \defeq \ x - 2 \frac{\inner{\alpha,x}}{\inner{\alpha,\alpha}} \alpha
\]
be the corresponding \Def{reflection}. By inspecting the proof of
Lemma~\ref{lem:key}, we make the following observation.

\begin{cor}\label{cor:key_reflect}
    Let $P \subset \R^d$ be a polytope inscribed to a sphere centered at the
    origin and $v \in V(P)$. The neighbors of $v$ are given by $s_W(v)$ where
    $W$ ranges over the walls of $N_v P$.
\end{cor}
\begin{proof}
    The unique nonzero solution of~\eqref{eqn:key} is given by $\lambda_i  =
    -2 \frac{\inner{\alpha_i,v}}{\inner{\alpha_i,\alpha_i}}$ and hence
    \[
        u \ = \ v - 2 \frac{\inner{\alpha_i,v}}{\inner{\alpha_i,\alpha_i}}
        \alpha_i \ = \ s_W(v)
        \, ,
    \]
    where $W = \{ c \in N_vP : \inner{\alpha_i,c} = 0 \}$ is the wall of $N_vP$
    corresponding to $\alpha_i$.
\end{proof}

A second observation drawn from the proof of Lemma~\ref{lem:key} is the
following relation between vertices and their normal cones.

\begin{cor}\label{cor:key_int}
    Let $P \subset \R^d$ be a polytope inscribed to a sphere centered at the
    origin.  Then $v \in \interior(N_vP)$ holds for every vertex $v \in V(P)$.
\end{cor}
\begin{proof}
    Observe that the nonzero solution $\lambda_i$ to~\eqref{eqn:key} is
    positive if and only if $\inner{\alpha_i,v} < 0$. Thus $v \in \{ c :
    \inner{\alpha_i,c} < \nolinebreak 0 \text{ for } i =1,\dots, m \} =
    \interior(N_vP)$.
\end{proof}

\begin{dfn}[Inscribed cone of a fan]
    Let $\Fan$ be a fan in $\R^d$. The \Def{inscribed cone} of $\Fan$
    is the set of all inscribed polytopes $P$ with $\Fan(P) = \Fan$
    modulo translations. We call $\Fan$ \Def{inscribable} if
    $\InCone(\Fan) \neq \emptyset$.
\end{dfn}

Calling $\InCone(\Fan)$ a \emph{cone} is justified as $\lambda P \in
\InCone(\Fan)$ for all $P \in \InCone(\Fan)$ and $\lambda > 0$.
Theorem~\ref{thm:InSpc} asserts that $\InCone(\Fan)$ is in fact a \emph{convex}
cone. 

For every translation class in $\InCone(\Fan)$, a canonical representative can
be obtained as follows: Let $P \subset \R^d$ be a polytope inscribed into a sphere $S$. If $P$
is full-dimensional, then $S$ is unique. If $P$ is of lower dimension,
then $S \cap \aff(P)$ is the unique inscribing sphere relative to its
affine hull. We write $c(P)$ for the \Def{center} of $S \cap
\aff(P)$. Since $c(P + \x) = c(P) + \x$ for all $\x \in \R^d$, we can
always assume that $c(P) = \0$ and we write $\overline{P} \defeq P - c(P)$.

The space $\InCone(\Fan)$ is endowed with the Hausdorff metric 
\[
    d_H(P, Q) \ \defeq \ \min \{ \mu \ge 0 : \overline{P} \subseteq
    \overline{Q} + \mu B^d, \overline{Q} \subseteq \overline{P} + \mu B^d \}
    \, ,
\]
where $B^d$ is the unit ball and $P,Q \in \InCone(\Fan)$.

Let $P$ be a polytope with normal fan $\Fan$.  For a fixed region $R_0 \in
\Fan$, we write $v_{R_0}(P)$ for the unique vertex $v$ of $P$ with $N_vP = R_0$.
Let us denote the set of possible $v_{R_0}(P)$ for inscribed $P$ by
\[
    \InCone(\Fan, R_0) \ \defeq \ \{ v_{R_0}(P) : \Fan(P) = \Fan, \text{$P$
      inscribed}, c(P) = \0 \}\,.
\]
We call $\InCone(\Fan, R_0)$ the \Def{inscribed cone based at $R_0$}.

It follows from Lemma~\ref{lem:key} and
Corollaries~\ref{cor:key_reflect} and~\ref{cor:key_int} that the map
$v_{R_0} : \InCone(\Fan) \to \InCone(\Fan, R_0)$ is a
homeomorphism.

\subsection{Virtually inscribable fans and the reflection game}
In order to find necessary conditions for a fan to be inscribed, we use the
following \emph{reflection game} for fans: Let $\Fan$ be a pure and
full-dimensional fan. The \Def{dual graph} of $\Fan$ is the simple undirected
graph $G(\Fan)$ with nodes given by the regions of $\Fan$. Two regions $R, R'$
are adjacent in $G(\Fan)$ if $R \cap R'$ is a wall. We call $\Fan$
\Def{strongly connected} if $G(\Fan)$ is connected.  For example, every
complete fan is strongly connected.  If $R,R'$ are two adjacent regions, then
let $s_{RR'}$ be the reflection in the hyperplane $\lin(R \cap R')$. Every
walk $\Walk = R_0 R_1 \dots R_k$ in $G(\Fan)$ yields an orthogonal
transformation
\[
    t_\Walk \ \defeq \ s_{R_{k}R_{k-1}}\cdots s_{R_2R_1} s_{R_1R_0} \, .
\]
\begin{dfn}[Virtually inscribable]
    Let $\Fan$ be a full-dimensional and strongly connected fan in
    $\R^d$ and let $R_0 \in \Fan$ be a region. The fan $\Fan$ is
    \Def{virtually inscribable} if there is a point
    $v \in \R^d \setminus \lineal(\Fan)$ such that
    \begin{equation}\label{eqn:t_W}
        t_\Walk(v) \ = \ v
    \end{equation}
    for all closed walks $\Walk$ starting in $R_0$.
\end{dfn}

Note that we do not require that $\Fan$ is polytopal. Also note that every
$t_\Walk$ fixes $\lineal(\Fan)$ pointwise. The linear subspace
$\InSpc(\Fan,R_0) \subset \R^d$ of all $v \in (\lineal(\Fan))^\perp$
satisfying~\eqref{eqn:t_W} for all closed walks $\Walk$ will be called the
\Def{based inscribed space} of $\Fan$. Thus, $\Fan$ is virtually inscribable
if and only if $\InSpc(\Fan) \neq \{\0\}$. The actual choice of $R_0$ is
immaterial: for a different base region $R_0'$, we have
\[
    \InSpc(\Fan,R_0') \ = \  t_{\Walk'} \InSpc(\Fan,R_0)
\]
for any walk $\Walk'$ from $R_0$ to $R_0'$. Hence $\InSpc(\Fan,R_0)$ is
\emph{based} at $R_0$. In Section~\ref{sec:computing}, we will discuss
inscribed spaces that do not require the choice of a base region.  Clearly,
\[
    \InCone(\Fan,R_0) \ \subseteq \ \InSpc(\Fan,R_0) \, .
\]

\begin{prop}\label{prop:ref_game}
    Let $\Fan$ be a complete fan. Then $\Fan$ is inscribed if and only
    if $\Fan$ is virtually inscribed and there is
    $v_0 \in \InSpc(\Fan,R_0)$ such that for every region $R$
    \[
        t_\Walk(v_0) \in \interior(R)
    \]
    for some path $\Walk$ from $R_0$ to $R$.
\end{prop}
\begin{proof}
    If $P \subset \R^d$ is a polytope inscribed into a sphere centered at
    the origin and $\Fan(P) = P$, then Lemma~\ref{lem:key} and
    Corollary~\ref{cor:key_reflect} show that $\Fan$ is virtually inscribable.
    The graph of $P$ is exactly $G(\Fan)$ and $t_W(v_0) = u$ is the vertex
    with region $R$. Corollary~\ref{cor:key_int} now shows that $u \in
    \interior(R)$.

    For the converse, let $\Fan$ be a virtually inscribed fan
    satisfying the given conditions. For every region $R$, let
    $v_R \defeq t_\Walk(v_0)$, where $\Walk$ is a path connecting
    $R_0$ to $R$. Note that this implies $v_R \neq v_S$ for
    $R \neq S$. Since $\Fan$ is virtually inscribed, $v_R$ is
    independent of the chosen path. Define
    $P = \conv(v_R : R \text{ region} )$. Since $t_\Walk$ is an
    orthogonal transformation, all vertices lie on a sphere centered
    at the origin with radius $\|v_0\|$.  In particular $P$ is an
    inscribed polytope with vertices $v_R$ for $R \in \Fan$ region.
    
    We are left to show that $\Fan(P) = \Fan$.  By construction $N_{v_R}
    \subseteq R$ for every region $R$. Indeed, for every region $S$ adjacent
    to $R_0$, $v_S - v_0$ is an outer normal for $R_0$ and 
    \[
        N_{v_{R_0}} \ = \  \{ c \in \R^d : \inner{c,v_{R_0}} \ge \inner{c,v_R}
        \text{ for $R \in \Fan$ region} \} \, .
    \]
    As $P$ is independent of the choice of $R_0$, $N_{v_R} \subseteq R$ holds
    for all regions. Since $\Fan$ and $\Fan(P)$ are complete fans, this
    implies $N_{v_R} = R$ for all regions $R$ and completes the proof.
\end{proof}

We are now in the position to prove Theorem~\ref{thm:InSpc}.

\begin{proof}[Proof of Theorem~\ref{thm:InSpc}]
    Let $Q, Q'$ be two polytopes inscribed into a sphere centered at the
    origin with $\Fan(Q) = \Fan(Q') = \Fan$. For every region $R \in \Fan$,
    let $q_R, q'_{R} \in R$ be the respective vertices, whose existence is
    vouched for by Corollary~\ref{cor:key_int}. Since $\Fan(Q + Q') = \Fan$,
    we have
    \[
        Q+Q' \ = \ \conv\{ q_R  + q'_{R} : R \in \Fan \text{ region} \} \, .
    \]
    To show that $Q+Q'$ is inscribed, let $R_0 \in \Fan$ be a fixed
    region. By convexity, $q_{R_0} + q'_{R_0} \in \interior(R_0)$ and
    $q_R + q'_R = t_\Walk(q_{R_0} + q'_{R_0}) \in \interior(R_0)$ for
    all walks $\Walk$ from $R_0$ to any region $R \in \Fan$.
    Proposition~\ref{prop:ref_game} now shows that
    $Q + Q' \in \InCone(P)$.
    
    It remains to see that $\InCone(\Fan) \cong \InCone(\Fan, R_0)$ is a
    relatively open polyhedral cone of dimension at most $d$. The collection
    of points $v_0 \in \R^d$ with $t_\Walk(v_0) \in \interior(R)$ for $\Walk$
    a path from $R_0$ to $R$ is an open polyhedral cone of dimension $d$. If
    $v_0 \in \InSpc(\Fan,R_0)$, then $t_\Walk(v_0)$ is independent of the
    choice of $\Walk$.  Hence, if we choose a path $\Walk_R$ from $R_0$ to $R$
    for every region $R$, then
    \begin{equation}\label{eqn:InSpc_intersect_rep}
        \InCone(\Fan, R_0) \ = \ \InSpc(\Fan,R_0) \ \cap \ \bigcap_R
        t_{\Walk_R}^{-1}(\interior(R)) \, .
    \end{equation}
    The latter is the restriction of a linear subspace to the intersection of
    finitely many open polyhedral cones.
\end{proof}

\begin{cor} \label{cor:InSpc_iso}%
   The map $v_{R_0} : \InCone(\Fan) \to \InCone(\Fan, R_0)$ is a linear
   homeomorphism:
\begin{equation}\label{eqn:InSpc_real1}
   \InCone(\Fan)  \ \cong \ \InCone(\Fan, R_0)\,.
\end{equation}
\end{cor}

We are now in a position to proof Corollary~\ref{cor:symm_intro} and Corollary~\ref{cor:ideal_sum}.

\begin{proof}[Proof of Corollary~\ref{cor:symm_intro}]
    Let $P$ be an normally inscribable polytope and let $G$ be a group of orthogonal
    transformations such that $gP = P$ for all $g \in G$. For $P_0 \in \InCone(P)$
    consider 
    \[
        P' \ \defeq \ \frac{1}{|G|} \sum_{g \in G} g P_0 \, .
    \]
    $P'$ is clearly invariant under $G$ and since $\InCone(\Fan)$ is convex by 
    Theorem~\ref{thm:InSpc}, it follows that $P' \in \InCone(\Fan)$.
\end{proof}

\begin{proof}[Proof of Corollary~\ref{cor:ideal_sum}]
    Let $Q,Q' \subseteq \R^d$ be polytopes inscribed into the unit sphere with $\Fan(Q) =
    \Fan(Q') = \Fan$. Fix a region $R_0$ and let $q = v_{R_0}(Q)$ and $q' =
    v_{R_0}(Q')$. For any region $R \in \Fan$, we have that $v_R(Q) =
    t_\Walk(q)$ and $v_R(Q') = t_\Walk(q')$ for any path $\Walk$ from $R_0$ to
    $R$.  Since $t_\Walk$ is a product of reflection, it follows that
    $\inner{t_\Walk(q),t_\Walk(q')} = \inner{q,q'}$ and hence $\theta(Q,Q')$
    is independent of the choice of $R_0$ or, equivalently, the choice of a
    generic linear function $l(x)$.

    For the second statement, we simply note that $\| q + q' \|^2 = 2 + 2
    \inner{q,q'} = 2 + 2 \cos \theta(Q,Q')$.
\end{proof}

\subsection{Inscribed Permutahedra}\label{sec:permutahedra}

The \Defn{braid arrangement} $\braid_{d-1}$ is the arrangement of linear
hyperplanes 
\[
    H_{ij} \ = \ \{ x \in \R^d : x_i = x_j \} \quad  \text{ for } 1 \le i
    < j \le d \, .
\]
Every connected component of $\R^{d} \setminus \bigcup_{H \in
\braid_{d-1}} H$ is an open cone, and the closures of these cones form a
fan which we will also denote by $\braid_{d-1}$.  Note that $\braid_{d-1}$
is not pointed.  Its lineality space is the line spanned by
$(1,1,\dots,1)$.  It is straightforward to verify that the $d!$ regions of
$\braid_{d-1}$ are given by 
\[ 
    R_\sigma \ \defeq \ \{ z \in \R^d : z_{\sigma(1)} \leq
    z_{\sigma(2)} \leq \dots \leq z_{\sigma(d)} \} \, ,
\]
where $\sigma \in \Sym_{d}$ is a permutation. Two regions
$R_\sigma$ and $R_\tau$ are adjacent if they differ by an adjacent
transposition, that is, $\sigma\tau^{-1} = (i,i+1)$ for some $1 \le i < d$.

Let $R_0 = \{ x_1 \le x_2 \le \cdots \le x_d \}$ be the region for the
identity permutation.  If $\Walk$ is a path from $R_0$ to $R_\sigma$, then
$t_\Walk(x_1,\dots,x_d) = (x_{\sigma(1)},\dots,x_{\sigma(d)})$ is the
permutation of coordinates by $\sigma$. Thus $t_{\Walk}^{-1}(R_\sigma) = R_0$.
This shows that $\InSpc(\braid_{d-1}, R_0) = \R^d$ and we can conclude
\[
    \InCone(\braid_{d-1}, R_0) \ = \ R_0 \ = \ \{ z : z_1 \le z_2 \le
    \dots \le z_d \} \]
and for $z \in R_0$, the corresponding polytope is 
\[
    P(z) \ = \ \conv \{ (x_{\sigma(1)},\dots,x_{\sigma(d)}) : \sigma \in 
    \Sym_{d} \}
\]
a \Def{permutahedron} or \Def{weight polytope}
of type $\braid_{d-1}$; cf.~\cite{Coxeter_submodular_functions}.

The rays of the closure $\cInCone(\braid_{d-1}) \cong R_0$ are of the form
$(0,\dots,0,1,\dots,1)$ and the associated polytopes are precisely the
\emph{hypersimplices} $\Delta(d,k)$ for $0 < k < d$; see next section.

\begin{cor}\label{cor:permutahedra}
    Let $\braid_{d-1}$ be the fan of the braid arrangement.
    Then every $P \in \cInCone(\braid_{d-1})$ is of the form
    \[
        P \ = \ P(z) \ = \ z_1 \Delta(d,1) + (z_2 - z_1) \Delta(d,2) + \cdots
        + (z_d - z_{d-1}) \Delta(d,d-1)
    \]
    for $z= (z_1 \le z_2 \le \cdots \le z_d)$.
    In particular, every $P \in \cInCone(\braid_{d-1})$ is symmetric with
    respect to $\Sym_d$.
\end{cor}

\subsection{Delaunay subdivisions and visibility complexes}\label{sec:delaunay}

\newcommand\hP{\hat{P}}%
Let $[n] = \{1,\dots,n\}$. A \Def{labelled point configuration} is an injective
map $U : [n] \to \R^{d-1}$. We will mostly identify $U$ with its underlying set
$U([n])$.  Let $P = \conv(U)$ be the corresponding convex hull and let $\hat{P}
= \conv(\eta^{-1}(V))$, where $\eta : S^{d-1} \to \R^{d-1} \times \{0\}$ is
the stereographic projection from $e_d$. The projection of the faces of $\hP$
not visible from $e_d$ is the Delaunay subdivision $\Del(U)$ of $P$.  The
distinguishing feature of the Delaunay subdivision is that it is the coarsest
subdivision such that every face $F$ of some cell $P_i \in \Del(U)$ is
inscribed to some $(d-1)$-dimensional ball $B$ such that $U \cap B = V(F)$.

A pair of distinct points $u_1,u_2 \in U$ is a \Def{hidden edge} of $\Del(U)$
if the segment $[u_1,u_2]$ is inscribed to some ball $B$ containing all points
$U \setminus \{u_1,u_2\}$ in its interior.  It follows that if $u_1u_2$ is a
hidden edge, then $u_1, u_2$ are contained in the boundary of $P$. We write
$G(U)$ for the graph with nodes $[n]$ and $i,j \in [n]$ form an edge if $U(i)U(j)$
is an edge or a hidden edge of $\Del(U)$. It is not hard to see that $G(U)$ is
exactly the edge graph of $\hP$.

Let $I \subset \R^{d-1}$ be a segment. There is a unique sphere $S = S(I)$
such that $I$ is invariant under inversion in $S$ and $S$ meets $S^{d-2}$ in a
great-sphere. We call two segments $I,I'$ \Def{co-circular} if $S(I) = S(I')$.
If $I$ and $I'$ are oriented, then they are \Def{positively} co-circular if
the positive endpoints of $I$ and $I'$ are not separated by $S(I)$.
Figure~\ref{fig:co-circular} shows three positively co-circular segments.
\begin{figure}[h]
    \centering
    \includegraphics[width=0.40\textwidth]{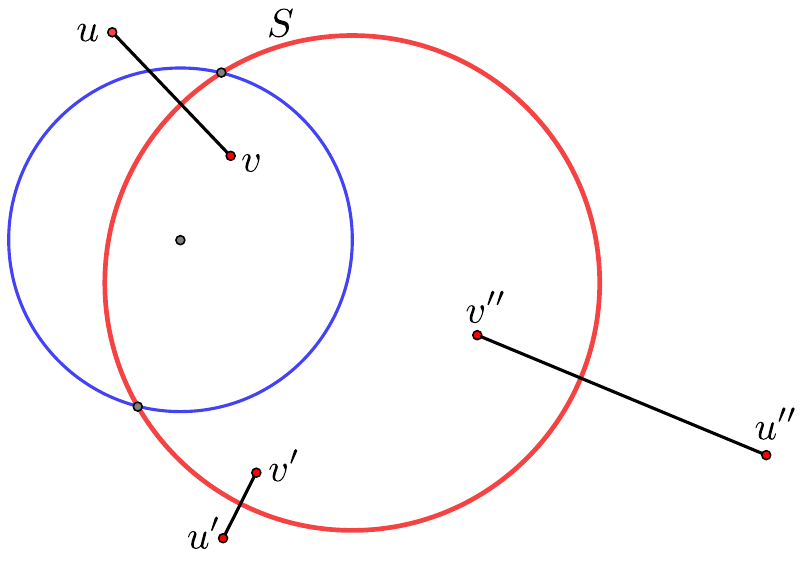}
    \caption{Three positively co-circular segments in the plane.}
    \label{fig:co-circular}
\end{figure}
Let $U$ and $U'$ be labelled point configurations. We call $\Del(U)$ and
$\Del(U')$ \Def{normally equivalent} if $G(U) = G(U')$ as labelled graphs and
for every edge $e$ of $G(U)$, the corresponding segments of $U$ and $U'$ are
positively co-circular.

With respect to stereographic projection, we get the following.
\begin{prop}
    Let $U,U'$ be labelled configurations. Then $\Del(U)$ and $\Del(U')$ are
    normally equivalent if and only if $\hP(U) \simeq \hP(U')$.
\end{prop}

\begin{proof}[Proof of Corollary~\ref{cor:delaunay_deform}]
    Normally equivalent Delaunay subdivisions are represented by a polyhedral
    fan $\Fan$ in $\R^d$. Recalling that for a fixed region $R_0$, the cone
    $\InCone(\Fan,R_0)$ represents all polytopes $P \subset \R^d$ with
    $\Fan(P) = \Fan$ and inscribing sphere centered at the origin, we obtain
    that the spherical polytope $S(\Fan,R_0) \defeq \InCone(\Fan,R_0) \cap
    S^{d-1}$ parametrizes all normally equivalent Delaunay subdivisions in
    $\R^{d-1}$ represented by $\Fan$.
\end{proof}

\newcommand\Vis{\mathrm{Vis}_\xi}%
\newcommand\ur{{\overline{r}}}%
Choose $\xi \in S^{d-1}$ and let $P \subset \R^d$ be a
full-dimensional polytope inscribed to the unit sphere. Recall from the
introduction that the \Def{visibility complex} $\Vis(P)$ is the collection of
faces $F \subseteq P$ such that $\conv(\xi \cup F)$ does not meet the interior
of $P$. We call such $F$ \Def{visible} from $\xi$.  The visibility complex is
a full-dimensional pure polyhedral complex and hence is determined by the
collection of facets of $P$ visible from $\xi$. For a ray $r \in \Fan$, let us
denote by $\ur \defeq r \cap S^{d-1}$ the unit vector in $r$.  Every ray $r$
determines the oriented hyperplane 
\[
    H_r \ \defeq \ \{ x \in \R^d : \inner{\ur,x} = \inner{\ur,\xi} \}
\]
It follows that the facet $P^\ur$ is visible from $\xi$ if and only if $P^\ur$
is contained in the open halfspace
\[
    H^-_r \ = \ \{ x \in \R^d : \inner{\ur,x} < \inner{\ur,\xi} \} \, .
\]
In fact, if $v$ is a vertex of $P$ contained in $P^\ur$, then this is
equivalent to $v \in H^-_r$. We denote the opposite closed halfspace by $H^+_r$.

\newcommand\HypArr{\mathcal{H}}%
Let $\Fan$ be a polyhedral fan with base region $R_0$.  For every ray $r \in
\Fan$, choose a path $\Walk_r$ to a region $R$ with $r \in R$. We define the
arrangement of hyperplanes
\[
    \HypArr(\Fan,R_0,\xi) \ \defeq \ \{ t_{\Walk_r}^{-1}(H_r) : r \in \Fan \text{
    ray} \} \, .
\]
The arrangement determines an equivalence relation on $\R^d$ by setting $x
\sim y$ if and only if $|\{x,y\} \cap H^-| \neq 1$ for all $H \in
\HypArr(\Fan,R_0,\xi)$. If $P$ is a polytope with normal fan $\Fan$ inscribed
to the unit sphere, then $t_{\Walk_r}(v_{R_0}(P)) \in P^\ur$ for all rays $r$.
This proves the following.

\begin{thm}\label{thm:del_equiv}
    Let $\Fan$ be an inscribable fan with base region $R_0$ and $\xi \in
    S^{d-1}$. Then $P,Q \in \InCone(\Fan,R_0)$ have the same visibility complex
    if and only if $v_{R_0}(P)$ and $v_{R_0}(Q)$ are equivalent with respect
    to $\HypArr(\Fan,R_0,\xi)$.
\end{thm}

Note that the equivalence classes on $S^{d-1}$ with respect to $\sim$ are not
necessarily convex or even connected. This is due to the fact that the
hyperplanes $H_r$ do not pass through the origin and therefore do not induce
great-spheres. We close this section with a simple criterion when the set of
points in $S(\Fan.R_0)$ with a fixed visibility complex is convex. If $H = \{
x : \inner{c,x} = \delta \}$ is an affine hyperplane with normal vector
$c$, then $H^- \cap S^{d-1}$ is a convex subset if and only if $\delta < 0$.

\begin{cor}\label{cor:del_equiv_convex}
    Let $M$ be a collection of rays of $\Fan$. Then the polytopes $P \in
    S(\Fan,R_0)$ with visibility complex induced by $P^{\ur}$ for $r \in M$
    form a (spherically) convex subset if $\inner{\ur_i,\xi} < 0$ for $r \in
    M$ and $\inner{\ur_i,\xi} \ge 0$ for all rays $r \not\in M$.
\end{cor}

\section{Inscribable fans in dimension $2$}\label{sec:dim2}

In this section we study inscribable fans in the plane. The situation in the
plane is simple enough to give a complete classification of inscribable fans.
However, since faces of inscribable polytopes are inscribable, the results
obtained in this section give simple necessary conditions for inscribable fans
in higher dimensions.

Throughout this section, let $\Fan$ be a complete and pointed fan in
$\R^2$.  We order the $n\ge 3$ regions of $\Fan$ counterclockwise and
denote them by $R_0,\dots,R_{n-1}$. For $i = 0,\dots,n-1$, let
$\beta_i$ the angle of $R_i$ (cf.~Figure~\ref{fig:2fan}). We call
$\beta(\Fan) = (\beta_0,\dots,\beta_{n-1})$ the \Def{profile} of
$\Fan$. The profile $\beta(\Fan)$ determines $\Fan$ up to rotation.

\begin{figure}[ht]
    \centering
    \includegraphics[width=0.35\textwidth]{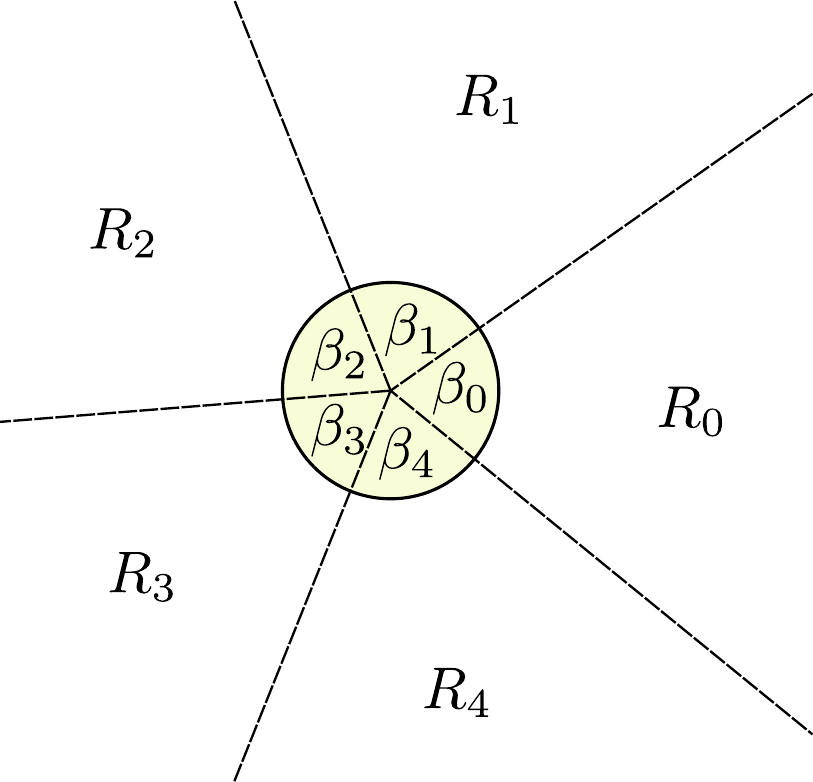}\qquad\qquad
    \includegraphics[width=0.35\textwidth]{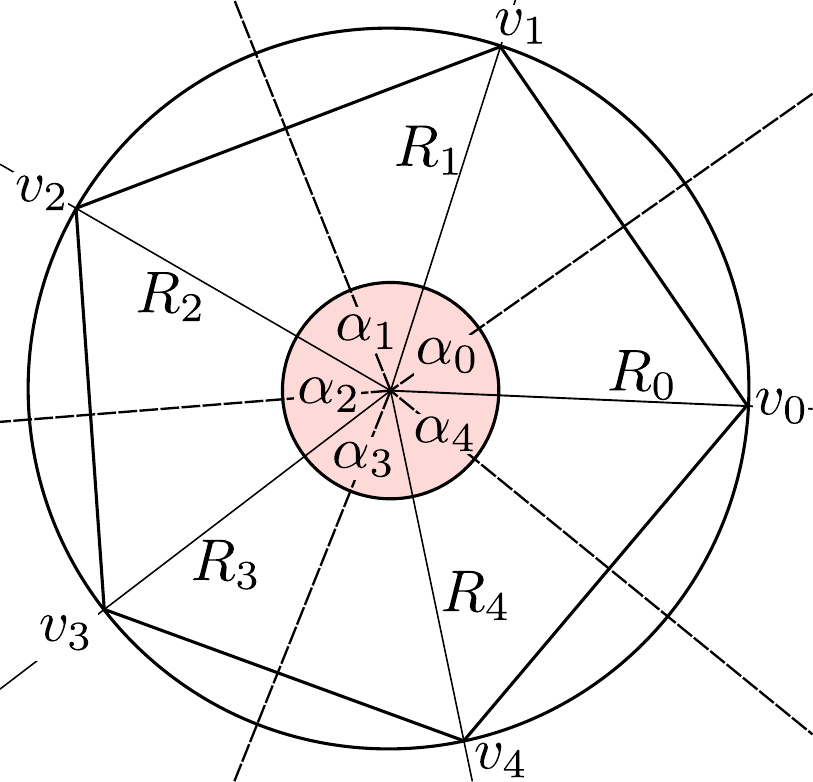}
    \caption{ Left: A $2$-dimensional fan $\Fan$ with regions and
      its profile $\beta = (\beta_0, \beta_1, \beta_2, \beta_3, \beta_4)$.
      Right: An inscribed polygon with normal fan $\Fan$ and labels as in
      the proof of Theorem~\ref{thm:2d_insc}.
      \label{fig:2fan} }
\end{figure}

It is clear that the set of all profiles of complete fans with $n$ regions is
\[
    \Prof_n \ \defeq \ \left \{ \beta \in \R^n : 
    \begin{array}{c} 
        0 < \beta_i < \pi \text{ for } i=0,\dots,n-1 \\
        \beta_0 + \cdots + \beta_{n-1} = 2 \pi
    \end{array} \right \} \, .
\]
Recall that the \Def{$(n,k)$-hypersimplex}~\cite[Ex.~0.11]{ziegler}  is the
polytope given by the convex hull of points $v \in \{0,1\}^n$ with exactly $k$
entries equal to $1$. The $(n,1)$-hypersimplex is the \Def{standard simplex}
$\Delta_{n-1} = \{ x \in \R^n : x \ge 0, x_1+\cdots+x_n = 1\} =
\conv(e_1,\dots,e_n)$, where $e_i$ are the standard basis vectors.

It follows that $\Prof_n$ is the relative interior of $\pi \cdot
\Delta(n,2)$. This description also highlights the fact that cyclic shifts of
$\beta$ correspond to cyclic relabellings of $\Fan$. To ease notation, we
decree
\[
    \beta_{n+i} \defeq \beta_i \quad \text{ for } 0 \le i < n \, .
\]
Our first result determines the locus of virtually inscribable fans.

\begin{prop}\label{prop:2d_virtual_insc}
    Let $\Fan$ be a $2$-dimensional fan with $n$ regions and $\beta(\Fan) =
    (\beta_0,\dots,\beta_{n-1})$. If $n$ is odd, then $\Fan$ is virtually
    inscribable and $\dim \InSpc(\Fan) = 1$. If $n$ is even, then $\Fan$ is
    virtually inscribable if and only if
    \[
        \beta_2 + \beta_4 + \cdots + \beta_{n-2} \ = \ \pi \, .
    \]
    In this case $\dim \InSpc(\Fan) = 2$.
\end{prop}
\begin{proof}
    The only relevant closed walk is $\Walk = R_0R_1\dots R_{n-1}R_0$ and the
    corresponding transformation 
    \[
        t_\Walk \ \defeq \ s_{R_{0}R_{n-1}}\cdots s_{R_2R_1} s_{R_1R_0} 
    \]
    is a product of $n$ reflections in $\R^2$. If $n$ is odd, then $t_\Walk$
    is a reflection and hence there is a unique $v \in \R^2$ up to scaling
    such that $t_\Walk(v) = v$. If $n$ is even, then $t_\Walk$ is a rotation
    by $2(\beta_2 + \beta_4 + \cdots + \beta_n)$ and $t_\Walk$ has a fixpoint
    if and only if $\beta_2 + \beta_4 + \cdots + \beta_n = \pi$. In this case
    $t_\Walk = \mathrm{id}$ and $\InSpc(\Fan) = \R^2$.
\end{proof}

\begin{cor}
    Let $\VInProf_n \subseteq \Prof_n$ be the profiles of virtually
    inscribable fans.  If $n$ is odd, then $\VInProf_n = \Prof_n$. If $n=2k$
    is even, then $\frac{1}{\pi}\VInProf_n$ is the relative interior of
    $\Delta_{k-1} \times \Delta_{k-1}$, the product of two standard simplices.
\end{cor}

\begin{example}\label{ex:insc_quadrangle}
    Let $\Fan$ be a two-dimensional fan with $4$ regions. Then
    Proposition~\ref{prop:2d_virtual_insc} shows that $\Fan$ is
    virtually inscribed if and only if its profile
    $\beta(\Fan) = (\beta_0, \beta_1, \beta_2, \beta_3)$ satisfies:
    \[
        \beta_0 + \beta_2 \ = \ \beta_1 + \beta_3 \ = \ \pi\,.
    \]
    This is precisely the condition that $\Fan$ is the fan of a cyclic
    quadrangle of Euclidean geometry and therefore any realization of
    $\Fan$ will be inscribed, if $\Fan$ satisfies this condition. Thus,
    any rhombus ($\beta_0 = \beta_2$, $\beta_1 = \beta_3$) which is
    not a rectangle can not be (virtually) inscribed, while any
    isosceles trapezoid ($\beta_0 = \beta_1$, $\beta_2 = \beta_3$) is
    virtually inscribed (see Figure~\ref{fig:insc_quadrangle}).
    \begin{figure}[h]
        \centering
        \begin{subfigure}{.3\textwidth}
            \includegraphics[width=\textwidth]{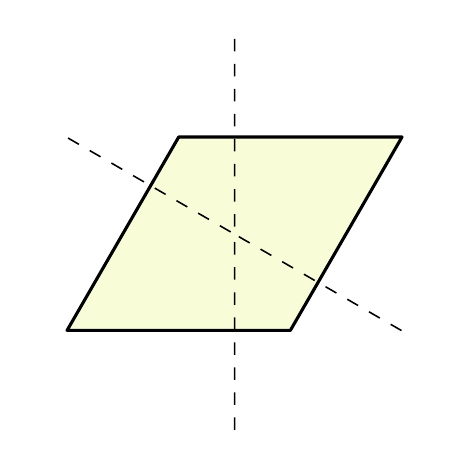}
            \caption{}
        \end{subfigure}
        \begin{subfigure}{.3\textwidth}
            \includegraphics[width=\textwidth]{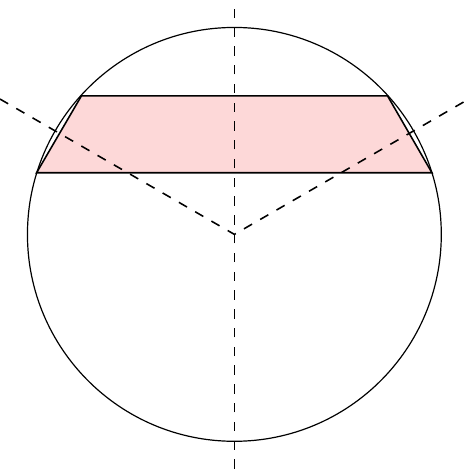}
            \caption{}
        \end{subfigure}
        \begin{subfigure}{.3\textwidth}
            \includegraphics[width=\textwidth]{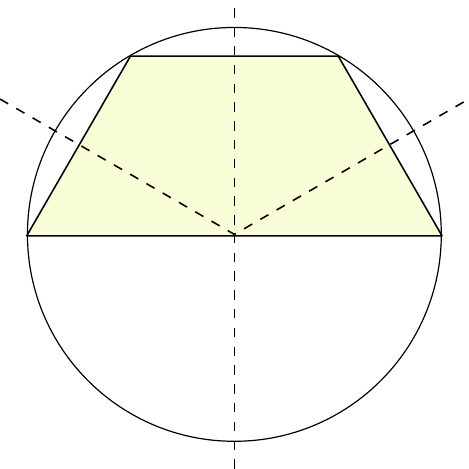}
            \caption{}
        \end{subfigure}
        \begin{subfigure}{.3\textwidth}
            \includegraphics[width=\textwidth]{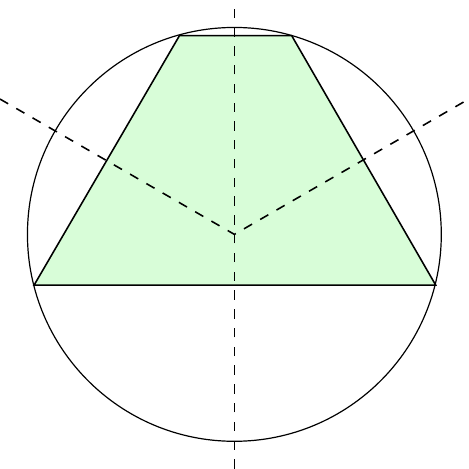}
            \caption{}
        \end{subfigure}
        \begin{subfigure}{.3\textwidth}
            \includegraphics[width=\textwidth]{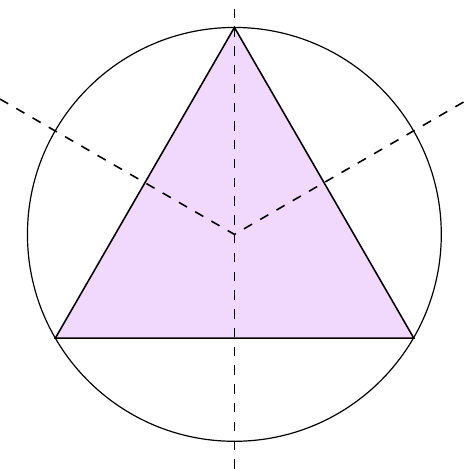}
            \caption{}
        \end{subfigure}
        \begin{subfigure}{.3\textwidth}
            \includegraphics[width=\textwidth]{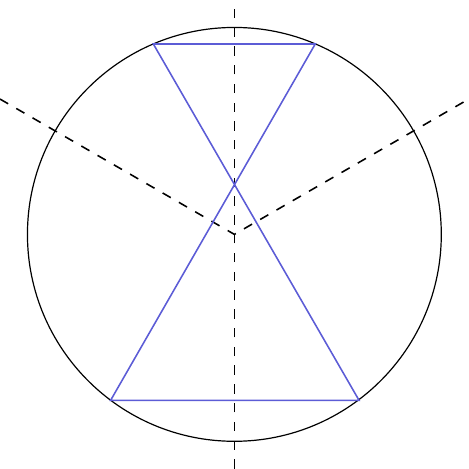}
            \caption{}
        \end{subfigure}
        \caption{(A): Normal fan of rhombus which can not be virtually
          inscribed.
          (B) - (D): Multiple inscribed realizations of the normal fan of an
          isosceles trapezoid. (E) - (F): Virtually inscribed realizations.
          While the trajectory of (E) closes to a polygon, its
          normal fan differs. The trajectory of (F) overlaps with itself.
        \label{fig:insc_quadrangle}
      }
    \end{figure}
\end{example}

\begin{example}\label{ex:noninsc_pentagon}
    Figure~\ref{fig:noninsc_pentagon} shows a two-dimensional fan with
    profile the
    $\beta = (\frac{2\pi}{3}, \frac{\pi}{3}, \frac{\pi}{3},
    \frac{\pi}{3}, \frac{\pi}{3})$. It is virtually inscribable, but
    there does not exist a inscribed polygon with this normal fan,
    since all trajectories starting from the based inscribed space
    degenerate to triangles.
    \begin{figure}[h!]
        \centering
        \begin{subfigure}{.28\textwidth}
            \includegraphics[width=\textwidth]{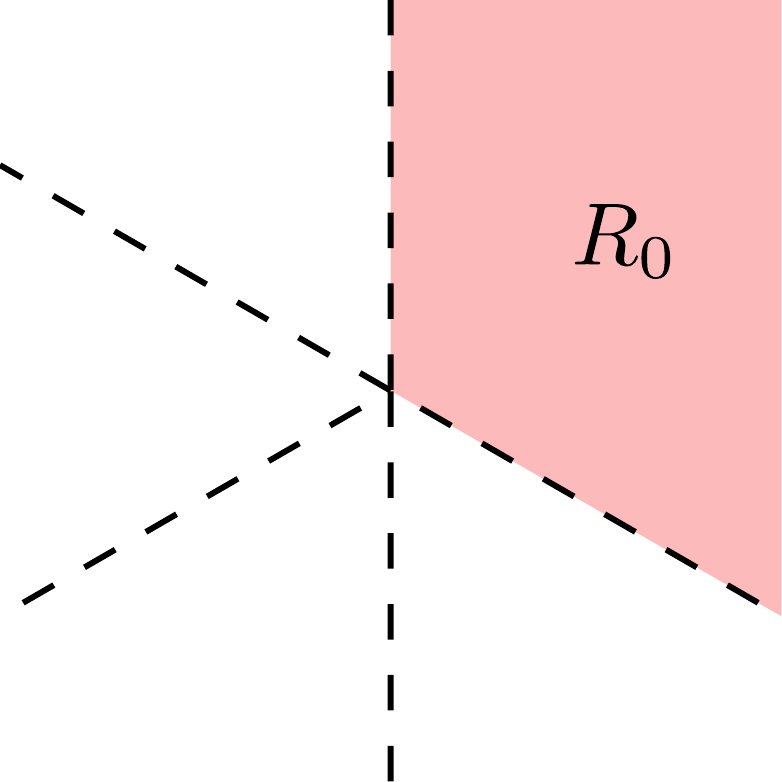}
            \caption{}
        \end{subfigure}\qquad
        \begin{subfigure}{.28\textwidth}
            \includegraphics[width=\textwidth]{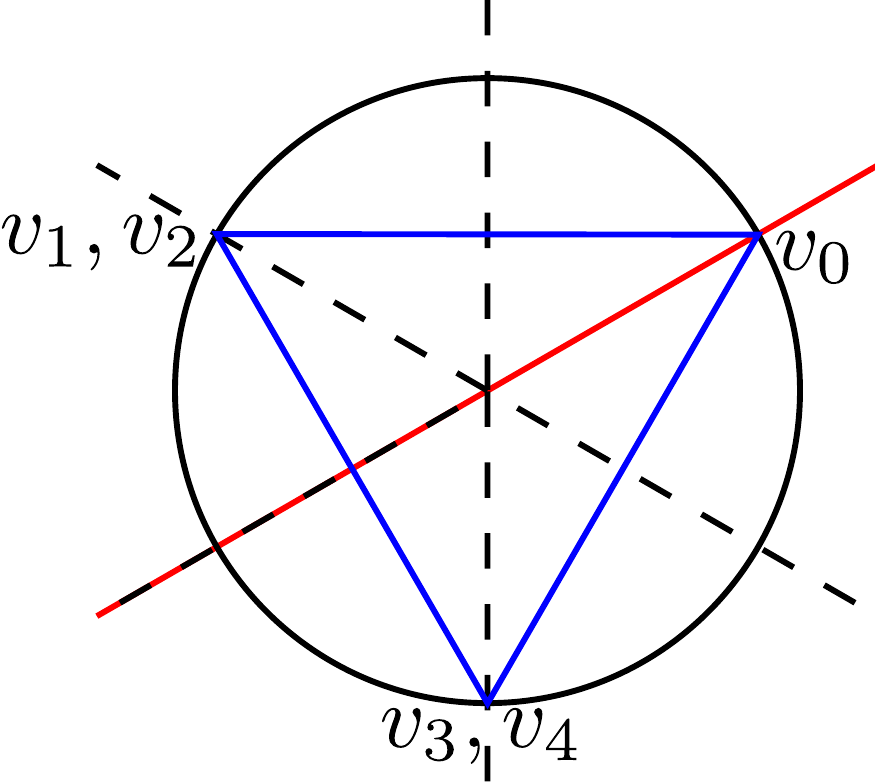}
            \caption{}
        \end{subfigure}\qquad
        \begin{subfigure}{.28\textwidth}
            \includegraphics[width=\textwidth]{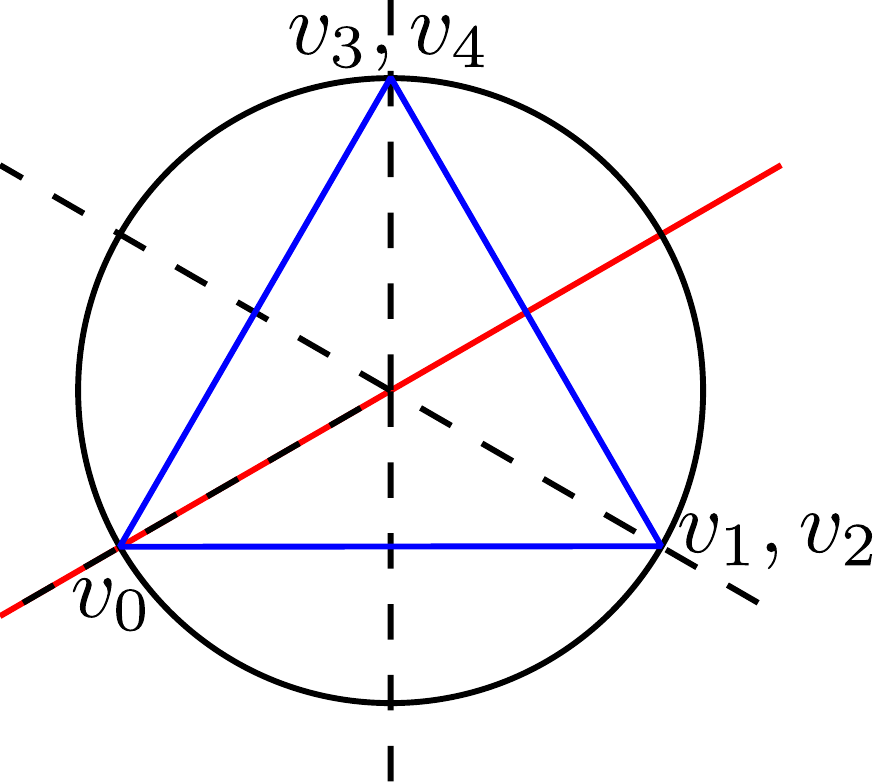}
            \caption{}
        \end{subfigure}
        \caption{(A) Virtually inscribable fan $\Fan$, which is not
          inscribable. Base region $R_0$ in highlighted. (B) and (C)
          two virtually inscribed realizations, where multiple
          vertices coincide. The $1$-dimensional inscribed space in
          red.
          \label{fig:noninsc_pentagon}
        }
    \end{figure}
\end{example}

We next determine the subset $\InProf_n \subseteq \VInProf_n$ of profiles of
inscribable fans. 

\begin{thm}\label{thm:2d_insc}
    Let $\beta \in \VInProf_n$ be a virtually inscribable profile. If $n$ is
    odd, then $\beta \in \InProf_n$ if and only if
    \[
        \beta_{j+0} - \beta_{j+1} + \cdots - \beta_{j+n-2} + \beta_{j+n-1} 
        \ > \ 0
    \]        
    for all $0 \le j < n$. If $n = 2m$ is even, then $\beta \in \InProf_n$ 
    if and only if 
    \[ 
        \sum_{i = 1}^h \beta_{2i+j} + \sum_{i = h+1}^{m-1} \beta_{2i+1+j} \
        < \ \pi
    \]    
    for all $0 \le h < m$ and $0 \le j < n$. 
\end{thm}

\begin{proof}
    Consider $n$ distinct points on the unit circle which are labelled counterclockwise
    by $v_0, v_1, \dots, v_{n-1}$.  Then $P = \conv(v_0,\dots,v_{n-1})$ is
    inscribed and clearly every inscribable fan is obtained this way. In
    particular if we label the regions of $\Fan$ by $R_i \defeq N_{v_i} P$, then
    every such choice of $n$ points yields a unique
    inscribable profile $\beta \in \InProf_n$ up to rotation.

    For $n$ ordered points $v_0,\dots,v_{n-1}$, define $\alpha_i$ to be the
    angle between $\overline{0 v_i}$ and $\overline{0v_{i+1}}$, where we set
    $v_{n} \defeq v_0$; see Figure~\ref{fig:2fan}. Then $\alpha =
    (\alpha_0,\dots,\alpha_{n-1})$ determines $v_0,\dots,v_{n-1}$ up to
    rotation. In particular $\alpha$ arises from a point configuration if and
    only if 
    \[
        \alpha_0 + \cdots + \alpha_{n-1} = 2\pi \qquad \text{ and } \qquad 
        \alpha_i > 0 \quad \text{ for } 0 \le i < n \, .
    \]
    Thus, the set of admissible $\alpha$ is $\interior(2 \pi \cdot
    \Delta_{n-1})$. By Corollary~\ref{cor:key_reflect}, each angle $\alpha_i$
    is bisected by the ray of $\Fan(P)$ corresponding to $[v_i,v_{i+1}]$ and
    therefore $2 \beta_i = \alpha_{i-1} + \alpha_i$, which gives a linear map
    $\Psi : \R^n \to \R^n$ with $\Psi(\interior(2 \pi \cdot \Delta_{n-1})) =
    \InProf_n$.

    If $n$ is odd, then $\Psi$ is bijective and $\alpha = \Psi^{-1}(\beta)$ is
    given by
    \[
        \alpha_i \ = \ \frac{1}{2}\sum_{j = 0}^{n-1} (-1)^j \beta_{i + j}
    \]
    for $0 \le i < n$.

    For the case $n = 2m$, we consider the closure $Q = \frac{1}{\pi}\operatorname{cl}(\InProf_n) = 2 \cdot \Psi(
    \Delta_{n-1} )$. The polytope $Q$ is the convex hull of the $n$ cyclic
    shifts of $(1,1,0,\dots,0)$. This polytope is known as the \emph{edge
    polytope} of the $n$-cycle; cf.~Ohsugi--Hibi~\cite{OH} and
    Villarreal~\cite{Villarreal}. Consider the undirected cycle $C_n$ with
    nodes $0,\dots,n-1$ and edges $(i-1,i)$ for $1 \le i < n$ and $(0,n-1)$.
    The vertices of $Q$ are naturally in bijection to the edges of $C_n$.
    Since $n$ is even, we may color the edges alternatingly with two colors.
    The facets, and hence the defining linear inequalities, are given by
    omitting an edge of each color. Deleting the two edges from $C_n$ leaves a
    disjoint union of two paths of length  $2r$ and $2s$, respectively. Summing
    the coordinates of the independent sets in each of the two paths of size
    $r$ and $s$ yields the given inequalities.
\end{proof}

\begin{cor}
    If $n$ is odd, then the closure of $\InProf_n$ is an $(n-1)$-dimensional
    simplex. If $n=2k$ is even, then the closure of $\InProf_n$ is a free sum
    of two $(k-1)$-simplices. In both cases, the vertices are given by the
    cyclic shifts of $(\pi,\pi,0,\dots,0)$.
\end{cor}

\begin{example}\label{ex:hexagon_profile}
    We can reprove Corollary~\ref{cor:permutahedra} for the braid
    arrangement $\Arr_2$ using Theorem~\ref{thm:2d_insc}.  $\Arr_2$
    becomes pointed when dividing by its lineality space. The
    resulting fan $\Fan$ is the normal fan of a regular hexagon and
    has the profile
    \[
        \beta \defeq \beta(\Fan) = (\tfrac{\pi}{3}, \tfrac{\pi}{3},
        \tfrac{\pi}{3}, \tfrac{\pi}{3}, \tfrac{\pi}{3}, \tfrac{\pi}{3})
    \]
    and a two-dimensional inscribed cone.   With $\Psi : \R^6 \to \R^6$ as in the proof of
    Theorem~\ref{thm:2d_insc}, we see that:
    \[
        \Psi^{-1}(\beta) \ = \ \{ \alpha \in \R^6 : \alpha_0 = \alpha_2 =
        \alpha_4, \,\alpha_1 = \alpha_3 = \alpha_5,\, \alpha_0 + \alpha_1
        = \tfrac{2\pi}{3}\}\,.
    \]
    The extreme rays of $\cInCone(\Fan)$ are given by
    extreme rays of $\Psi^{-1}(\beta) \cap \R^6_{\geq 0}$, which
    correspond to two triangles
    $\Delta = \Delta(3, 1), \nabla = \Delta(3, 2)$, see
    Figure~\ref{fig:insc_hexagon}. Therefore:
    \[
        \cInCone(\Fan) \ = \ \{\mu_1 \Delta + \mu_2 \nabla :
        \mu_1, \mu_2 \in \R_{\geq 0}\}\,.
    \]%
    \begin{figure}[h!]
        \centering
        \includegraphics[width=0.12\textwidth]{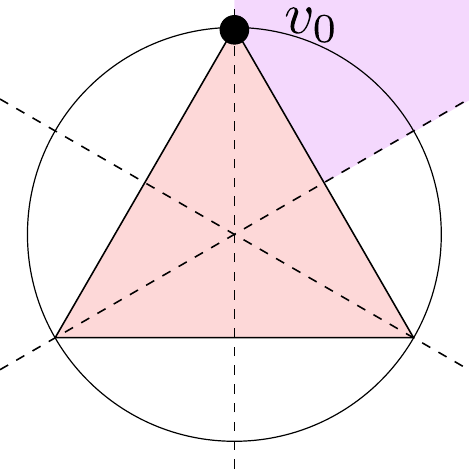}\qquad
        \includegraphics[width=0.12\textwidth]{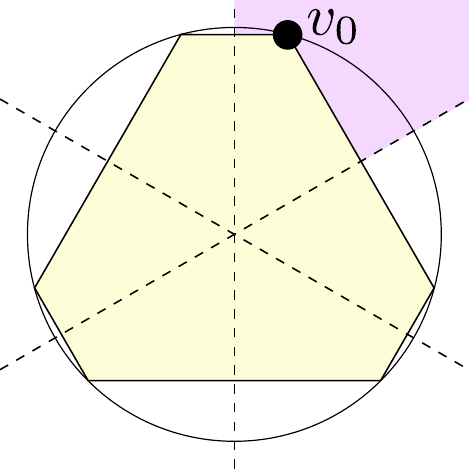}\qquad
        \includegraphics[width=0.12\textwidth]{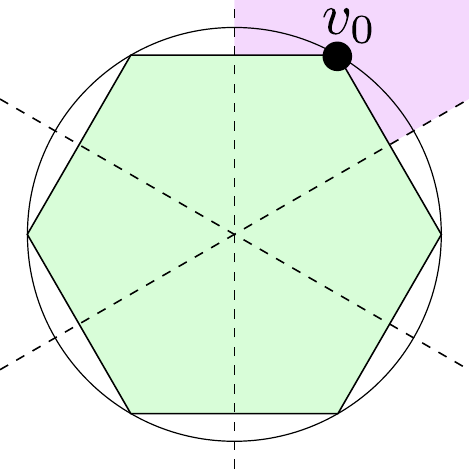}\qquad
        \includegraphics[width=0.12\textwidth]{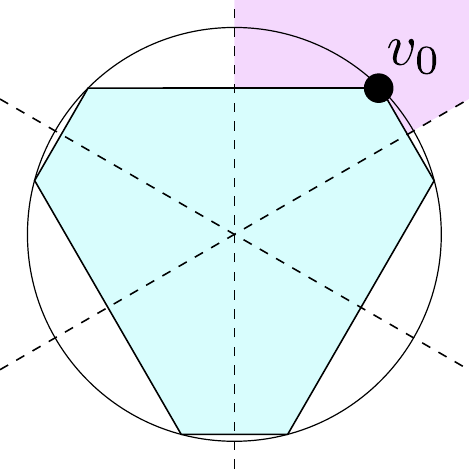}\qquad
        \includegraphics[width=0.12\textwidth]{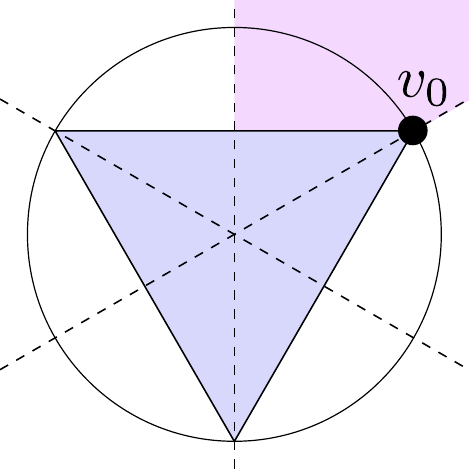}
        \caption{Multiple inscribed realizations of the normal fan of a
          regular hexagon by varying $v_0$ in $R_0$.
        \label{fig:insc_hexagon}
      }
    \end{figure}
\end{example}

Let us close with a different parametrization of $\InProf_n$. Again, let
$v_0,\dots,v_{n-1}, v_n = v_0$ be the cyclically labelled vertices of a polygon
$P$ inscribed to a circle centered at the origin. Then $P$ is determined up to
rotation by the edge length $\ell_i = \|v_{i-1} - v_i\|$ for $0 \le i < n$.
Set $\ell(P) = (\ell_0,\dots,\ell_{n-1})$. Up to positive scaling, $\ell(P)$
uniquely determines $\Fan(P)$. 

\begin{prop}
    For $n \ge 3$, we have
    \[
        \InProf_n \ \cong \ \interior(\Delta(n,2)) \, .
    \]
\end{prop}
\begin{proof}
    Let $\ell = (\ell_1,\dots,\ell_n)$ with $\ell_i > 0$ for all $i$.
    It should be clear that $\ell$ can be realized as the length vector of a
    polygon if and only if $\ell$ can be realized as the length vector of an
    inscribed polygon. Hence, we only need to determine admissible $\ell$. 
    
    As scaling $\ell$ does not change the fan, we can normalize $\ell_1 +
    \cdots + \ell_n = 2$. It follows from the triangle inequality that $\ell$
    is admissible if and only if $\ell_j  <  \sum_{i \neq j} \ell_i$,  which
    is equivalent to $\ell_j <  1$.  Thus
    \[
        \InProf_n \ \cong \
    \left \{ \ell \in \R^n : 
    \begin{array}{c} 
        0 < \ell_i < 1 \text{ for } i=1,\dots,n \\
        \ell_1 + \cdots + \ell_n = 2 
    \end{array} \right \}  \ = \ \interior(\Delta(n,2)) \, . \qedhere
    \]
\end{proof}

Note that the isomorphism in the proposition maps $\ell(P)$ to $\beta(\Fan)$,
which is a highly nonlinear map.

\section{Inscribable fans in general dimensions}\label{sec:gen_insc_fans}

Using the results of the previous section, we can give some conditions for
inscribability of fans in higher dimensions as well as for particular classes
of fans and polytopes.

\subsection{Restrictions from faces} \label{sec:restrict}
We spell out the trivial
observation that makes the connection to the previous section transparent.

\begin{prop}\label{prop:in_faces}
    If $P \subset \R^d$ is an inscribed polytope and $F \subseteq P$ is a
    face, then $F$ is inscribed.
\end{prop}

Let $\Fan$ be a fan in $\R^d$ and a $C \in \Fan$ a cone. For two cones $C,D
\subset \R^d$ we write $D-C$ for the convex cone  $\{ d-c : d \in D, c \in
C\}$. The \Defn{localization} of $\Fan$ at $C$, is the fan
\[
    \Fan_C \defeq \{ D - C : D \in \Fan, C \subseteq D \} \, .
\]
This is a fan with lineality space $\lineal(\Fan) = C - C$.

\begin{prop}
    Let $\Fan$ be a (virtually) inscribable fan and $C \in \Fan$ a cone. Then
    the localization $\Fan_C$ is (virtually) inscribable.
\end{prop}

\begin{proof}
    If $\Fan$ is inscribable and hence polytopal, then the proposition easily
    follows from Proposition~\ref{prop:in_faces}: Let $P \in \InCone(\Fan)$.
    There is a face $F \subseteq P$ such that $C = N_FP$ and it is easy to see
    that $\Fan_C$ is the normal fan of $F$. By
    Proposition~\ref{prop:in_faces}, it follows that $\Fan_C$ is inscribable. 

    For the general case, pick a region $R \in \Fan$ with $C \subseteq R$.  If
    $W \in \Fan$ is a wall with $C \subseteq W$, then $(W-C) - (W-C) = W - W$
    and hence $s_{W} = s_{W-C}$. Moreover, $G(\Fan_C)$ is a vertex-induced
    subgraph of $G(\Fan)$. This shows that $\InSpc(\Fan,R) \subseteq
    \InSpc(\Fan_C,R - C)$.
\end{proof}

Note that the lineality space of $\Fan_C$ contains $C - C$. Hence, we may
consider the intersections $\{ (D-C) \cap C^\perp : D-C \in \Fan_C \}$. This
corresponds to the projection of $\Fan_C$ to $C^\perp$, which is again a
(virtually) inscribable fan. If $C \in \Fan$ is a cone of codimension $2$,
then we can identify $\Fan_C$ with a $2$-dimensional fan. 

\begin{cor}
    Let $\Fan$ be a virtually inscribable fan. Then for all codimension-$2$
    cones $C \in \Fan$ one has $\beta(\Fan_C) \in \VInProf_n$, where $n$ is
    the number of regions containing $C$.  If $\Fan$ is inscribed, then
    $\beta(\Fan_C) \in \InProf_n$.
\end{cor}

It is non-trivial to compute the profiles of $2$-dimensional localizations of
$\Fan$ and hence the previous result is of limited applicability. It would be
very interesting if realizations of $3$\nobreakdash-dimensional fans with a fixed
combinatorics can be parametrized by their profiles around rays. To be
precise, consider a planar $3$-connected graph $G$. By Steinitz result, this
is graph of a $3$-polytope. Consider the collection of fans $\Fan$ with
$G(\Fan) = G$. Let $r_1,\dots, r_m$ be the rays of $\Fan$ and $\Fan_i \defeq
\Fan_{r_i}$ the corresponding $2$-dimensional localizations.

\begin{quest}
    What can be said about the image
    \[
        \{ (\beta(\Fan_i))_{i=1,\dots,m} : \Fan \text{ $3$-dimensional fan
        with } G(\Fan) = G \} \, ?
    \]
\end{quest}

Let $\Fan$ be a fixed pointed and polytopal fan in $\R^d$ and
$R \in \Fan$ a region. We may use the realization of $\InCone(\Fan)$ as
a subcone of $R$ as given in~\eqref{eqn:InSpc_real1}. It follows that
\[
    \InCone(\Fan,R) \ \subseteq \ \InCone(\Fan_C,R - C) +
    \lineal(\Fan_C)
\]
and the inclusion is typically strict, as
can be seen for a pyramid over a quadrilateral;
cf.~Example~\ref{ex:pyramid} below.

We record the following simple observation.

\begin{cor}\label{cor:inspc_intersect}
    Let $\Fan$ be a full-dimensional and strongly connected
    fan in $\R^d$ and $R$ a region. For fixed $1 \le k \le d$
    \[
        \InCone(\Fan,R) \ \subseteq \ \bigcap_{C}  \InCone(\Fan_C, R -
        C) + \lineal(\Fan_C) \, ,
    \]
    where the intersection is over all $k$-cones $C \subseteq R$.
\end{cor}

Recall that a $2$-dimensional polytope is \Def{even/odd} if it has an even/odd
number of vertices.

\begin{cor}\label{cor:3polytope_odd}
    Let $P$ be a normally inscribed $3$-polytope. If $P$ has an odd
    $2$-face, then $\dim \InCone(P) \le 2$. If $P$ has two adjacent odd
    $2$-faces, then $\dim \InCone(P) = 1$, that is, up to homothety $P$
    is the unique inscribed polytope with normal fan $\Fan(P)$.
\end{cor}
\begin{proof}
    Let $F_1 \subseteq P$ be an odd $2$-face and choose a vertex $v \in F_1$.
    We set $\Fan = \Fan(P)$, $R = N_vP$, and $C_1 \defeq N_{F_1}P$ the
    $1$-dimensional normal cone of $F_1$. It follows from
    Proposition~\ref{prop:2d_virtual_insc} that $\InCone(\Fan_{C_1},R-C_1)$ is
    $2$-dimensional and the first claim follows from
    Corollary~\ref{cor:inspc_intersect}.
    
    Now, let $F_2$ be an odd $2$-face sharing an edge with $F_1$ and
    set $C_2 = N_{F_2}P$. We may also assume that $v \in F_2$ and by
    the same argument as above $\InCone(\Fan_{C_2},R-C_2)$ is
    $2$-dimensional. In fact
    $\InCone(\Fan_{C_i},R - C_i) = \R_{\ge0} v + \R C_i$. Since
    $F_1 \cap F_2$ is an edge, it follows immediately that
    $\InCone(\Fan_{C_1},R - C_1) \cap \InCone(\Fan_{C_2},R - C_2)$ is
    $1$-dimensional. Appealing to Corollary~\ref{cor:inspc_intersect}
    again yields the claim.
\end{proof}

\begin{example}[Dodecahedra]
    Let $P_{12}$ be a $3$-polytope combinatorially equivalent to the
    dodecahedron. There is at most on inscribed normally equivalent
    realization of $P_{12}$ up to scaling.
\end{example}

\begin{example}[Pyramids]\label{ex:pyramid}
    Every $3$-dimensional pyramid over an inscribable polygon has a
    unique inscribed realization with fixed normal fan, up to scaling.
\end{example}

\begin{cor}\label{cor:ins_simplicial}
    Let $P$ be a simplicial polytope of dimension $d \ge 3$. Then $\dim
    \InCone(P) \le 1$.
\end{cor}
\begin{proof}
    We can reuse the argument of the proof of
    Corollary~\ref{cor:3polytope_odd}: If $F_1,F_2 \subset P$ are adjacent
    facets with $\dim \InCone(F_i) \le 1$, then
    Corollary~\ref{cor:inspc_intersect} implies the desired result.  It is
    straightforward (for example, by induction), to show that every simplex
    $S$ satisfies $\dim \InCone(S) = 1$.
\end{proof}

\subsection{Simple and even polytopes}
We now turn to the other extremal class of simple polytopes. For a set $X
\subseteq \R^d$, denote by $\Centers(X)$ the affine subspace of centers of
spheres containing $X$, that is,
\[
   \Centers(X) \defeq \{c \in \R^d : \|x - c\| = \|y - c\| \text{ for all $x, y \in X$}\}\,.
\]
We call $\Centers(X)$ the \Def{generalized bisector} of $X$, since
for two points $\x, \y \in \R^2$, $\Centers(\{\x, \y\})$ is the
well-known bisector of $\x$ and $\y$ of Euclidean geometry. Clearly
$\Centers(X) \subseteq \Centers(Y)$ for $Y \subseteq X \subseteq \R^d$.
Stronger even:
\begin{lem}\label{lem:centers_intersection}
   Let $X, Y \subseteq \R^d$. Then
   $\Centers(X \cup Y) \subseteq \Centers(X) \cap \Centers(Y)$. If
   $X \cap Y \neq \emptyset$, then
   $\Centers(X \cup Y) = \Centers(X) \cap \Centers(Y)$.
\end{lem}
\begin{proof}
   The first claim follows directly from the definition. Now let $c \in
    \Centers(X) \cap \Centers(Y)$, $x \in X$, $y \in Y$ and $z \in X \cap Y$.
    Then $\|c - x\| = \|c - z\| = \|c - y\|$ so $c \in M(X \cup Y)$.
\end{proof}
For two affine subspaces
$L, L' \subseteq \R^d$, write $L \parallel L'$ if $L = L' + t$ for
some $t \in \R^d$.

\begin{lem}\label{lem:centers_empty_or_aff_perp_parallel}
   Let $X \subseteq \R^d$. Either $\Centers(X) = \emptyset$ or
   $\Centers(X) \parallel \aff(X)^\perp$.
\end{lem}
\begin{proof}
   Let $\dim \aff(X) = k$ and let $B$ be an affine basis of $X$,
   $|B| = k+1$. Then there exists a unique sphere centered in a point
   $c_B \in \aff(X)$ that contains $B$ and $\Centers(B) = c_B + \aff(X)^\perp$. If
   $c_B \neq c_{B'}$ for two affine bases $B, B'$ of $X$, then
   $\Centers(X) = \emptyset$, otherwise $\Centers(X) = c_B +
    \aff(X)^\perp$, which is parallel to  $\aff(X)^\perp$.
\end{proof}

\newcommand\faces{\mathcal{F}}%
For a polytope $P \subset \R^d$, $P$ is inscribed if and only if
$\Centers(V(P)) \neq \emptyset$. For $F \subset P$ a face, let
\[
    \faces_k(F, P) \defeq \{G \subseteq P \text{ face} :  F \subseteq G,  \dim G = \dim F + k  \}
\]
and set
\[
   K(F, P) \defeq \bigcup_{G \in \faces_1(F, P)} V(G)\,.
\]
By projecting along $F$, one can convince oneself that
$\aff(K(F, P)) = \aff(P)$.

\begin{cor}\label{cor:K(F,P)}
   Let $F$ be a proper, non-empty face of an inscribed polytope
   $P \subset \R^d$. Then
   \[
       \Centers(K(F, P)) = \Centers(V(P))
   \]
\end{cor}
\begin{proof}
   Since $\aff K(F, P) = \aff P$, we have
   $\Centers(K(F, P)) \parallel \aff(K(F, P))^\perp =
   P^\perp$. Moreover $K(F, P) \subseteq V(P)$, so
   $\Centers(V(P)) \subseteq \Centers(K(F, P))$. But
   $\Centers(V(P)) \parallel P^\perp \parallel \Centers(K(F, P))$, so
   $\Centers(K(F, P)) = \Centers(V(P))$.
\end{proof}

The following result gives a local criterion on inscribability.

\begin{thm}\label{thm:inscribability_j_k}
   Let $P \subseteq \R^d$ be a polytope and
   $0 \leq j < j + 2 \leq k \leq d$. Then $P$ is inscribed if and
   only if the following two conditions hold: every $k$-face is
   inscribed and $\Centers(K(F, P)) \neq \emptyset$ for all $j$-faces $F$.
\end{thm}
The necessity of both conditions is exemplified by simplicial and
simple polytopes respectively, that is, every $k$-face of a simplicial
polytope is a simplex and thus inscribed, while
$\Centers(K(F, P)) \neq \emptyset$ for every simple polytope. Thus, it
is necessary to combine both.

For $j = 0$, we obtain the following remarkable characterization:
\begin{cor}\label{cor:inscribability_0_k}
    Let $P \subset \R^d$ be a $d$-polytope and $2 \le k \le d$. If
    every vertex $v$ together with its neighbors lie on some sphere
    and if all $k$-faces are inscribed to some sphere, then $P$ is
    inscribed.
\end{cor}
\begin{proof}[Proof of Theorem~\ref{thm:inscribability_j_k}]
   If $P$ is inscribed, then every face of $P$ is
   inscribed. Moreover, since $K(F, P) \subseteq V(P)$ we have
   $\emptyset \neq \Centers(V(P)) \subseteq \Centers(K(F, P))$. For
   the converse, we only need to consider the case $k = j+2$. Let $F$
   be a $j$-face and $G \supseteq F$ a $(j+1)$-face of $P$. We want
   to show that $M(K(F, P)) = M(K(G, P))$. By varying $F$ and $G$ over all
   $j$- and $(j+1)$-faces respectively, we can then conclude that
   $M(K(F, P)) = M(K(F', P))$ for any two $j$-faces $F, F'$ of $P$, so that
   $\Centers(V(P)) = K(F, P) \neq \emptyset$, i.e.\ $P$ is inscribed.

   Thus, we are left to show that $M(K(F, P)) = M(K(G, P))$. We have
   \begin{align*}
     \Centers(K(F, P))
     \ & \stackrel{\mathclap{\eqref{lem:centers_intersection}}}{=} \ \bigcap_{G \in \faces_{1}(F, P)} \Centers(V(G))
       = \bigcap_{H \in \faces_{2}(F, P)} \Centers(K(F, H))\\
     \ &\stackrel{\mathclap{\eqref{cor:K(F,P)}}}{=} \  \bigcap_{H \in \faces_{2}(F, P)} \Centers(V(H))
       \ \subseteq \ \bigcap_{H \in \faces_{1}(G, P)} \Centers(V(H))
       \ \stackrel{\mathclap{\eqref{lem:centers_intersection}}}{=} \ \Centers(K(G, P))\,.
   \end{align*}
   Since $\emptyset \neq \Centers(K(F, P)) \subseteq \Centers(K(G, P))$, we
   obtain $\Centers(K(G, P)) \parallel P^\perp \parallel \Centers(K(F, P))$ by Lemma~\ref{lem:centers_empty_or_aff_perp_parallel}. We
   conclude that $\Centers(K(F, P)) = \Centers(K(G, P))$.
\end{proof}

We stated above that the condition $\Centers(K(F, P)) \neq \emptyset$
in Theorem~\ref{thm:inscribability_j_k} is not necessary for simple
polytopes. Stronger even, a $d$-polytope is called \Defn{$k$-simple}
if every $(d-k-1)$-face is contained in exactly $k + 1$
facets~\cite[Ch.~4.5]{grunbaum}. Every polytope is at
least $1$-simple and a $(d-1)$-simple polytope is simply a simple
polytope.

\begin{cor}\label{cor:simple_inscribed}
   Let $P$ be a $k$-simple $d$-polytope, $1 \leq k \leq d-1$. Then $P$ is
   inscribed, if and only if all its $(d-k+1)$-faces are
   inscribed. In particular a simple polytope is inscribed if and
   only if all its $k$-faces are inscribed.
\end{cor}
\begin{proof}
   Let $F$ be a $(d - k - 1)$-face of $P$. Let $B_F \subseteq V(F)$
   be an affine basis off $\aff(F)$, and for each
   $G \in \faces_1(F, P)$, let $v_G \in V(G) \setminus V(F)$. Set
   $B \defeq B_F \cup \{c_G : G \in \faces_1(F, P)\}$. Then
   $|B_F| = d + 1$ and $\aff(B_F) = \aff(K(F, P)) = P^\perp$, so
   $M(B_F) \neq \emptyset$. Let $c \in M(B_F)$. Since $F$ is a facet
   of all $G \in \faces_1(F, P)$ and $G$ is inscribed, we have
   $c \in M(V(G))$ and therefore $c \in M(K(F, P))$. We can now apply
   Theorem~\ref{thm:inscribability_j_k}.
\end{proof}

\begin{rem}
   There is also a dual notion of an generalized angular
   bisector. For this, take a circumscribed polytope $P$, that is, a
   polytope all of whose facets are \emph{tangent} to a
   sphere. Assume that the sphere is centered at the origin. Then its
   polar $P^\triangle$ is inscribed. For $F \subseteq P$ a face, let
   $F^\diamond \subseteq P^\triangle$ denote its associated dual
   face. Let $\hat{M}(F)$ be the linear subspace of all points which
   are the center of a sphere tangent to all affine hulls of facets
   in $P$ containing $F$. We call $\hat{M}(F)$ the \Defn{generalized
     angular bisector} at $F$. Then
   $\hat{M}({F^\diamond}) = \Centers(V(F))$.  This generalizes a
   well-known theorem of Euclidean geometry: The angular bisectors of
   a triangle coincide with the perpendicular bisectors of the
   Gergonne triangle (the contact triangle of the inscribed circle).
\end{rem}

Let us call a polytope $P$ \Def{even} if all $2$-faces are even.  Equivalently,
$P$ is even if the graph of $P$ is bipartite. 

\begin{thm}\label{thm:inspc_even}
    Let $P \subset \R^d$ be a normally inscribed and full-dimensional
    polytope. Then the following are equivalent:
    \begin{enumerate}[\rm (i)]
        \item $P$ is even;
        \item $\dim \InCone(P) = d$;
        \item $\InSpc(P, R_0) = \R^d$ for any region $R_0$ of $\Fan(P)$.
    \end{enumerate}
\end{thm}
\begin{proof}
    Let $\Fan = \Fan(P)$ and choose a region $R_0 \in \Fan$.  The
    representation~\eqref{eqn:InSpc_intersect_rep} of $\InCone(\Fan)$ given in
    the proof of Theorem~\ref{thm:InSpc} is the intersection of open
    $d$-dimensional cones with the inscribed space $\InSpc(\Fan,R_0)$. Thus
    it suffices to show that $\InSpc(\Fan,R_0) = \R^d$ if and only if $P$
    is even.
    
    For a $2$-face $F \subseteq P$ and $C = N_FP$, let us write $\Walk_C$ for
    the unique cycle in $\Fan_C$. Now $\InSpc(\Fan,R_0)$ is the intersection
    of $\ker(t_{\Walk\Walk_C\Walk^{-1}} - \id)$, where $C$ ranges over all
    codimension-$2$ cones and $\Walk$ is a walk to a region of $\Fan_C$ and
    $\Walk^{-1}$ is the reversed walk.  It follows from
    Proposition~\ref{prop:2d_virtual_insc} that $\InSpc(\Fan_C,R_0 - C) =
    C^\perp$ precisely if $F$ is even.  Hence $\InSpc(\Fan,R_0) = \R^d$ if and
    only if $P$ is even.
\end{proof}

We call a complete fan $\Fan$ \Defn{even} if all its codimension-$2$
cones are incident to an even number of regions. Using essentially the
same argument we arrive at:

\begin{cor}
    Let $\Fan$ be a full-dimensional, strongly connected and even fan
    in $\R^d$. Then $\InSpc(\Fan, R_0) = \{0\}$ or $= \R^d$ for any
    region $R_0$ of $\Fan$. Moreover, $\Fan$ is virtually inscribable
    if and only if $\Fan_C$ is virtually inscribable for all
    $C \in \Fan$ of codimension $2$.
\end{cor}

\begin{rem}
    If $P$ is simple and even, then its combinatorial dual polytope is
    \emph{balanced}; see~\cite{joswig}. Let us call a $3$-connected
    planar graph $G$ \Def{even} if all faces are even. Interestingly,
    Dillencourt--Smith~\cite{DillencourtSmith} showed that every
    trivalent and even polytopal graph can be realized as the graph of an
    inscribed $3$-polytope. This translates to our setting as follows:
    Every trivalent and even polyhedral graph is the dual graph of an
    inscribable $3$-dimensional fan.
\end{rem}

A last observation on even fans stems from
Corollary~\ref{cor:inspc_intersect}. Given a fan $\Fan$, can we extend every
inscribed realization of a localization $\Fan_C$ ($C$ being some cone of
$\Fan$) to an inscribed realization of $\Fan$? Together with
Theorem~\ref{thm:inspc_even}, we can answer this in the affirmative for
\emph{virtual} inscribability of even fans. Let $\pi_C : \R^d \to C^\perp$ be
the orthogonal projection.

\begin{cor}
    Let $\Fan$ be a full-dimensional and strongly connected fan,
    $C \in \Fan$ a cone and $C \subseteq R \in \Fan$ a region. Then
    $\pi_C(\InSpc(\Fan, R)) \subseteq \InSpc(\Fan_C, R - C)$.
    Equality holds when $\Fan$ is inscribable and even.
\end{cor}

Clearly, $\pi_C(\InCone(\Fan, R)) \subseteq \InCone(\Fan_C, R - C)$ but
the inclusion is strict in general. The next example shows that
equality is not attained for simplicial and even fans.

\begin{example}
    Let $P \defeq \conv \{ (x, y, z) \in \Z^3 : x^2 + y^2 + z^2 = 61 \}$.
    This is an inscribed \mbox{$3$-dimensional} polytope with $72$ vertices shown
    in Figure~\ref{fig:non_surj}.
    \begin{figure}[h] 
        \centering
        \includegraphics[width=0.35\textwidth]{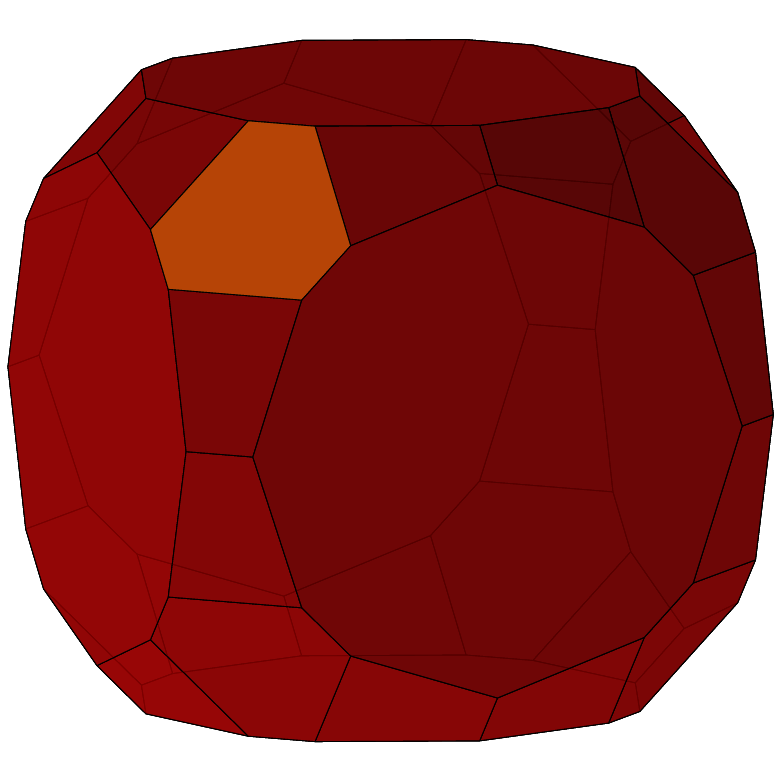}\qquad 
        \includegraphics[width=0.35\textwidth]{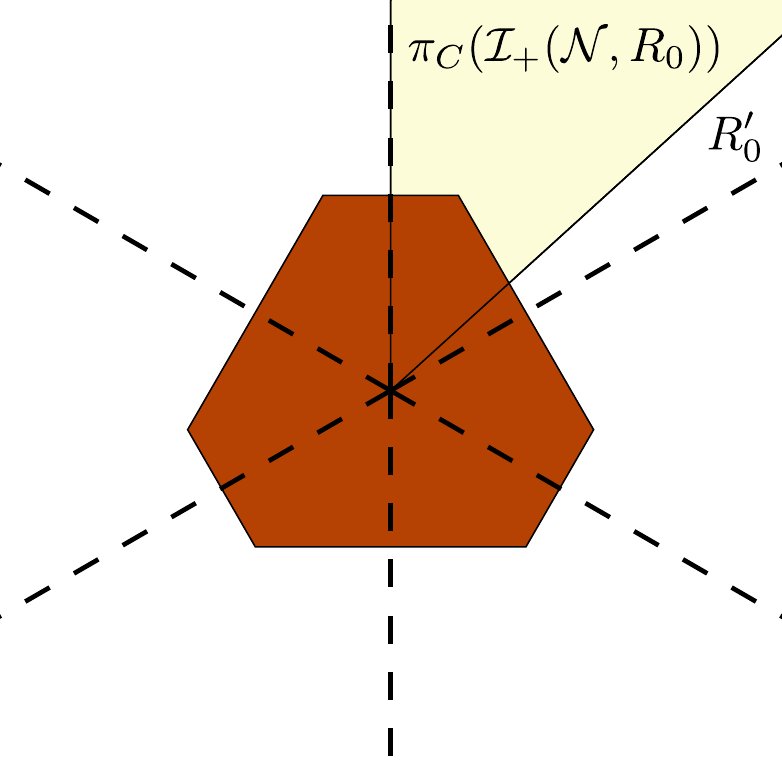}
        \caption{On the left: Inscribed simple and even polytope $P$
          with a highlighted facet $F = P^c$ for $c = (1, 1, 1)$. On
          the right: The normal fan of $F$ and the image of the
          projection $\pi_C(\InCone(P, R_0))$.}
        \label{fig:non_surj}
    \end{figure}
    It can be checked that $P$ is even and simple. Let $F \subset P$
    be the hexagonal $2$-face that maximizes the linear function
    $c = (1,1,1)$.  This is a permutahedron for the point
    $v_0 = (3, 4, 6)$. Let $R_0 \defeq N_{v_0}P$ and
    $R_0' \defeq N_{v_0}F \cap c^\perp$. The based inscribed cone of $F$ coincides
    with $R_0'$;
    cf.~Section~\ref{sec:permutahedra}. Figure~\ref{fig:non_surj}
    shows the projection $\pi_C(\InCone(P, R_0))$ as a proper
    subcone of $R_0' = \InCone(F, R_0')$.
\end{example}

This behavior is similar to that of simplicial fans and their
polytopal (virtual) realizations,
see~\cite[Lem.~8.5]{McMullen-weights}.

\subsection{Inscribed Nestohedra}\label{sec:nestohedra}
In this section, we use Corollary~\ref{cor:simple_inscribed} to give a
characterization of inscribed nestohedra, an important subclass of generalized
permutahedra.

A polytope $P \subset \R^d$ is a \Def{generalized permutahedron} if for any
two adjacent vertices $u,v \in V(P)$, there are $i \neq j$ such that $u-v =
\mu(e_i-e_j)$ for some $\mu \in \R$.  Generalized permutahedra were introduced
by Postnikov~\cite{post}, where it was also shown that the above condition is
equivalent to the existence of a polytope $Q$ such that $P+Q$ is a
permutahedron; cf.~Section~\ref{sec:permutahedra}. Generalized permutahedra
constitute a important class in the field of algebraic and geometric
combinatorics and many well-known combinatorial polytopes can be realized as
generalized permutahedra; cf.~\cite[Sect.~8]{post}. In~\cite{Coxetermatroids},
a type-A \Def{matroid polytope} is defined as a polytope whose edges are along
directions $e_i - e_j$ and whose whose vertices are equidistant from some
point. That is, type-A matroid polytopes are precisely inscribed generalized permutahedra.

Postnikov described a large class of \emph{simple} generalized permutahedra,
the so-called nestohedra. A collection $\BB$ of subsets of $[d]$ is called a
\Def{building set} if $I \cap J \neq \emptyset$ implies $I \cup J \in \BB$ for
any $I,J \in \BB$. For $I \subseteq [d]$, write $\Delta_I \defeq \conv( e_i : i
\in I)$. The generalized permutahedron associated to a building set $\BB$ 
\[
    \Delta_\BB \ \defeq \ \sum_{I \in \BB} \Delta_I
\]
is called a \Def{nestohedron}. The name derives from the fact that $\Delta_\BB$
is a simple polytope whose face lattice is anti-isomorphic to the complex of
nested sets of $\BB$. A \Def{nested set} is a subset $N \subseteq \BB$ such
that $I \cap J \neq \emptyset$ implies $I \subseteq J$ or $J \subseteq I$ for
any $I,J \in N$ and if $J_1,\dots,J_k \in N$ are disjoint, then
$\bigcup_i J_i \not\in B$. The \Def{nested complex} is the collection of
nested sets of $\BB$ ordered by inclusion.

We mentioned two important examples of nestohedra: For $d \ge 2$, let
\[
    \BBass \ \defeq \ \{ \{i,i+1,\dots,j\} : 1 \le i < j \le d \} \, .
\]
then $\Delta_{\BBass}$ is combinatorially isomorphic to the
\Def{associahedron} or Stasheff polytope~\cite{Stasheff}. 
The collection of all cyclic intervals of $[d]$
\[
    \BBcyc \ \defeq \ \BBass \cup \{ [d] \setminus \{i,\dots,j\} : 1 < i \le
    j < d \} 
\]
gives rise to the \Def{cyclohedron}~\cite{Stasheff}.

For $i,j,k \in [d]$, let us write $N_{i,j}^k(\BB)$ for the set of
$I \in \BB$ with $i,j \in I$ and $k \not\in I$ and
$n_{i,j}^k(\BB) \defeq |N_{i,j}^k(\BB)|$ for its size. Furthermore, for
$J \subseteq [d]$, define the \Defn{restriction} of $\BB$ be the
building set
\[
    \BB|_J \defeq \{I \in \BB : I \subseteq J\}\,.
\]

\begin{thm}\label{thm:ins_nestohedra}
    Let $\BB$ be a building set. The nestohedron $\Delta_\BB$ is
    inscribed if and only if for all $J \subseteq [d]$ the following
    condition holds: If for $i,j,k \in J$ both
    $n_{i,j}^k(\BB|_J) > 0$ and $n_{j,k}^i(\BB|_J) > 0$ holds, then
    \[
        n_{i,j}^k(\BB|_J) \  = \
        n_{j,k}^i(\BB|_J) \  = \
        n_{i,k}^j(\BB|_J) \, .
    \] 
\end{thm}
\begin{proof}
    Since $\Delta_\BB$ is simple, Corollary~\ref{cor:simple_inscribed} yields
    that it suffices to check that all $2$-faces of $\Delta_\BB$ are
    inscribed. For $c \in \R^d$, we have
    \[
        \Delta_\BB^c \ = \ \sum_{I \in \BB} \Delta_I^c   \, .
    \]
    Hence, for every face $F \subseteq \Delta_\BB$ of dimension $2$,
    there is a $c$ with $d-2$ distinct coordinates such that
    $\Delta_\BB^c = F$. If $c_i = c_j < c_k = c_l$ for some
    $i, j, k, l$, then $F$ is the Cartesian product of two segments
    and hence inscribable. The remaining case is if $c_i = c_j = c_k$
    for some distinct $i,j,k$. Let
    $J \defeq \{l \in [d] : c_l \leq c_i\}$. If
    $I \in \BB \setminus \BB|_J$, then $\Delta_I^c$ is a vertex. For
    this reason, the face $F$ is a translate of
    \[
        n_{i,j}^k(\BB|_J) \Delta_{\{i,j\}}  + n_{j,k}^i(\BB|_J) \Delta_{\{j,k\}}
        + n_{i,k}^j(\BB|_J) \Delta_{\{i,k\}}  + m \Delta_{\{i,j,k\}}
    \]
    for some $m \ge 0$. The normal fan of $F$ is determined by how
    many numbers of $n_{i,j}^k(\BB|_J)$, $n_{i,k}^j(\BB|_J)$,
    $n_{j,k}^i(\BB|_J)$ are greater than zero. If zero of them are,
    then $F$ is a translate of $m\Delta_{\{i,j,k\}}$ and thus
    inscribable. If there is only one, then $F$ is a translate of an
    isosceles triangle and inscribable again. If there are two, then
    $F$ is a rhombus or a pentagon with normal fan as in
    Example~\ref{ex:noninsc_pentagon} and not inscribable. Finally, if
    there are three, then $F$ is a hexagon with normal fan as in
    Example~\ref{ex:hexagon_profile}, which is inscribable if and only
    if all three numbers are equal.
\end{proof}

\begin{example}[Pitman--Stanley polytopes]
    The $(d-1)$-dimensional \Def{Pitman--Stanley polytope} is the nestohedron
    of the building set
    \[
        \BBflag \ \defeq \ \{ \{1,\dots,k\} : 1 \le k \le d \} \, ,
    \]
    see~\cite[Sect.~8.5]{post}. It is combinatorially but not
    linearly isomorphic to the $(d-1)$-dimensional cube. By
    Theorem~\ref{thm:ins_nestohedra}, it is an inscribed nestohedron.
\end{example}

A particularly nice subclass of nestohedra is given by the graph associahedra.
Let $G = (V,E)$ be a simple graph on nodes $V = [d]$. The \Def{graphical
building set} $\BB(G)$ consists of all non-empty $I \subseteq [d]$ such that
the vertex-induced subgraph $G[I]$ is connected. It is easy to see that
$\BB(G)$ is indeed a building set. If $G$ is a path or a cycle, then
$\Delta_{\BB(G)}$ is the associahedron and the cyclohedron, respectively. If
$G$ is the complete graph, then $\Delta_{\BB(G)}$ is the permutahedron.

\begin{cor}
    Let $G$ be a graph. Then $\Delta_{\BB(G)}$ is inscribed if and only if $G$
    is a disjoint union of complete graphs.
\end{cor}
\begin{proof}
    If $G$ is disconnected with connected components $G_1,\dots,G_k$, then 
    \[
        \Delta_{\BB(G)} \ = \ \Delta_{\BB(G_1)} \times \Delta_{\BB(G_2)}
        \times \cdots \times \Delta_{\BB(G_k)} \, .
    \]
    So it suffices to restrict to connected $G$ and $|V| = d \ge 2$. 
    
    Assume that $G$ is not complete. We can find nodes $i,j,k$ such
    that $ij, jk \in E$ and $ik \not\in E$. Let
    $J \defeq \{i, j, k\}$. Then $n_{i,j}^k(\BB|_J) = 1$,
    $n_{j,k}^i(\BB|_J) = 1$, but $n_{i,k}^j(\BB|_J) = 0$. It follows from
    Theorem~\ref{thm:ins_nestohedra} that $\Delta_{\BB(G)}$ is not inscribed.
\end{proof}

The corollary implies that neither the associahedron nor the cyclohedron (in
their realization as a generalized permutahedron) is inscribed.

For $I, J \subseteq [d]$, let $I \triangle J \defeq (I \setminus J) \cup (J
\setminus I)$ be their \Defn{symmetric difference}. 
A building set is called
\Defn{$\triangle$-closed}, if for all $I, J \in \BB$ with $I \not\subseteq J$
and $J \not\subseteq I$
\[
    I \cap J \neq \emptyset \ \Longrightarrow \ 
    I \triangle J \in \BB \, .
\]
For example, both the building sets of Pitman--Stanley polytopes, as well as
the graphical building sets of complete graphs are $\triangle$-closed.
\begin{prop}
    If $\BB$ is a $\triangle$-closed building set, then $\Delta_\BB$ is
    inscribed.
\end{prop}
\begin{proof}
    We want to check the condition of Theorem~\ref{thm:ins_nestohedra}.
    Clearly, if $\BB$ is triangle closed, then also $\BB|_J$ for any $J
    \subseteq [d]$, so we only need to show the condition for $J = [d]$.  Let
    $i, j, k \in \BB$ and both $n_{i,j}^k(\BB) > 0$ and $n_{j,k}^i(\BB) > 0$.
    Furthermore, let $K \in N_{i,j}^k(\BB)$. Then for all $J \in
    N_{i,k}^j(\BB)$, we see that $J \cap K \not\in \{\emptyset, J, K\}$ and thus
    $J \triangle K \in N_{j,k}^i(\BB)$. Therefore, the map
    \[
        N_{i,k}^j(\BB) \to N_{j,k}^i(\BB)\,, \qquad J \mapsto J \triangle K\,,
    \]
    is a bijection. By a symmetric argument, using some
    $I \in N_{j,k}^i(\BB)$, the map $J \mapsto J \triangle I$ gives a
    bijection between $N_{i,k}^j(\BB)$ and $N_{i,j}^k(\BB)$, so all
    three set have equal size.
\end{proof}

While this gives a quite general class of inscribed nestohedra, not
all examples arise this way. 

\begin{example}
    Consider the following building set for $d = 4$:
    \[
        \BB \ \defeq \ \{ \{1, 2, 3\},\;\{1, 2, 4\},\;\{1, 3, 4\},\;\{2,
        3, 4\},\;\{1, 2, 3, 4\}\}\,.
    \]
    Using Theorem~\ref{thm:ins_nestohedra} one can check that $\Delta_\BB$ is
    inscribed, but $\BB$ is not $\triangle$-closed.
\end{example}

\section{Type cones and inscribed virtual polytopes}\label{sec:type}

In this section, we develop a different perspective on the inscribed cone
$\InCone(\Fan)$ that is invariant under orthogonal transformations of $\Fan$
and that allows us to give a polynomial time algorithm to check if a rational
polytope $P$ is normally inscribable (Theorem~\ref{thm:algo_rational}).

The collection $\TypeCone(P)$ of polytopes normally equivalent to $P$ modulo
translation is an open polyhedral cone~\cite{McMullen-rep} and we realize
$\InCone(P)$ as a subcone. We recall in Section~\ref{sec:PL} that polytopes
normally equivalent to $P$ can be seen as the cone of strictly-convex
piecewise-linear functions supported on $\Fan = \Fan(P)$. This way, the
Grothendieck group associated to the monoid $(\TypeCone(P),+)$ is the 
space $\PL(\Fan)$ of general piecewise-linear functions on $\Fan$
modulo linear functions. We will
realize the inscribed space $\InSpc(P)$ as a subspace of $\PL(\Fan) /
(\R^d)^*$. This gives an interpretation of $\InSpc(\Fan,R_0) \setminus
\cInCone(\Fan,R_0)$ as inscribed \emph{virtual} polytopes
(Section~\ref{sec:virtual}).  Interestingly, a fan $\Fan$ may have virtual
inscribed polytopes but no \emph{actual} inscribed polytopes. This parallels
the situation of fans not admitting a strictly convex PL-functions. Explicit
descriptions in terms of inequalities and equations are given in
Section~\ref{sec:computing}.

\subsection{Type cones and piecewise-linear functions}\label{sec:PL}
The type cone $\TypeCone(P)$ of a polytope $P \subset \R^d$ is the collection
of polytopes $P' \subset \R^d$ normally equivalent to $P$, up to translation.
In this section we recall the well-known result that $\TypeCone(P)$ is an open
polyhedral cone.  In order to understand the polytopes in the closure
$\cTypeCone(P)$, let us call a polytope $Q \subset \R^d$ a \Def{weak Minkowski
summand} of $P$ if there are $\mu > 0$ and a polytope $R$ such that 
\begin{equation}\label{eqn:wMS}
    \mu P \ = \  Q + R \, .
\end{equation}
In particular, $Q$ is normally equivalent to $P$ if, in addition, $P$ is a weak
Minkowski summand of $Q$. We shall see that $\cTypeCone(P)$ is precisely the
collection of weak Minkowski summands of $P$, up to translation.

To understand the boundary structure of $\cTypeCone(P)$, let $\Fan, \Fan'$ be
fans with the same support. We say that $\Fan'$ \Def{coarsens} $\Fan$ (or that
$\Fan$ \Def{refines} $\Fan'$) if for
every region $C \in \Fan$ there is $C' \in \Fan'$ with $C \subseteq C'$.
`Coarsening' defines a partial order on complete fans in $\R^d$ with minimum
$\{\R^d\}$.  It follows from~\eqref{eqn:wMS} that $\Fan(Q)$ coarsens
$\Fan(P)$ whenever $Q$ is a weak Minkowski summand of $P$.  In the
introduction, we called a fan $\Fan$ \emph{polytopal} if
$\Fan = \Fan(P)$ for some polytope $P$ and we write $\cTypeCone(\Fan) \defeq
\cTypeCone(P)$. Conversely, Theorem.~15.1.2 of~\cite{grunbaum} implies that
every polytopal coarsening of $\Fan(P)$ is the normal fan of a weak Minkowski
summand of $P$. The following captures the most important structural
properties of type cones.

\begin{thm}[{\cite[Theorem~7]{McMullen-rep}}]\label{thm:typecones}
    Let $\Fan$ be a complete fan. Then $\cTypeCone(\Fan)$ is a polyhedral
    cone. The set of non-empty faces of $\cTypeCone(P)$ ordered by inclusion
    is isomorphic to the poset of polytopal coarsenings of $\Fan$. If $F
    \subseteq \cTypeCone(\Fan)$ is a non-empty face corresponding to $\Fan'$,
    then $F = \cTypeCone(\Fan')$.
\end{thm}

A polytope $P$ is \Def{indecomposable} if $\cTypeCone(P)$ is $1$-dimensional,
that is, if every weak Minkowski summand of $P$ is homothetic to $P$.
Theorem~\ref{thm:typecones} implies that the rays of $\cTypeCone(P)$
correspond to the indecomposable weak Minkowski summands of $P$.

It is advantageous to view Theorem~\ref{thm:typecones} from the perspective of
piecewise-linear functions. In fact, the results leading to
Corollary~\ref{cor:TypeCone_desc} essentially give a complete proof of
Theorem~\ref{thm:typecones}. Let $\Fan$ be a full-dimensional and strongly
connected fan in $\R^d$. A continuous function $\ell : |\Fan| \to \R$ is a
\Def{piecewise-linear} (PL) function supported on $\Fan$ if for every region
$R \in \Fan$, the restriction $\ell_R$ is given by a linear function. It is
clear that the collection $\PL(\Fan)$ of all piecewise-linear functions
supported on $\Fan$ has the structure of an $\R$-vector space. 

The \Def{support function} of a polytope $P \subset \R^d$ is the function
$h_P : \R^d \to \R$ 
\[
    h_P(c) \ \defeq \ \max\{ \inner{c,x} : x \in P \} \ = \
     \max\{ \inner{c,v} : v \in V(P) \} \, .
\]
This is a convex PL-function supported on $\Fan(P)$  that uniquely determines
$P$.
A convex PL-function $\ell$ is \Def{strictly} convex with respect to $\Fan$ if 
\[
    \ell(x+y) \ < \  \ell(x) + \ell(y)
\]
for any points $x,y \in \R^d$ not contained in the same region.

\begin{prop}\label{prop:PL}
    Let $\Fan$ be a polytopal fan and $\ell$ a PL-function supported
    on $\Fan$. The following are equivalent:
    \begin{enumerate}[\rm (i)]
        \item $\ell$ is strictly convex;
        \item For all $x \in \R^d$ it holds that 
            $ \ell(x)  =  \max \{ \ell_R(x) : R \in \Fan \text{ region} \} $;
        \item $\ell = h_P$ for some polytope $P$ with $\Fan(P) = \Fan$.
    \end{enumerate}
\end{prop}
\begin{proof}
    (i) $\Rightarrow$ (ii): Let $R \in \Fan$ be a region and $r \in
    \interior(R)$. For any other region $S \in \Fan$, let $s \in S$
    such that $r+s \in \interior(S)$. Then
    \[
        \ell_S(r) + \ell_S(s) \ = \ \ell_S(r+s) \ = \ \ell(r+s) \ < \ \ell(r) +
        \ell(s) \ = \ \ell_R(r) + \ell_S(s) \, .
    \]
    This shows $\ell_R(r) \ge \ell_S(r)$ for all $r \in R$ and all regions
    $S$ and hence (ii) holds.

    (ii) $\Rightarrow$ (iii): For $R \in \Fan$, let $v_R \in \R^d$ such that
    $\ell_R(x) = \inner{v_R,x}$ for all $x \in R$ and define $P = \conv(v_R :
    R \in \Fan)$. Convexity now forces $h_P(x) = \max\{ \inner{v_R,x} : R\} =
    \ell(x)$ for all $x$.

    (iii) $\Rightarrow$ (i): Let $r \in \interior(R)$ and $s \in
    \interior(S)$ for two distinct regions $R,S \in \Fan = \Fan(P)$.
    Let $v_R = P^r$ and $v_S = P^s$ be the corresponding vertices. That is,
    $h_P(r) = \inner{r,v_R} > \inner{r,x}$ for all $x \in P \setminus \{v_R\}$
    and likewise for $s$ and $v_S$. For $r+s$ there is a point $x \in P$ such
    that 
    \[
        h_P(r+s) \ = \ \inner{r+s,x} \ = \ 
        \inner{r,x} + \inner{s,x} \ < \ 
        \inner{r,v_R} + \inner{s,v_S} \ = \ h_P(r) + h_P(s) \, .
        \qedhere
    \]
\end{proof}

The domains of linearity of a convex PL-function $\ell \in \PL(\Fan)$
are convex and yield a coarsening of $\Fan$. The proof now implies
that any convex PL-function $\ell$ determines a polytope
$Q \ \defeq \ \{ x : \inner{c,x} \le \ell(c) \text{ for all } c \}$
whose normal fan $\Fan(Q)$ coarsens $\Fan$. Hence, if $\Fan = \Fan(P)$
is polytopal, then convex PL-functions supported on $\Fan$ are in
bijection to weak Minkowski summands of $P$.

Recall that the dual graph of $\Fan$ is the undirected graph $G(\Fan) = (V,E)$
whose nodes $V(\Fan)$ are the regions of $\Fan$ and regions $R,S$ satisfy $RS
\in E(\Fan)$ if $\dim R \cap S = d-1$. If $R,S$ are adjacent, then let
$\alpha_{RS}$ be the unit vector such that $\lin(R\cap S) = \alpha_{RS}^\perp$
and $\inner{\alpha_{RS},x} \le 0$ for all $x \in R$. That is, $\alpha_{RS}$ is
an \emph{outer} normal vector to $R$ and  note that $\alpha_{SR} =
-\alpha_{RS}$. Let $\ell_R : \R^n \to \R$ be a linear function for each region
$R \in \Fan$. The collection $(\ell_R)_R$ determines a PL function supported
on $\Fan$ if $\ell_R(x) = \ell_S(x)$ for all $x \in R \cap S$. If $R$ and $S$
are adjacent, then this is equivalent to 
\begin{equation}\label{eqn:PL_adjacent}
    \ell_S(x) - \ell_R(x)  \ = \ \lambda_{RS}\inner{\alpha_{RS},x}
\end{equation}
for some $\lambda_{RS} \in \R$. Clearly, $\lambda_{SR} = \lambda_{RS}$ and
thus we can regard $\lambda$ as a function $ E(\Fan) \to \R$.

Let $R_0$ be a fixed base region with $\ell_0 \defeq \ell_{R_0}$. For any region
$R$ let $R_0 R_1\dots R_k = R$ be a walk in $G(\Fan)$. Then
\begin{equation}\label{eqn:PL-R0}
    \ell_R(x) \ = \ \ell_{0}(x) + \sum_{i=1}^k
    \lambda_{R_{i-1}R_i} \inner{\alpha_{R_{i-1}R_i},x} \, .
\end{equation}
In particular, this means that for any closed walk ($R_k = R$), we
have
\begin{equation}\label{eqn:closed_walk}
    \sum_{i=1}^k \lambda_{R_{i-1}R_{i}}  \alpha_{R_{i-1}R_{i}} \ = \ 0 \, .
\end{equation}
A map $\lambda : E(\Fan) \to \R$ satisfying \eqref{eqn:closed_walk} for all
closed walks is called a \Def{$1$-weight} on $\Fan$;
cf.~\cite{McMullen-weights}. As we do not consider higher weights in this
paper, we simply call $\lambda$ a \Def{weight}. The argument above now shows
the following.

\begin{prop}\label{prop:PL_as_pairs}
    Let $\Fan$ be a full-dimensional and strongly connected fan in $\R^d$. For
    a fixed region $R_0$ we have that $\PL(\Fan)$ is isomorphic to the
    collection of pairs $(\ell_0,\lambda) \in (\R^d)^* \times \R^{E(\Fan)}$,
    where $\ell_0$ is a linear function and $\lambda = (\lambda_{RS})_{RS \in
    E(\Fan)}$ is a weight.
    Changing the base region only affects $\ell_0$ and leaves $\lambda$
    invariant.
\end{prop}
For a PL-function $\ell \in \PL(\Fan)$ we will write $\lambda(\ell)$ for the
corresponding element in $\R^{E(\Fan)}$ in this isomorphism. 
We next give a well-known description of strictly convex PL-functions in
terms of the weights $\lambda \in \R^{E(\Fan)}$.

\begin{thm}\label{thm:PL-lambda}
    Let $\Fan$ be a complete fan in $\R^d$ and let $\ell \in \PL(\Fan)$.
    Then $\ell$ is strictly convex if and only if $\lambda(\ell)_{RS} > 0$ for
    all adjacent regions $R,S \in \Fan$.
\end{thm}
\begin{proof}
    Let $\ell$ be strictly convex and $R,S$ adjacent regions. Let $x \in
    \interior(R)$ and $y \in \interior(S)$ such that $x+y \in R$. It follows
    from~\eqref{eqn:PL_adjacent} that 
    \[
        \ell(x+y) \ = \ \ell_R(x+y) < \ell_R(x) + \ell_S(y) \ = \
        \ell_R(x) + \ell_R(y)  + \lambda_{RS} \inner{\alpha_{RS},y} \, .
    \]
    Since $S$ and $R$ are separated by $\alpha_{RS}^\perp$ and 
    $y \not\in R$, we have $\inner{\alpha_{RS},y} > 0$ and hence
    $\lambda_{RS} > 0$.

    For the converse, note that by Proposition~\ref{prop:PL}(ii) it is
    sufficient to show that for $x \in \interior(R_0)$, $\ell_{R_0}(x) >
    \ell_S(x)$ for all regions $S \neq R_0$. As $R_0$ is arbitrary, this shows
    it for all regions. Let $x \in \interior(R_0)$ be generic, that is,
    $\inner{\alpha_{RS},x} \neq 0$ for all $RS \in E(\Fan)$. Any such point
    induces an acyclic orientation on $G(\Fan)$ as follows: For two adjacent
    regions $R,S$, we orient the edge from $R$ to $S$ if
    $\inner{\alpha_{RS},x} < 0$. Note that $R_0$ is the unique source and for
    every region $R$ there is a directed path $R_0R_1\dots R_k = R$. We
    compute
    \[
        \ell_{R_0}(x) - \ell_{R}(x) \ = \ -\sum_{j=1}^k
        \underbrace{\lambda_{R_{i-1}R_i}}_{> 0}
        \underbrace{\inner{\alpha_{R_{i-1}R_i},x}}_{<0} \ > \ 0 \, . \qedhere
    \]
\end{proof}

The result implies that for a convex PL-function $\ell$ supported on $\Fan$,
we can read off the induced coarsening of $\Fan$.

\begin{cor}
    Let $\ell$ be a convex PL-function supported on $\Fan$ and let
    $\lambda = \lambda(\ell)$.  Denote the connected components of the
    graph $(V(\Fan), \{ RS \in E(\Fan) : \lambda_{RS} = 0 \})$ by
    $V_1,\dots,V_s \subseteq V(\Fan)$.  The coarsening $\Fan'$ of
    $\Fan$ induced by $\ell$ has regions $\bigcup V_i$ for
    $i=1,\dots,s$.
\end{cor}

As we are only interested in polytopes $P$ with normal fan $\Fan$ up to
translation, we get the following.
\begin{cor}\label{cor:TypeCone_desc}
    Let $\Fan$ be a complete fan with dual graph $G(\Fan) = (V,E)$. Then 
    \[
        \cTypeCone(\Fan)  \ \cong \ \{ \lambda \in \R^E_{\ge 0} : \lambda
        \text{ satisfies~\eqref{eqn:closed_walk} for all closed walks}  \} \,
        .
    \]
\end{cor}

Note that~\eqref{eqn:closed_walk} has to hold for only finitely many closed
walks. We give two particular choices for these finitely many walks.  Let $T$
be a spanning tree of $G(\Fan)$. Every edge $RS \in E(\Fan) \setminus E(T)$
closes a cycle $C_{RS}$ in $T \cup \{RS\}$, called the \emph{fundamental
cycle} with respect to $T$. The following is standard and follows from the
fact that both collections of cycles give a basis for the cycle space of
$G(\Fan)$.

\begin{prop}\label{prop:cycle_basis}
    Let $\Fan$ be a full-dimensional and strongly connected fan and $\lambda
    \in \R^E$. Then $\lambda$ satisfies~\eqref{eqn:closed_walk} for all closed
    walks if $\lambda$ satisfies~\eqref{eqn:closed_walk}  for
    \begin{enumerate}[\rm (i)]
        \item all fundamental cycles with respect to a fixed spanning tree
            $T$, or
        \item the cycles of links of codimension-$2$ cones not in the boundary
            of $|\Fan|$.
    \end{enumerate}
\end{prop}

\subsection{Inscribed Minkowski summands}\label{sec:relative}
Let $\Fan$ be a polytopal fan.  Theorem~\ref{thm:typecones} states that the
facial structure of the polyhedral cone $\cTypeCone(\Fan)$ is given in terms
of polytopal coarsenings of $\Fan$, partially ordered by refinement.  Let
$\cInCone(\Fan) \subseteq \cTypeCone(\Fan)$ be the closure of $\InCone(\Fan)$ in
the Hausdorff metric. The following example illustrates that the geometry
and the combinatorics of $\cInCone(\Fan)$ is more subtle and deserves further
study.

\begin{example} \label{ex:relatively_inscribable}
    Let $\Fan$ be the normal fan of the regular hexagon $P$. We already
    computed the inscribed cone of $\Fan$ using its profile in
    Example~\ref{ex:hexagon_profile}, but we now want to illustrate how the
    inscribed cone is embedded into the type cone.
    Corollary~\ref{cor:TypeCone_desc} allows us to embed $\cTypeCone(\Fan)$ as
    a $4$-dimensional subcone of $\R^6$. The hyperplane $\lambda_1 + \lambda_2
    + \dots + \lambda_6 = 1$ meets all rays of $\cTypeCone(\Fan)$ and the
    intersection yields a $3$-dimensional polytope $T$, depicted in
    Figure~\ref{fig:hexagon_type_cone}.  The faces of $T$ are in one-to-one
    correspondence with the polytopal coarsenings of $\Fan$. The vertices of
    $T$ correspond to the indecomposable summands of $P$. North- and south
    pole correspond to the triangles $\triangle$ and $\nabla$. The
    three vertices on the equator of $T$ are segments along the three distinct
    edge directions of $P$.

    \begin{figure}[h] 
        \centering
        \includegraphics[width=0.45\textwidth]{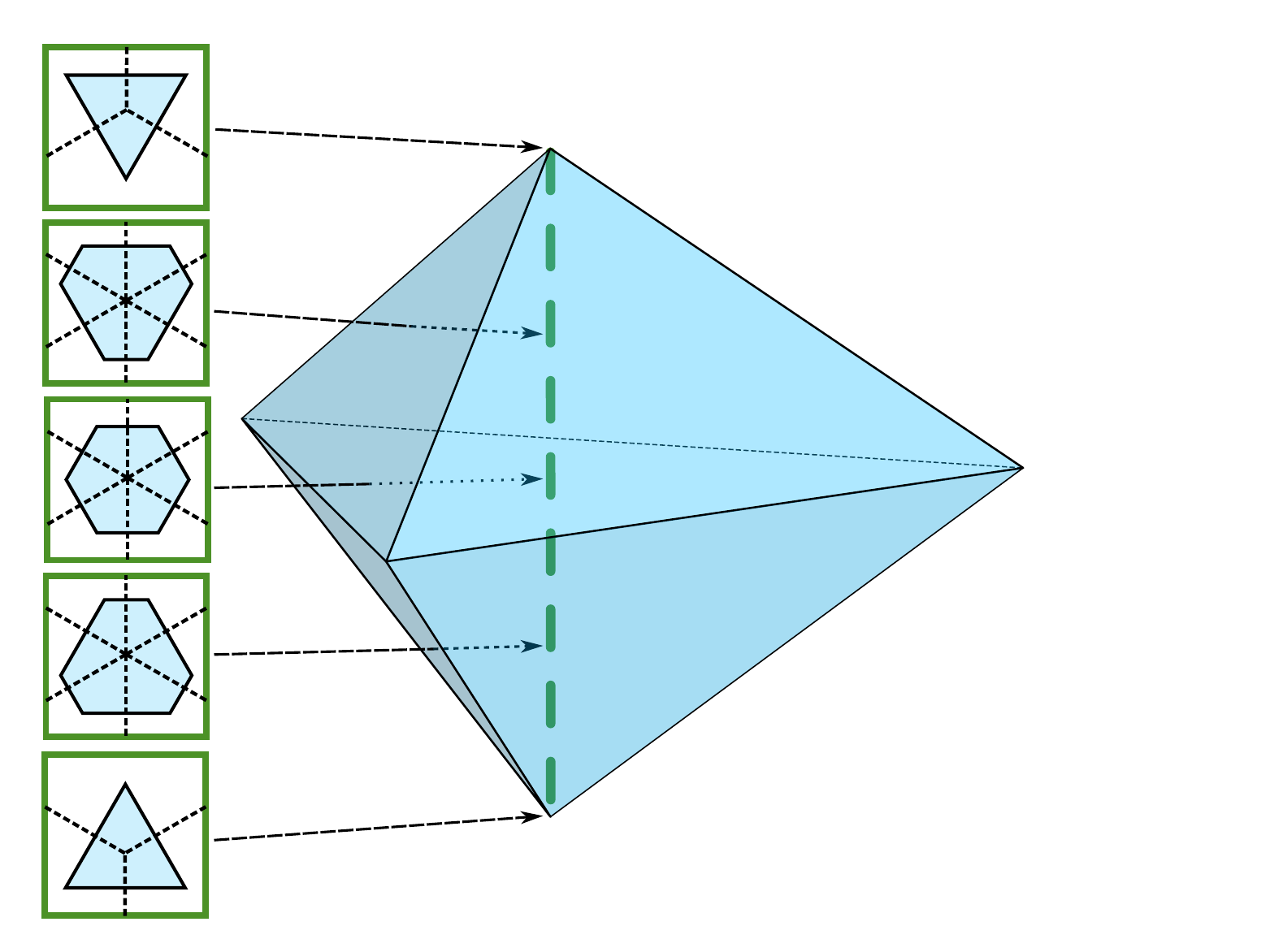}
        \includegraphics[width=0.45\textwidth]{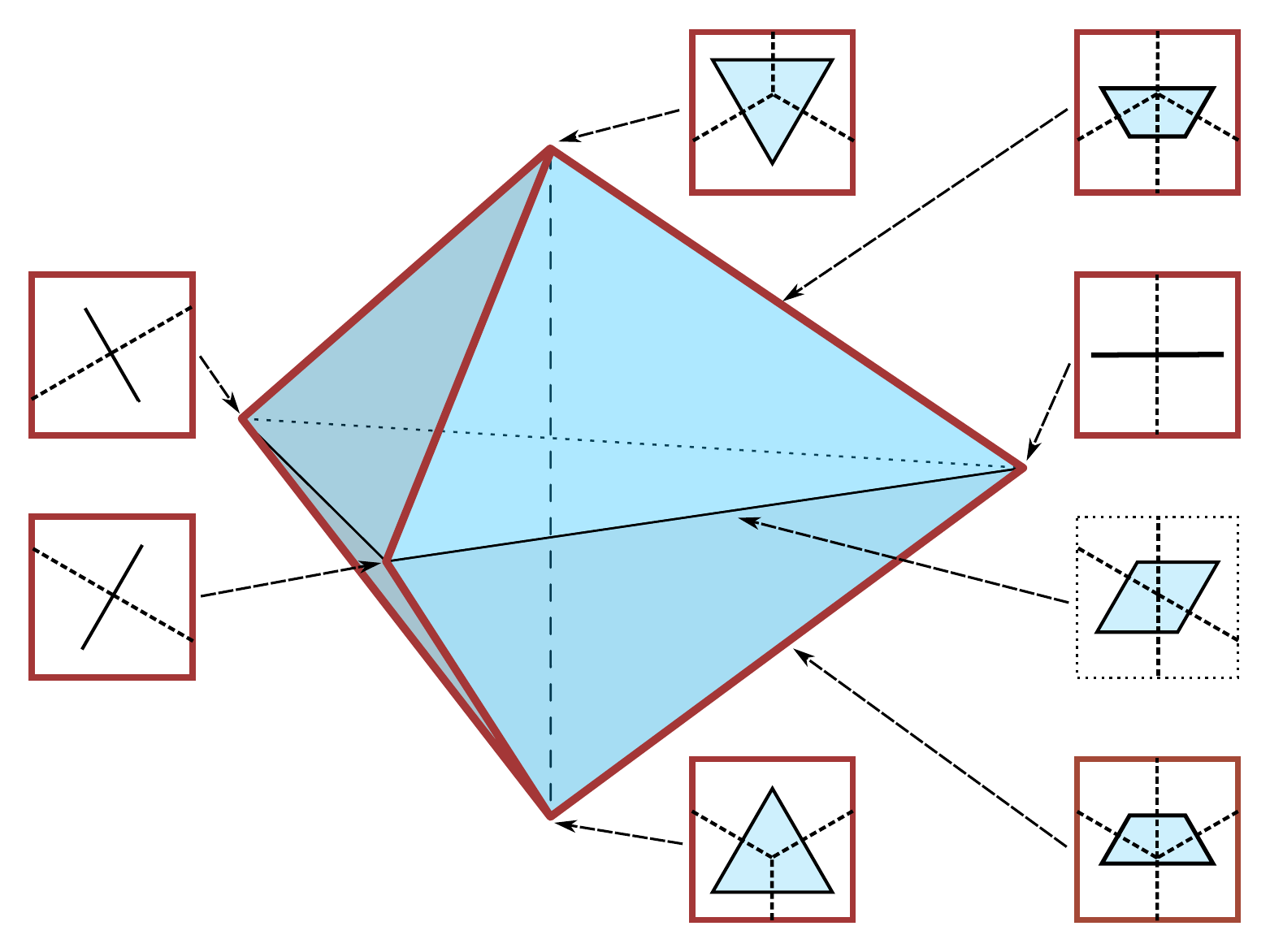}
        \caption{The type cone of a regular hexagon (intersected with
          a hyperplane) is shown on the left. The inscribed cone is
          the vertical green segment. Inscribed weak Minkowski summands are
          highlighted in red on the right. Most of them are not relatively
          inscribed.}
        \label{fig:hexagon_type_cone}
    \end{figure}

    All indecomposable summands are inscribed to some sphere but, as we have
    seen in Example~\ref{ex:hexagon_profile}, the inscribed cone
    $\InCone(\Fan)$ is spanned by two rays corresponding to $\triangle$ and
    $\nabla$.  In $T$, $\cInCone(\Fan)$ is thus given by the segment
    connecting the two polytopes; see Figure~\ref{fig:hexagon_type_cone} left.
    The other decomposable summands of $P$ are rhombi, isosceles trapezoids,
    and certain pentagons. We already worked through all weak Minkowski
    summands of $P$ in the Examples~\ref{ex:insc_quadrangle}
    and~\ref{ex:noninsc_pentagon} and saw that only the isosceles trapezoids
    are normally inscribable. The extraneous inscribed examples are marked in
    red in the right of Figure~\ref{fig:hexagon_type_cone}.
\end{example}

The example demonstrates that not every inscribable coarsening $\Fan'$
of $\Fan$ occurs in the boundary of $\cInCone(\Fan)$.  We will call a
polytope $Q$ inscribed \Defn{relatively to} $\Fan$ if
$Q \in \cInCone(\Fan)$.  Consequently, a coarsening $\Fan'$ of $\Fan$
is inscribable \Def{relative to $\Fan$}, if $\Fan' = \Fan(Q)$ for some
$Q \in \cInCone(\Fan)$.

\begin{thm}\label{thm:Q_in_cInSpc}
    Let $\Fan$ be an inscribable fan and $Q \in \cTypeCone(\Fan)$.
    Then $Q \in \cInCone(\Fan)$ if and only if $Q$ is inscribed and 
    $(V(Q)-c(Q)) \cap R \neq \emptyset$ for all regions $R \in \Fan$.
\end{thm}
\begin{proof}
    Let $(P_n)_{n \in \N}$ be a sequence of polytopes in $\InCone(\Fan)$ which
    converges to $Q$. We may assume that $c(P_n) = 0$ for all $n \ge 0$ and by
    Lemma~\ref{lem:key} and Corollary~\ref{cor:key_int}, $P_n$ is uniquely
    determined by $\{p_n\}  = V(P_n) \cap \interior(R)$ for some arbitrary but
    fixed region $R \in \Fan$. It follows that $\lim_{n \to \infty} p_n = p
    \in R$. We infer from Corollary~\ref{cor:key_reflect} that $Q$ is
    inscribed with $c(Q) = 0$ and thus $p$ is a vertex of $Q$. This shows that
    $V(Q) \cap R \neq \emptyset$ for all regions $R \in \Fan$.

    For the only if part, assume that $c(Q) = 0$. Since
    $Q \in \cTypeCone(\Fan)$, the normal fan $\Fan(Q)$ is a coarsening
    of $\Fan$ and there is a region $C \in \Fan(Q)$ with
    $R \subseteq C$. Now $Q$ is inscribed and
    Corollary~\ref{cor:key_int} gives that $\interior C$ contains
    exactly one vertex of $Q$ and that $\partial C$ does not contain
    any vertex of $Q$. Since $V(Q) \cap R \neq \emptyset$, we see
    that $|V(Q) \cap R| = 1$.  We write $q_R$ for the unique vertex of
    $Q$ contained in $R$.

    Let $R, R' \in \Fan$ be two regions such that $R \cap R'$ is a wall and
    associated reflection $s_{RR'}$.  Let $C, C' \in \Fan(Q)$ be the unique
    regions such that $R \subseteq C$ and $R' \subseteq C'$.  If $C = C'$,
    then $q_R = q_{R'} \in R \cap R'$ and therefore $s_{RR'}(q_{R}) = q_{R'}$.
    If $C \neq C'$, then $s_{CC'} = s_{RR'}$ and therefore $s_{RR'}(q_R) =
    s_{CC'}(q_C) = q_{C'} = q_{R'}$. Let $P \in \InCone(\Fan)$ with $c(P) = 0$.
    It follows that $Q + \eps P \in \InCone(\Fan)$ for every $\eps > 0$ and
    hence $Q \in \cInCone(\Fan)$.
\end{proof}

\begin{cor}\label{cor:based_cInSpc}
    Let $R_0 \in \Fan$ be a region. The linear map $v_{R_0}$ extends to a linear
    homeomorphism
    \[
        \cInCone(\Fan) \ \cong \ \cInCone(\Fan, R_0)\,.
    \]
    This extension is given by $\{v_{R_0}(Q)\} = (V(Q) - c(Q)) \cap R$.
\end{cor}

A second characterization of inscribability relative to a fan is given
by the following. Let $C \in \Fan$ be a cone and
$P \in \cTypeCone(\Fan)$. The face $P^C \defeq P^c$ does not depend on
$c \in \relint C$.

\begin{thm}\label{thm:edges_orthogonal}
    Let $\Fan$ be a polytopal fan and $Q \in \cTypeCone(\Fan)$ be an
    inscribed polytope with $c(Q) = \0$. The following are equivalent:
    \begin{enumerate}[\rm (i)]
    \item $Q \in \cInCone(\Fan)$;
    \item $c(Q^C) \in \lin C$ for all walls $C$ of $\Fan$.
    \end{enumerate}
\end{thm}
\begin{proof}
    (i) $\Rightarrow$ (ii): Let $C \in \Fan$ be a cone of codimension $1$ and
    let $R, R' \in \Fan$ be the regions such that $C = R \cap R'$.  If $v_R =
    v_{R'}$, then $P^C = v_R$ and $v_R \in R \cap R' = C$ by
    Theorem~\ref{thm:Q_in_cInSpc}. Otherwise, $v_R \neq v_{R'}$ are the
    endpoints of the edge $P^C$ whose center is in $C$ by
    Corollary~\ref{cor:key_reflect}.

    (ii) $\Rightarrow$ (i): Let $R \in \Fan$ be a region and $v_R$ be the
    corresponding vertex of $Q$. Let $C \subset R$ be a wall with supporting
    hyperplane $H \defeq \lin C$. Now $v_R \in Q^C$ and $c(Q^C) \in H$
    implies that $v_R$ is not separated from $R$ by $H$. This holds for all
    walls $C$ and shows $v_R \in R$.  Theorem~\ref{thm:Q_in_cInSpc} then
    implies $Q \in \cInCone(\Fan)$.
\end{proof}

We close this section with a description of the facets of $\cInCone(\Fan)$ when
$\Fan$ is an \emph{even} inscribable fan.
The coarsenings of $\Fan$ are encoded by certain contractions of $G(\Fan)$.
Recall that the \Def{contraction} of an edge $e = uv$ in a simple graph $G =
(V,E)$ is the graph $G/e$ with nodes $V \setminus v$ and edges $\{ e \in E : v
\not\in e \} \cup \{ uw : vw \in E \}$. Our definition of contractions does
not produce parallel edges and hence stays in the category of simple graphs.
If $\Fan'$ coarsens $\Fan$, then let 
\[
    \Contr(\Fan,\Fan')  \ \defeq \ \{ RS \in E(\Fan) : R, S \subseteq C \text{
    for some region } C \in \Fan' \} \, .
\]
Then $G(\Fan')$ is isomorphic to $G(\Fan)/\Contr(\Fan,\Fan')$. On the level
of weights $\lambda(\ell)$ for PL-functions $\ell$ supported on $\Fan$,
coarsenings correspond to $\lambda : E(\Fan) \to \R$ with $\lambda(e) =
0$ for $e \in \Contr(\Fan,\Fan')$.

\begin{prop}
    Let $\Fan$ be an inscribable and even fan. Let $\Fan'$ be inscribable
    relative to $\Fan$ corresponding to a facet of $\cInCone(\Fan)$.
    Then $\Contr(\Fan,\Fan') \subseteq E(\Fan)$ is a matching.
\end{prop}
\begin{proof}
    Let $Q \in \partial \cInCone(\Fan)$ be an inscribed polytope with
    $c(Q) = \0$ and $\Fan(Q) = \Fan'$ and assume that
    $\Contr(\Fan, \Fan')$ is not a matching. Thus, there exists a
    region $R_0$ and two edges
    $R_0 S, R_0 S' \in \Contr(\Fan, \Fan')$, i.e.\ $R_0$, $S$ and $S'$
    are coarsened to a common region $C$ in $\Fan(Q)$. By
    Theorem~\ref{thm:Q_in_cInSpc} we see that
    $v_C(Q) \in R_0 \cap S \cap S'$. But, since $\Fan$ is even and
    hence $\cInCone(\Fan,R_0) \subseteq R_0$ is a full-dimensional
    subcone, $v_C(Q)$ does not lie in the interior of a facet of
    $\cInCone(\Fan,R_0)$, so neither does $Q$ lie in the interior of a
    facet of $\cInCone(\Fan)$ by Corollary~\ref{cor:based_cInSpc}.
\end{proof}

\subsection{Virtual polytopes and the inscribed space}\label{sec:virtual}
Recall from Section~\ref{sec:intro_virtual} that the type space
$\TypeSpc(\Fan)$ is the Grothendieck group of the monoid
$(\TypeCone(\Fan),+)$. Since $\TypeCone(\Fan)$ is a convex cone, we can
identify $\TypeSpc(\Fan)$ with $\TypeCone(\Fan) + (-\TypeCone(\Fan))$. The
following well-known result gives $\TypeSpc(\Fan)$ a simple interpretation in
terms of piecewise-linear functions.
\begin{prop}\label{prop:diff_convex}
    Let $\Fan$ be a polytopal fan in $\R^d$. For all
    $\ell \in \PL(\Fan)$ there exist polytopes $P,Q \subseteq \R^d$
    with $\Fan = \Fan(P) = \Fan(Q)$ such that $\ell = h_P - h_Q$.
    Moreover, if $\ell = h_{P'} - h_{Q'}$ for polytopes $P',Q'$, then
    $P + Q' = P' + Q$.
\end{prop}
\begin{proof}
    Let $P$ be a polytope with normal fan $\Fan$.  By linearity, for any $\mu
    > 0$, we have $\lambda(\ell + \mu h_P) = \lambda(\ell) + \mu
    \lambda(h_P)$. Since $h_P$ is strictly convex, there exists a $\mu > 0$
    such that $\lambda(\ell) + \mu \lambda(h_P) > 0$. That is, the function
    $\ell + \mu h_P$ is strictly convex and hence the support function of a
    polytope $Q \in \TypeCone(\Fan)$, by Proposition~\ref{prop:PL}. This
    implies $\ell = h_Q - h_{\mu P}$.  For the final statement, observe that
    $h_{P+Q'} = h_{P} + h_{Q'}$.
\end{proof}

Hence $\TypeSpc(\Fan)$ can be identified with piecewise-linear functions
supported on $\Fan$ up to translation and we write $P-Q$ for $h_P - h_Q$.  If
$h_P - h_Q$ is convex, then it is the support function of a polytope $R$ with
$P = Q+R$ and hence $P-Q = R$ is a polytope. Otherwise, we call $P-Q$ a
\Def{virtual polytope}. Virtual polytopes arise naturally in a number of
settings; cf.~\cite{PS}. Notably, $\TypeSpc(\Fan)$ is the first graded piece
in McMullen's polytope algebra~\cite{McMullen-algebra} and its dual space are
$1$-homogeneous translation-invariant valuations on polytopes with normal fan
$\Fan$. 

Let $\ell \in \PL(\Fan)$. For every region $R \in \Fan$, let
$v_R(\ell) \in \R^d$ such that $\inner{v_R(\ell), x} \equiv \ell_R(x)$. We
define the \Def{vertex set}
$V(\ell) \defeq \{v_R(\ell) : R \in \Fan \text{ region}\}$. If
$\ell = h_P$, then $V(h_P)$ is precisely the vertex set of $P$. In
general $V(\ell)$ is not in bijection to the regions of $\Fan$. We
call a PL function $\ell$ \Def{inscribed} if $V(\ell)$ lies on a
sphere. In this case, there exists a unique sphere containing
$V(\ell)$ whose center, denoted $c(\ell)$, is in
$\aff V(\ell)$. For a different geometric perspective, let us associate
an
\[
    \mathcal{H}(\ell) \defeq \{ \ker( \ell_R(x) - t) \subset \R^d \times \R :
    R \in \Fan \text{ region} \} \, .
\]
This is the smallest arrangement of linear hyperplanes containing the graph of
$\ell$.

It follows that $\ell$ is inscribed if and only if there is a sphere in
$\R^{d+1}$ that is tangent to all planes in $\mathcal{H}(\ell)$. If $\ell$ is
strictly convex, then the sphere can be chosen to be contained in the epigraph
of $\ell$ and shows the duality between inscribed and \emph{circumscribed}
polytopes.

If $C \in \Fan$ is a cone, we define the \Defn{localization} $\ell_C$
of $\ell$ at $C$ to be the PL function supported on the localization
$\Fan_C$, where $\ell_C$ coincides with $\ell_R$ on the cone
$C - R \in \Fan_C$. If $F$ is a face of a polytope $P$, then
$(h_P)_C = h_F$, where $C \defeq N_FP$ is the normal cone of $F$, so
localizations are the natural generalization of faces to PL functions.
Clearly, $V(\ell_C) \subseteq V(\ell)$ and therefore an inscribed PL
function has inscribed localizations. 

Example~\ref{ex:relatively_inscribable} implies that the collection
$\{ P - Q \in \TypeSpc(\Fan) : h_P - h_Q \text{ inscribed}\}$ is not a
linear subspace of $\TypeSpc(\Fan)$.  In light of
Theorem~\ref{thm:edges_orthogonal}, we define an inscribed PL function
$\ell \in \PL(\Fan)$ to be \Defn{inscribed relative to $\Fan$}, if
$c(\ell_C) - c(\ell) \in \lin C$ for all cones $C \in F$ of
codimension $1$. Note that this definition retains
Corollary~\ref{cor:key_reflect} in that ``neighboring vertices''
should be obtained by reflecting at the hyperplanes of walls. We
define
\[
    \InSpc(\Fan) \defeq \{ P - Q \in \TypeSpc(\Fan) : h_P - h_Q
    \text{ inscribed relative to $\Fan$} \}\,.
\]
to be the \Defn{inscribed space} of $\Fan$.
Clearly, $P \in \cInCone(\Fan)$ if and only if $h_P \in \InSpc(\Fan)$,
thus:

\begin{prop}\label{prop:InSpc_is_pos_VInSpc}%
    Let $\Fan$ be a polytopal fan. Then
    $\InCone(\Fan) = \InSpc(\Fan) \cap \TypeCone(\Fan)$. If $\Fan$ is
    inscribable, then
    $\cInCone(\Fan) = \InSpc(\Fan) \cap \cTypeCone(\Fan)$.
\end{prop}

As a first consequence, we can now easily describe the facial structure of
$\cInCone(\Fan)$:
\begin{cor}\label{cor:relatively_inscribed}
    Let $\Fan$ be an inscribable fan. The faces of $\cInCone(\Fan)$ are
    in correspondence to relatively inscribable coarsenings of $\Fan$.
\end{cor}
\begin{proof}
    Since by assumption $\InCone(\Fan) \neq \emptyset$, we see from
    Proposition~\ref{prop:InSpc_is_pos_VInSpc} and
    Theorem~\ref{thm:typecones} that the faces if $\cInCone(\Fan)$ stem
    from relatively inscribable coarsenings of $\Fan$.
\end{proof}

We will now show that the based inscribed space $\InSpc(\Fan,R_0)$ of
Section~\ref{sec:ref_game} is linearly isomorphic to $\InSpc(\Fan)$.
Recall that for a complete fan $\Fan$, the based
inscribed space $\InSpc(\Fan,R_0)$ is given by all $v \in \R^d$ such that
\[
    v \ = \ t_\Walk(v) \ = \ s_{R_{k}R_{k-1}}\cdots s_{R_2R_1} s_{R_1R_0}(v)
\]
for every closed walk $\Walk = R_0 R_1 \dots R_k$ in $G(\Fan)$ starting in
$R_0$.

Let $v_0 \in \InSpc(\Fan,R_0)$ and $R \in \Fan$ a region. Choose a walk
$\Walk$ from $R_0$ to $R$ and define $v_R \defeq t_\Walk(v_0)$. We claim that
the collection $(\ell_R)_R$ given by
\[
    \ell_R(x) \ \defeq \ \inner{v_R,x} \quad \text{ for all } x \in R \,.
\]
is a piecewise-linear function supported on $\Fan$. If $S$ and $R$ are
adjacent regions, then choose an appropriate walk $\Walk$ from $R_0$ to $R$
which extends to a walk $\Walk' = \Walk S$. Then for $v_R \defeq t_\Walk(v_0)$
and $v_S \defeq t_{\Walk'}(v_0)$
\[
    \ell_S(x) - \ell_R(x) \ = \ \inner{v_S, x} - \inner{v_R, x} \ = \
    \inner{s_{R,S}(v_R), x} - \inner{v_R, x} \ = \ -2
    \inner{\alpha_{RS}, v_R} \cdot \inner{\alpha_{RS}, x}\,,
\]
so it satisfies equation~\eqref{eqn:PL_adjacent}, from which we can
deduce that $(\ell_R)_R$ is in fact a PL function on $\Fan$. We denote this
PL-function by  $\ell_{\Fan,R_0,v} \in \PL(\Fan)$.

\begin{prop}\label{prop:VInSpc_embedd}
    Let $\Fan$ be a full-dimensional and strongly connected fan in $\R^d$.
    The map 
    \[
        \Xi_{\Fan,R_0} : \InSpc(\Fan,R_0) \to \PL(\Fan) \qquad \Xi_{\Fan,
        R_0}(v) \ \defeq \ \ell_{\Fan, R_0, v}
    \]
    is a linear embedding of $\InSpc(\Fan,R_0)$ into $\PL(\Fan)$. The
    image $\Xi_{\Fan,R_0}(\InSpc(\Fan, R_0))$ does not depend on
    $R_0$.
\end{prop}
\begin{proof}
    Let us note that the map $\Xi_{\Fan,R_0}$ is indeed linear.  We will
    first prove that $\Xi_{\Fan,R_0}$ defines a linear map of
    $\InSpc(\Fan,R_0)$ into $\PL(\Fan)$. 
    
    We will prove the claim using Theorem~\ref{thm:PL-lambda}. Recall that
    for $RS \in E(\Fan)$, $\alpha_{RS}$ is the unit vector satisfying
    $\inner{\alpha_{RS},x} \le 0$ for all $x \in R$. If $\Walk = R_0R_1\dots
    R_k$ is a path in $G(\Fan)$ and $v \in \InSpc(\Fan,R_0)$, then define
    $v_0 \defeq v$ and $v_i \defeq s_{R_{i-1}R_i}(v_{i-1})$ for $1 \le i \le
    k$. Since $s_{R_{i-1}R_i}$ is a reflection in
    $\alpha_{R_{i-1}R_i}^\perp$, we get
    \[
        v_i \ = \ s_{R_{i-1}R_i}(v_{i-1}) \ = \ v_{i-1} -
        2\inner{\alpha_{R_{i-1}R_i}, v_{i-1}} \alpha_{R_{i-1}R_i} \, .
    \]
    If we set $\lambda_{R_{i-1}R_i} \defeq -2\inner{\alpha_{R_{i-1}R_i}, v_{i-1}}$,
    then
    \begin{equation}\label{eqn:point_to_lambda}
        v_i \ = \ v_0 + \sum_{i=1}^k \lambda_{R_{i-1}R_i} \alpha_{R_{i-1}R_i} \, .
    \end{equation}
    To see that this is well defined, let $\Walk' \defeq R_0'R_1'\dots
    R_{l-1}R_{l}'$ be a different walk with $R_0' = R_0$ and $R_{l-1}'R_l =
    R_{k-1}R_k$. Then $R_0\dots R_{k-1} R_{l-2}'\dots R_1'R_0$ is a closed
    walk and the definition of $\InSpc(\Fan,R_0)$ implies that hence $v_{k} =
    s_{R_{l-1}R_l}(v'_{l-1})$. 
    
    If $\Walk$ is a closed walk in $G(\Fan)$, then this implies $v_k = v_0$
    and hence $\sum_{i=1}^k \lambda_{R_{i-1}R_i} \alpha_{R_{i-1}R_i} = 0$.
    This means that $\lambda$ satisfies~\eqref{eqn:closed_walk} and hence
    $\Xi_{\Fan,R_0}(v)$ is a piecewise-linear function supported on $\Fan$
    with  data $(\inner{v,\cdot},\lambda)$.

    If $R_0'$ is a different base region and $\Walk'$ is a walk from
    $R_0$ to $R_0'$, then
    $\InSpc(\Fan, R_0') = t_{\Walk'}(\Fan, R_0)$ and
    $\ell_{\Fan, R_0, v} = \ell_{\Fan, R_0', v'}$ for
    $v' \defeq t_{\Walk'}(v) \in \InSpc(\Fan, R_0')$, thus
    $\Xi_{\Fan,R_0}(\InSpc(\Fan, R_0)) =
    \Xi_{\Fan,R_0'}(\InSpc(\Fan, R_0'))$.    

    To show that $\Xi_{\Fan,R_0}$ yields an embedding of $\InSpc(\Fan,R_0)$
    into $\PL(\Fan)$, we argue that the point $v$ can be recovered from
    $\lambda$. For this observe that $R_0$ is a full-dimensional cone and let
    $l \defeq \dim \lineal(\Fan)$. There are regions $S_1,\dots,S_{d-l} \in
    \Fan$ adjacent to $R_0$ such that the vectors $\alpha_{R_0S_i}$ for
    $i=1,\dots,d$ are linearly independent and span $(\lineal(\Fan))^\perp$.
    Thus $v$ is the unique point with $-2\inner{\alpha_{R_0S_i},v} =
    \lambda_{R_0S_i}$ and therefore can be reconstructed from $\lambda$.
\end{proof}

Notice that $\Xi_{\Fan,R_0}(v)$ is inscribed relatively to $\Fan$ for
all $v \in \InSpc(\Fan,R_0)$ since $v_R(\ell_{\Fan, R_0, v}) =
v_R$. Conversely, if $\ell$ is inscribed relative to $\Fan$, then for
some linear function $l(x)$, we have that $V(\ell + l)$ is contained
in a sphere centered at the origin. It is now easy to see that
$\ell + l = \Xi_{\Fan,R_0}(v)$ for $v \defeq v_{R_0}(\ell + l)$. This shows
the isomorphism between $\InSpc(\Fan, R_0)$ and $\InSpc(\Fan)$.

\begin{cor}\label{cor:VInSpc_embedd}
    $\InSpc(\Fan, R_0) \cong \InSpc(\Fan)$ for every complete fan $\Fan$ and
    base region $R_0 \in \Fan$.
\end{cor}

Using Proposition~\ref{prop:InSpc_is_pos_VInSpc}, we also see, that:
\begin{cor}
    Let $\Fan$ be an inscribable fan. Then $\InCone(\Fan) + (-\InCone(\Fan)) =
    \InSpc(\Fan)$.
\end{cor}
Clearly, this result is false, if $\Fan$ is not inscribable but virtually
inscribable, i.e.~$\InCone(\Fan) = \emptyset$, but $\InSpc(\Fan) \neq
\emptyset$.  For example, Proposition~\ref{prop:2d_virtual_insc} implies that
$2$-dimensional complete fans with an odd number of rays always have an
inscribed virtual polytope. Figure~\ref{fig:non_insc_hexagon} shows an even
$2$-dimensional fan all whose inscribed polytopes are virtual.

According to Steinitz, combinatorial types of $3$-dimensional polytopes are in
bijection with planar and $3$-connected graphs; cf.~\cite{ziegler}. While
there are complete $3$-dimensional fans $\Fan$ that are not
polytopal~\cite[Example~7.5]{ziegler}, Steinitz result shows that $G(\Fan)$ is
always the dual graph of a polytopal fan. It was shown by Steinitz that there
are planar and $3$-connected graphs $G$ that cannot be realized as the graph
of an inscribed $3$-polytope.

\begin{quest}
    Is there a planar and $3$-connected graph $G$ that can not be realized as
    the graph of an inscribed polytope but as the the graph of an inscribed
    \emph{virtual} polytope?
\end{quest}

Let $G$ be a planar and $3$-connected graph and $G' = (V, E)$ be its planar
dual. Rivin~\cite{Rivin} showed that there is an inscribed $3$-polytope $P$
with graph $G$ if and only if there is $\omega : E \to (0, \pi)$ such that for
every cycle in $C \subseteq E$
\begin{equation}\label{eqn:Rivin}
\begin{aligned}
    \sum_{e \in C} \omega(e) \ &= \ 2\pi \quad \text{ if $C$ is
    non-separating (i.e.\ bounding a $2$-face of $G'$)} \\
    \sum_{e \in C} \omega(e) \ &> \ 2\pi \quad \text{ if $C$ is separating}.
\end{aligned}
\end{equation}

Let us call a planar and $3$-connected graph $G$ \Def{virtually
  inscribable} if there is a virtually inscribable polytopal fan
$\Fan$ with $G(\Fan) = G$.

\begin{quest}
    Can Rivin's characterization~\eqref{eqn:Rivin} be extended to decide if a
    graph is virtually inscribable?
\end{quest}

\subsection{Computing the inscribed cone}\label{sec:computing}%
The reflection game of Section~\ref{sec:ref_game} gives a mean to compute
$\InCone(\Fan, R_0)$ via \eqref{eqn:InSpc_intersect_rep}. However, 
the process is computationally quite involved. In the
remainder of this section, we give a simpler and more elegant description of
$\InSpc(\Fan)$ utilizing Corollary~\ref{cor:TypeCone_desc}.

In order to describe $\lambda(\InSpc(\Fan))$, we first consider a
single cycle in $G(\Fan)$. More generally, let
$A = (\alpha_1,\dots,\alpha_n) \in \R^{d \times n}$ be an ordered
collection of unit vectors. We write $s_i$ for the reflection in
$\alpha_i^\perp$. We are interested in the question when there is a
point $q_0 \in \R^d$ such that $s_ns_{n-1} \dots s_1(q_0) = q_0$.  Let
$q_i \defeq s_is_{i-1} \dots s_1(q_0)$, that is,
\[
    q_i \ = \ s_i(q_{i-1}) \ = \ q_{i-1} - 2 \inner{\alpha_i,q_{i-1}} \alpha_i
    \, .
\]
Furthermore, set $\lambda_i \defeq -2\inner{\alpha_i,q_{i-1}}$. Then
\[
    q_i \ = \ q_0 + \sum_{j=1}^i \lambda_i \alpha_i
\]
and, in particular,
\[
    \lambda_i \ = \ -2\inner{\alpha_i,q_0} - 2 \sum_{j=1}^{i-1}
    \inner{\alpha_i,\alpha_j}
    \lambda_j \, .
\]
Note that $\lambda : \R^d \to \R^n$ with $q_0 \mapsto
(\lambda_i)_{i=1,\dots,n}$ is a linear map and there is $q_0$ with $q_n = q_0$
if and only if $A\lambda(q_0) = 0$. 

Let $G$ be a \Def{Gale transform} of $A$, that is, a matrix $G$ such that
$\ker G = \im A^t$. We define the \Def{skew Gram matrix} of
$\alpha_1,\dots,\alpha_n$ as the skew-symmetric matrix $R =
R(\alpha_1,\dots,\alpha_n) \in \R^{n \times n}$ with $R_{ij} = -R_{ji} =
\inner{\alpha_i,\alpha_j}$ for $i < j$ and $R_{ii} = 0$.

\begin{lem}\label{lem:prod_reflections}
    Let $A = (\alpha_1,\dots,\alpha_n) \in \R^{d \times n}$ be an ordered
    collection of unit vectors. Let $G$ be a Gale transform for $A$ and $R =
    R(\alpha_1,\dots,\alpha_n)$ the skew Gram matrix. There is $q_0 \not\in
    \im(A)^\perp$
    with $s_n \dots s_1(q_0) = q_0$ if and only if there is $\lambda \in \R^n
    \setminus \{0\}$ with
    \[
        A \lambda \ = \  0 \quad \text{ and } \quad GR \lambda \ = \ 0  \, .
    \]
\end{lem}
\begin{proof}
    Let $\gamma = A^t q_0 = (\inner{\alpha_i,q_0})_{i=1,\dots,n}$ and
    denote by $\delta_{i,j}$ the Kronecker delta. Then
    $\lambda = S \gamma$, where $S = S_n S_{n-1} \dots S_1$ and
    \[
        (S_k)_{ij} \ = \
        \begin{cases}
            -2\inner{\alpha_k,\alpha_j} & \text{ if } j \leq i = k,\\
            \delta_{i,j} & \text{ otherwise. } \\
        \end{cases}
    \]
    In particular, $\lambda = \lambda(q_0)$ for some $q_0$ if and only if
    $S^{-1}\lambda \in \im A^t$ if and only if $G S^{-1} \lambda = 0$.

    Note that $S_k$ is a lower triangular matrix with inverse given by
    \[
    (S_k)^{-1}_{ij} \ = \ 
        \begin{cases}
            -\frac{1}{2} & \text{ if } j = i = k,\\
            -\inner{\alpha_k,\alpha_j} & \text{ if } j < i = k,\\
            \delta_{i,j} & \text{ otherwise. }\\
        \end{cases} 
    \]
    In particular, $S^{-1} = S_1^{-1} \cdots S_n^{-1}$ is given by
    \[
        S^{-1}_{ij} \ = \ 
        \begin{cases}
            -\inner{\alpha_i,\alpha_j} & \text{ if } j < i, \\
            -\frac{1}{2} \delta_{i,j} & \text{ otherwise. } 
        \end{cases}
    \]
    If we consider the canonical decomposition of $S^{-1}$ as a sum of a
    symmetric and a skew-symmetric matrix, we see
    \[
        S^{-1} + (S^{-1})^t = -A^t A \qquad \text{and} \qquad
        S^{-1} - (S^{-1})^t = R  \, .
    \]
    Since $G$ is a Gale-transform of $A^t$, we have $G S^{-1} = \frac{1}{2} G
    R$ and this yields the claim.
\end{proof}

We can combine Corollary~\ref{cor:TypeCone_desc} with
Proposition~\ref{prop:cycle_basis}(i) and
Lemma~\ref{lem:prod_reflections} to give a description of
$\InSpc(\Fan)$ and of $\InCone(\Fan)$ as a subcone of
$\TypeCone(\Fan)$ that is effectively computable.

For the purpose of computation, we view $G(\Fan) = (V,E)$ as a directed
symmetric graph.  More precisely, $G(\Fan)$ is a directed graph with nodes $V$
corresponding to the regions of $\Fan$ and directed edges $E \subseteq V
\times V$ corresponding to the walls of $\Fan$, with the property that $RS \in
E$ implies $SR \in E$. If $RS \in E$, then $\alpha_{RS}$ is the unit outer
normal to $R$. This yields a map $\alpha : E \to S^{d-1}$ into the unit sphere
$S^{d-1} \subset \R^d$ with $\alpha_{SR} = -\alpha_{RS}$. The pair
$(G(\Fan),\alpha)$ gives an encoding of $\Fan$ that is the input to our
algorithm.

If $\Fan = \Fan(P)$ for a convex polytope $P \subset \R^d$, then $G(\Fan)$ can
be computed from $P$ by first determining the edge graph of $P$. The regions
$R$ are in bijection to the vertices $v_R$ of $P$ and $\alpha_{RS}$ is the
unit vector $\frac{v_S - v_R}{\|v_S - v_R\|}$.

\newcommand\cyc{\mathrm{Cyc}(\Fan)}%
Let $\cyc$ be the directed cycles around codimension-$2$ cones
of $\Fan$.  That is, cycles correspoding to $2$-faces of $P$. For any such
cycle $C = R_1\dots R_n$, let $A_C =
(\alpha_{R_1R_2},\dots,\alpha_{R_{n}R_1})$. Further, let $G_C$ and $R_C$  be a
Gale transform and skew Gram matrix for $A_C$. Moreover, for every region $R
\in V$, let $N(R) = \{ S : RS \in E\}$ be the neighbors of $R$ and $G_R$ a Gale
transform of $A_R = (\alpha_{RS} : S \in N(R))$.

\begin{thm}\label{thm:VInSpc_rep}
    Let $\Fan$ be a complete fan represented by $(G(\Fan),\alpha)$. Then
    $\InSpc(\Fan)$ is linearly isomorphic to the collection of $\lambda \in
    \R^{E(\Fan)}$ such that
    \begin{align}
        A_{C} \lambda|_{C} \ &= \ 0 \quad \text{for all } C \in \cyc%
        \label{eqn:AC}\\
        G_{C}R_{C} \lambda|_{C} \ &= \ 0 \quad \text{for all } C \in \cyc%
        \label{eqn:GC}\\
        G_{R} \lambda|_{N(R)} \ &= \ 0 \quad \text{for all } R \in \Fan
        \label{eqn:GR}
    \end{align}
    The inscribed cone $\InCone(\Fan)$ is given by those $\lambda \in
    \R^{E(\Fan)}$ with $\lambda_{RS} > 0$ for all $RS \in E(\Fan)$.
\end{thm}
\begin{proof}
    We work with the based inscribed space $\InSpc(\Fan,R_0)$ for some fixed
    region $R_0 \in \Fan$. Every element of $\InSpc(\Fan)$ is thus represented
    by some $v_{R_0} \in \R^d$ and all other $v_R$ are determined by
    \eqref{eqn:point_to_lambda} for walks $R_0$ to $R$. This yields a $\lambda
    : E(\Fan) \to \R$ that satisfies \eqref{eqn:AC} and \eqref{eqn:GR}. We
    also note that $-2A^t_{R}v_{R} = \lambda|_{N(R)}$ for all $R$, which shows
    that \eqref{eqn:GR} is satisfied.

    Now let $\lambda$ satisfy the given conditions. For any $R \in \Fan$,
    $G_R$ is a Gale transform for $A_R$ and $A_R$ has full rank, so there is a
    unique $v_R$ with $-2 \inner{\alpha_{RS},v_R} = \lambda_{RS}$ for all $S
    \in N(R)$.  Fix $RS \in E$ and let $v = s_{RS}(v_R)$. We claim that $v_S =
    v$. Note that $\inner{\alpha_{RS},v_S} = \inner{\alpha_{SR},v} =
    -\frac{1}{2}\lambda_{RS}$ by construction.

    Let $C \in \cyc$ be a cycle containing $QRST$ for regions $Q,T \in \Fan$.
    Let $s_C$ be the composition of reflections along $C$ starting at $R$.  By
    condition \eqref{eqn:AC} and \eqref{eqn:GC},
    Lemma~\ref{lem:prod_reflections} assures us that there is point $u$ such
    that $u = s_C(u)$. The vectors $A_C$ span a subspace of dimension $2$ and
    $\alpha_{QR}$ and $\alpha_{RS}$ are linearly independent. Therefore,
    $s_C(u) = u$ if and only if $-2\inner{\alpha_{QR},u} = \lambda_{QR}$ and
    $-2\inner{\alpha_{RS},u} = \lambda_{RS}$. In particular $s_C(v_R) = v_R$.
    This implies $\inner{\alpha_{ST},v} = \inner{\alpha_{ST},v_S} =
    -\frac{1}{2}\lambda_{ST}$. Every codimension-$1$ cone is contained in at
    least $d-1$ codimension-$2$ cones. Thus there are at least $d-1$ such
    cycles $C$ through $RS$ for which the vectors $\alpha_{ST}$ are linearly
    independent. This means that $v$ is the unique solution to $-2A^t_{S}v =
    \lambda|_{N(S)}$. 
\end{proof}

\begin{rem}
    The result that precedes can also be interpreted in the context of
    Corollary~\ref{cor:inscribability_0_k}: For an ordinary (that is
    non-virtual) polytope $P$, conditions~\eqref{eqn:AC}
    and~\eqref{eqn:GC} are equivalent to all $2$-faces being
    inscribed, while~\eqref{eqn:AC} and~\eqref{eqn:GR} are equivalent
    to the existence of a sphere containing every vertex together with
    its neighbors.
\end{rem}

This, finally, allows us to prove Theorem~\ref{thm:algo_rational}.

\begin{proof}[Proof of Theorem~\ref{thm:algo_rational}]
    If $\Fan$ is \Def{rational}, that is, if $\alpha_{RS} \in \Q^d$ for all
    $RS \in E(\Fan)$, then $\InCone(\Fan) \subset \TypeCone(\Fan)$ is a
    rational cone and hence contains rational points in the relative interior,
    provided it is non-empty.  This shows the second statement. 
    
    For the first statement, let $\Fan$ be a (polytopal) fan in $\R^d$ with
    $n$ regions.  The linear equations of type~\eqref{eqn:GR} is of order $d$
    times the number of walls of $\Fan$ and is of order $dn^2$. If $C \in
    \cyc$ is a cycle of length $m < n$, then the linear equations of
    type~\eqref{eqn:AC} and \eqref{eqn:GC} are of order $m^2$. There are at
    most $\binom{n}{3}$ cycles. Computing a point in $\InCone(\Fan)$ is
    therefore equivalent to finding a solution to a system of polynomially
    many linear equations and strict inequalities.  The bit complexity of the
    system is polynomial in $(G(\Fan),\alpha)$ and the ellipsoid method
    determines in polynomial time if a solution exists; cf.~\cite{GLS}. This
    completes the proof of Theorem~\ref{thm:algo_rational}.
\end{proof}

\subsection{Non-inscribable fans and inscribable coarsenings}

If $\Fan$ is not polytopal, that is, not the normal fan of a polytope, then
there is a unique finest coarsening $\Fan_{pc}$ of $\Fan$ such that
$\Fan_{pc}$ is polytopal. A representative is given by taking the Minkowski
sum of polytopes for all possible polytopal coarsenings of $\Fan$. A
different way to see is, is to note that by
Corollary~\ref{cor:TypeCone_desc}, 
\[
    \cTypeCone(\Fan_{pc}) \ \cong \ \R^{E(\Fan)}_{\ge 0} \cap \{ \lambda \in
    \R^{E(\Fan)} : \text{$\lambda$ satisfies \eqref{eqn:closed_walk} for all
    closed walks} \}\, ,
\]
where the last set is isomorphic to $\TypeSpc(\Fan)$. This embedding of
the type space $\TypeSpc(\Fan) \hookrightarrow \R^{E(\Fan)}$ meets
the cone $\R^{E(\Fan)}_{\ge 0}$ in the relative interior of some
inclusion-maximal face and the intersection is precisely the closed
type cone of $\Fan_{pc}$.

Analogously, there is a \Def{canonical inscribable coarsening} (cic)
of $\Fan$ given by the fan $\Fan_{cic}$ corresponding to the largest
face of $\cTypeCone(\Fan_{pc})$ meeting the inscribed space
$\InSpc(\Fan)$. If $\Fan$ is inscribed, then
$\Fan_{cic} = \Fan_{pc} = \Fan$. It follows from
Example~\ref{ex:relatively_inscribable} that $\Fan_{cic}$ is
\emph{not} the common refinement of all inscribable coarsenings of
$\Fan$, but the common refinement of all \emph{relatively} inscribable
coarsenings of $\Fan$. In particular, $\Fan_{cic} = \{ \R^d \}$ is
possible even if $\Fan$ has inscribable coarsenings. It is clear that
\[
    \cTypeCone(\Fan_{pc}) \cap \InSpc(\Fan) \ \subseteq \
    \InCone(\Fan_{cic}) \, ,
\]
but the following example shows that the inclusion can be strict in general.

\begin{example}\label{ex:fan_cic}
    Let $\Fan$ be a $2$-dimensional fan with $5$ rays. If no rays are
    antipodal, then $\cTypeCone(\Fan)$ is a $3$-dimensional cone with $5$
    rays, i.e., the cone over a (different) pentagon. The precise
    combinatorics, i.e., which cones can be merged in a coarsening, depend on
    the geometry of $\Fan$.
    
    For example, let $\Fan$ be a $2$-dimensional fan with profile
    $\beta(\Fan) = (\frac{\pi}{3},
    \frac{\pi}{2},\frac{\pi}{3},\frac{\pi}{2},\frac{\pi}{3})$.  A
    cross section of $\cTypeCone(\Fan)$ is shown in
    Figure~\ref{fig:pentagon_type_cone}. The vertices of the cross
    section are labelled by the corresponding indecomposable
    coarsenings of $\Fan$. All of these indecomposable coarsenings
    correspond to triangles and are therefore inscribable.

    By Proposition~\ref{prop:2d_virtual_insc} and Theorem~\ref{thm:2d_insc},
    $\InSpc(\Fan)$ is a $1$-dimensional subspace that does not meet the
    interior of $\cTypeCone(\Fan)$. In fact, $\InSpc(\Fan)$ meets the
    relative interior of a $2$-dimensional face $F \subseteq
    \cTypeCone(\Fan)$. The corresponding fan $\Fan_{cic}$ has $4$ rays and it
    follows from Example~\ref{ex:insc_quadrangle} that $\InCone(\Fan_{cic}) =
    F$. Thus, $\InSpc(\Fan) \cap \cTypeCone(\Fan) \subsetneq
    \InCone(\Fan_{cic})$. This situation is depicted in
    Figure~\ref{fig:pentagon_type_cone}.
    \begin{figure}[h] 
        \centering
        \includegraphics[width=0.45\textwidth]{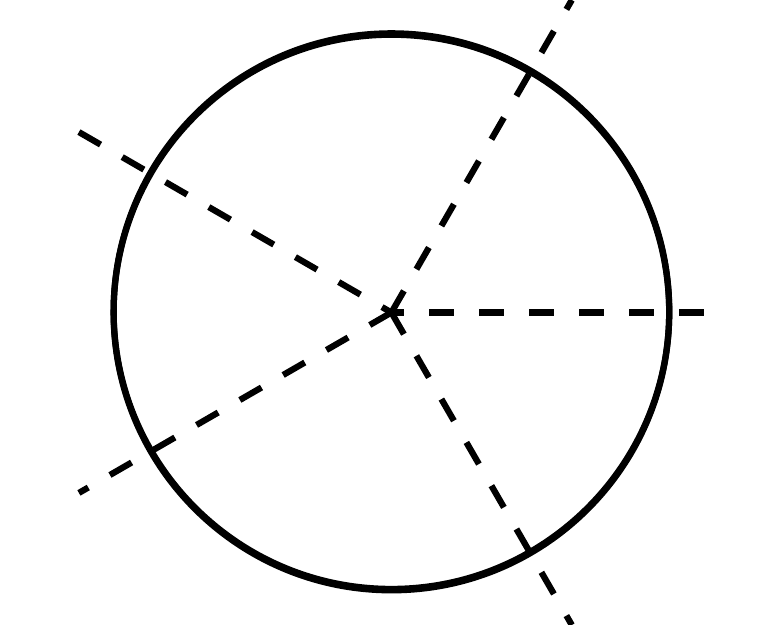}\qquad
        \includegraphics[width=0.45\textwidth]{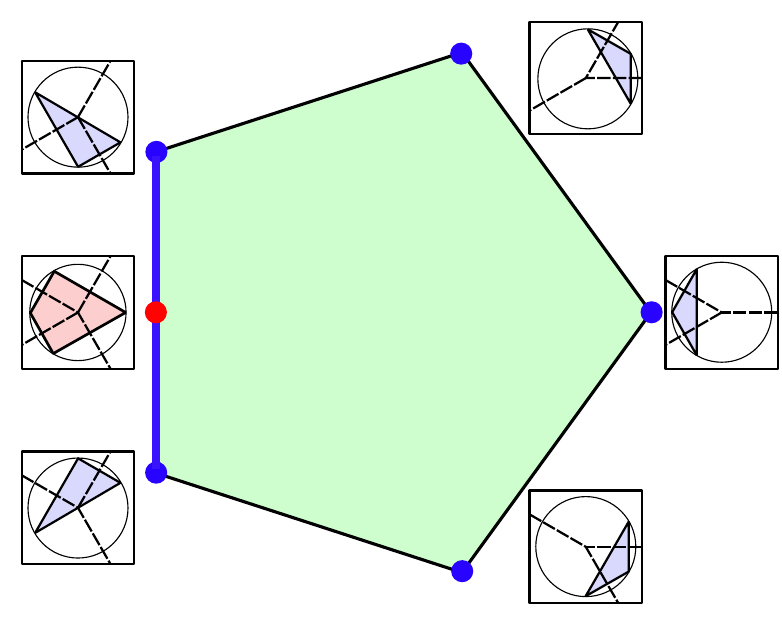}
        \caption{A $2$-dimensional fan $\Fan$ with profile
          $\beta(\Fan) = (\frac{\pi}{3},
          \frac{\pi}{2},\frac{\pi}{3},\frac{\pi}{2},\frac{\pi}{3})$
          and a cross section of the closed type cone. The inscribed
          weak Minkowski sums are marked in blue. The red point is the
          intersection $\InSpc(\Fan) \cap \cTypeCone(\Fan)$.}
        \label{fig:pentagon_type_cone}
    \end{figure}
\end{example}

\section{Routed trajectories and reflection groupoids}\label{sec:erratic}
The constructions in the Sections~\ref{sec:ref_game} and~\ref{sec:type} prompt
two related generalizations, namely \emph{routed trajectories} and
\emph{reflection groupoids}, which may provide a broader perspective on
normally inscribed polytopes.

\subsection{Routed trajectories}\label{sec:routed}

Let $B^d \subset \R^d$ be the $d$-dimensional unit ball and consider a
particle that starts from a point $t_0 \in \partial B^d = S^{d-1}$ along a
straight line trajectory. Upon collision with the boundary at the point $t_1
\in S^{d-1}$, the particle again takes off in a random direction and continues
to produce a trajectory $T = (t_0, t_1,\dots,t_k)$. We will assume that the
trajectory is closed and hence $t_k = t_0$. We define the \Def{route} of the
trajectory $T$ as $\alpha(T) = (\alpha_0,\dots,\alpha_{k-1})$, where $\alpha_i
= \R (t_{i+1} - t_i)$ is the line along which the particle moves. We will be
interested in the space of all trajectories with the same route.

More generally, let $G = (V,E)$ be a simple connected graph and let
$\alpha : E \to \PP^{d-1}$, where $\PP^{d-1}$ is the space of lines
through the origin in $\R^d$. We call $(G,\alpha)$ a \Def{routing
  scheme}. A \Def{trajectory} for $(G,\alpha)$ is a map
$T : V \to S^{d-1}$ such that if $C = v_0,v_1,\dots, v_k$ is a
cycle in $G$, then $T|_C = (T(v_0),\dots,T(v_k))$ is a trajectory with
route $\alpha|_C = (\alpha(v_0v_1),\dots,\alpha(v_{k-1}v_k))$. Thus,
we may interpret $G$ as the state space of a particle with $\alpha$
restricting the admissible directions for each state. The
\Def{trajectory space} is then
\[
    \Traj(G,\alpha) \ \defeq \ \{ T : V \to S^{d-1} : T \text{ is a trajectory
    for } (G,\alpha) \} \, .
\]

\begin{thm}\label{thm:subsphere}
    Let $(G,\alpha)$ be a routing scheme. Then $\Traj(G,\alpha)$ is
    homeomorphic to a subsphere $S^{d-1} \cap U$, where $U$ is a linear
    subspace.
\end{thm}

The crucial observation is again that if $uv \in E$, then $T(v)$ is the
reflection of $T(u)$ in the hyperplane $\alpha(uv)^\perp$. Hence $T$ is
determined by $T(v_0)$ for an arbitrary but fixed $v_0 \in V$. Moreover, let
us write $s_{uv}$ for the reflection in $\alpha(uv)^\perp$.  Then $t_0 \defeq
T(v_0)$ has to satisfy
\[
    s_{v_k v_{k-1}}\dots s_{v_1v_0}(t_0) \ = \ t_0   \, .
\]
for all closed walks $\Walk = v_0 v_1 \dots v_k$  starting at $v_0$. As in
Section~\ref{sec:ref_game}, we note that the collection of points $t_0 \in
\R^d$ satisfying these conditions yield a linear subspace $U \subset \R^d$ and
hence $\Traj(G,\alpha)$ is homeomorphic to $S^{d-1} \cap U$.

Every full-dimensional and strongly connected fan $\Fan$ naturally determines
a routing scheme $(G,\alpha)$. The results of Sections~\ref{sec:ref_game}
and~\ref{sec:type} can now be interpreted as follows.

\begin{cor}
    Let $\Fan$ be a full-dimensional and strongly connected fan in $\R^d$ with
    associated routing scheme $(G,\alpha)$. Then inscribed virtual polytopes
    with underlying fan $\Fan$ are in correspondence to trajectories for
    $(G,\alpha)$.
\end{cor}

To see the special role played by the inscribed polytopes
$\InCone(\Fan)$, observe that a route only determines a line along the
particle has to move but not a direction. Defining $\alpha$ to take values in
the space of oriented lines $\PP^{d-1}_+$, yields a correspondence to
$\InCone(\Fan)$.

\subsection{Reflection groupoids}\label{sec:groupoids}

Category theory gives a natural algebraic formalism for routing
schemes. Recall that a \Def{groupoid} is a (small) category $\Grp$ in
which every morphism is invertible; see~\cite[I.5]{Maclane}. Let
$G = (V,E)$ be a graph and $\alpha : E \to \PP^{d-1}$. As before, we
denote by $s_e$ the reflection in $\alpha_e^\perp$.  We can now
construct a groupoid $\Grp = \Grp(G,\alpha)$ with objects
$\Ob_\Grp = V$ and morphisms
\[
    \hom_\Grp(u,v) \ \defeq \  \{ t_\Walk : \Walk \text{ walk from $u$ to $v$}
    \}
\]
for $u,v \in V$, where the composition of morphisms is inherited from $O(d)$.
If $\Walk^{-1}$ is the reverse walk, then clearly $t_\Walk^{-1} =
t_{\Walk^{-1}}$. As $\Grp(G,\alpha)$ is \emph{generated} by reflections, we
call $\Grp(G,\alpha)$ a \Def{reflection groupoid}. A full-dimensional fan
$\Fan$ defines a reflection groupoid $\Grp(\Fan) \defeq (G(\Fan), \alpha)$
with $\alpha(RS) = (R \cap S)^\perp$ and we set $\Grp(P) \defeq \Grp(\Fan(P))$
for a polytope $P$.

For any $v \in V$, the set $\hom_\Grp(v) \defeq \hom_\Grp(v,v) = \{ t_\Walk :
\Walk \text{ closed walk based at $v$}\}$ is a group. If $G$ is connected,
then $\Grp$ is a \Def{connected} groupoid, that is, $\hom_\Grp(u,v) \neq
\emptyset$ for all $u,v \in V$. In this case, it is well-known that $\Grp$ is
determined up to isomorphism by $\hom_\Grp(v)$ for any fixed $v \in V$ and the
set $V$.  Indeed, for any $w \in V$ and a (non-canonical) choice $h_w \in
\hom_\Grp(v,w)$, we have for any $x,y \in V$
\[
    \hom_\Grp(x,y) \ = \ \{ h_y  g  h_x^{-1} : g \in \hom_\Grp(v) \} \, .
\]
In particular $\hom_\Grp(u) = h_u\hom_\Grp(u) h_u^{-1}$ for all $u \in V$.

The following makes an interesting connection to virtual inscribability.

\begin{thm}\label{thm:traj_inspc}
    Let $\Fan$ be a virtually inscribable fan and $\Grp = \Grp(\Fan)$ its
    reflection groupoid. Then the group $\hom_\Grp(R) \subset O(d)$ is
    generated by reflections for any region $R \in \Fan$.
\end{thm}
\begin{proof}
    As we stated above, $\hom_\Grp(R)$ is generated by elements of the
    form $t_\Walk^{-1} t_{\Walk'} t_\Walk$, where $\Walk'$ is a closed
    walk and $\Walk$ is a path from $R$ to the start point of
    $\Walk'$. Clearly, it suffices to only consider closed walks which
    form a cycle bases, and so we can choose $\Walk'$ to be a closed
    path around a cone $C$ of codimension $2$. We call $C$ odd/even if
    $\Walk'$ has odd/even length.
    
    If $\Fan$ is virtually inscribable, then the localization $\Fan_C$
    is virtually inscribable. Since $\Fan_C$ is essentially a
    $2$-dimensional fan, it follows from
    Proposition~\ref{prop:2d_virtual_insc} that $t_{\Walk'}$ is the
    identity if $C$ is even and $t_C$ is a reflection if $C$ is
    odd. Hence $\hom_\Grp(R)$ is generated by the reflections
    $t_\Walk^{-1} t_{\Walk'} t_\Walk$, for any walk $\Walk'$ around an odd
    codimension-$2$ cone.
\end{proof}

\begin{cor}
    If $\Fan$ is a virtually inscribable and even fan, then
    $\hom_{\Grp(\Fan)}(R)$ is trivial for all regions $R \in \Fan$.
    In particular $\hom_{\Grp(\Fan)}(R,S) = \{ s_{RS} \}$ for all
    adjacent regions $R,S \in \Fan$.
\end{cor}

Notice that the group $\hom_{\Grp(\Fan)}(R)$ is typically infinite,
even when it is generated by reflections. However, $\hom_\Grp(v)$ is
trivially a subgroup of the (possibly infinite) reflection group
$W(\Grp)$ generated by $s_e$ for $e \in E$.

\begin{cor}
    If $P$ is an inscribed generalized permutahedron, then
    $\hom_{\Grp(P)}(v)$ is a finite reflection group for any $v \in
    \Ob_{\Grp(P)}$.
\end{cor}

Let us look at some examples.

\begin{example}[Simplices]
    Let $S_d$ be the $d$-dimensional simplex with vertices
    $\0,e_1,\dots,e_d$.  The group $W(\Grp(S_d))$ is generated by
    reflections normal to $e_i$ and $e_i - e_j$ for $i < j$ and hence
    is the finite reflection group of type $B$.  Consider the vertex
    $v=\0$.  Walks around $2$-faces containing $\0$ give rise to
    reflections in $e_i - e_j$ for $i < j$ and since these walks form
    a cycle basis we have that $\hom_{\Grp(S_d)}(v)$ is in fact the
    finite reflection group of type $A_{d-1}$.
\end{example}

\begin{example}[Crosspolytopes]
    Let $K_d = \conv(\pm e_1,\dots, \pm e_d)$ be the $d$-dimensional
    cross\-polytope. The group $W(\Grp(K_d))$ is generated by reflections normal
    to $e_i - e_j$ and $e_i + e_j$ for $1 \le i < j \le d$ and hence is the
    finite reflection group of type $D_d$. The $2$-dimensional faces are of
    the form $\conv( \pm e_i, \pm e_j, \pm e_k)$ for distinct $i , j , k \in
    [d]$.  We denote the vertices by $\pm i$ and the $2$-faces by $(\pm i, \pm
    j, \pm k)$. 
    
    Starting in the vertex $i$, the $2$-faces $(i,j,k)$ and $(i,-j,-k)$ give
    rise to a reflection in $\{ x_k = x_j \}$. Similarly, the $2$-faces
    $(i,-j,k)$ and $(i,j,-k)$ yield a reflection in $\{ x_k = -x_j \}$.  For
    $v = e_1$, the group $\hom_{\Grp(K_d)}(v)$ is generated by the reflections in
    the hyperplanes $\{ x_k = \pm x_j \}$ for $1 < j < k \le d$. 
    This is group is isomorphic to $D_{d-1}$ fixing the line through $e_1$.
\end{example}

The group $\hom_{\Grp(G,\alpha)}(v)$ encodes the discrete \emph{holonomy} of
the routing scheme $(G,\alpha)$. For a fixed $v \in V$, we may associate a
reference frame to the starting point $T(v)$. Reflecting the reference frame
along a cycle in $G$ yields the holonomy of the trajectory.

In the stated generality, $\hom_{\Grp(G,\alpha)}(v)$ generalizes the
\emph{groups of projectivities} studied in~\cite{joswig}. Let $\Delta
\subseteq 2^V$ be a \Def{simplicial complex} on a finite ground set $V$. That
is, $\Delta \neq \emptyset$ and for every $\tau \subseteq \sigma \in \Delta$,
we have $\tau \in \Delta$. The \Def{dimension} of $\Delta$ is
$(d-1)$-dimensional if $d = \max( |\sigma| : \sigma \in \Delta)$ and $\Delta$
is \Def{pure} if every inclusion-maximal $\sigma \in \Delta$ is of cardinality
$d$. Inclusion-maximal sets are called \Def{facets} and two facets $\sigma,
\sigma' \in \Delta$ are \Def{adjacent} if $|\sigma \cap \sigma'| = d-1$.  The
dual graph $G(\Delta)$ of a pure complex $\Delta$ has nodes given by the
facets of $\Delta$ and $\sigma\sigma' \in E(\Delta)$ if $\sigma,\sigma'$ are
adjacent. Finally $\Delta$ is called \Def{strongly connected} if $G(\Delta)$
is connected. If $e = (\sigma, \sigma')\in E(\Delta)$, then $\sigma \setminus
\sigma' = \{v\}$ and $\sigma'\setminus \sigma = \{v'\}$.
Following~\cite{joswig}, this defines a \Def{perspective} $s_e : \sigma \to \sigma'$
with $s_e(v) \defeq v'$ and the identity on $\sigma \cap \sigma'$. In particular,
if $\Walk = \sigma_0 \sigma_1 \dots \sigma_k$ is a walk in $G(\Delta)$, then
this yields a bijection $t_\Walk : \sigma_0 \to \sigma_k$. Thus, for fixed
$\sigma_0$, this defines the \Def{group of projectivities}
\[
    \Pi(\Delta,\sigma_0) \ \defeq \  \{ t_\Walk : \Walk \text{ closed walk based
    at $\sigma_0$}\} \ \subseteq \mathrm{Bij}(\sigma_0) \, .
\]

For a strongly connected simplicial complex $\Delta \subseteq 2^V$ let $\R^V$
be the vector space with basis $e_v$ for $v \in V$. Let $G = G(\Delta)$ and
define $\alpha : E(G) \to \PP(\R^V)$ by $\alpha( \sigma\sigma') \defeq \R(e_v -
e_{v'})$. 

\begin{prop}
    Let $\Delta$ be a strongly connected simplicial complex and $\sigma_0 \in
    \Delta$ a facet. Let $(G,\alpha)$ be the routing scheme defined above.
    Then $\hom_{\Grp(G,\alpha)}(\sigma_0) \subset \mathrm{GL}(\R^V)$
    is the permutation representation of $\Pi(\Delta,\sigma_0)$.
\end{prop}

It would be interesting to further explore the connections to discrete holonomy
groups. Arrangements of linear hyperplanes induce fans that are polytopal. The
associated reflection groupoids show even stronger similarities to groups of
projectivities. In the sequel to this paper, we investigate arrangements whose
fans are inscribable and their connection to reflection groups.

\bibliographystyle{siam} \bibliography{bibliography}

\end{document}